\documentclass[12pt,a4paper,reqno]{amsart}

\usepackage[applemac]{inputenc}
\usepackage{amsmath}
\usepackage{amsthm}
\usepackage{amscd}
\usepackage{amsfonts}
\usepackage{amssymb}
\usepackage{mathrsfs}
\usepackage{geometry}
\usepackage{latexsym}
\usepackage{enumerate}
\usepackage{xcolor}
\usepackage{eucal}
\newgeometry{top=2.2cm,bottom=1.78cm,left=2cm,right=2cm}
\usepackage{tikz}
\usepackage{graphicx}
\usepackage{multicol}
\newtheorem{thm}{Theorem}[section]
\newtheorem{cor}[thm]{Corollary}
\newtheorem{lem}[thm]{Lemma}
\newtheorem{prop}[thm]{Proposition}

\numberwithin{equation}{thm}

\theoremstyle{definition}
\newtheorem{defn}[thm]{Definition}
\newtheorem{rem}[thm]{Remark}
\newtheorem{rems}[thm]{Remarks}
\newtheorem{conv}[thm]{Convention}
\newtheorem{notn}[thm]{Notation}
\newtheorem{exam}[thm]{Example}

\newcommand{\D}{\mathcal{D}}

\newcommand{\US}{\boldsymbol{\mathfrak{A}}}

\newcommand{\fcY}{\mathcal{Y}}

\def\f#1{\mathfrak{#1}}
\newcommand{\fS}{\f{S}}
\newcommand{\C}{\mathcal{C}}
\def\MN(#1){ M_{#1}(\mathbb{N})}
\def\MNR(#1,#2){ M_{#1}(\mathbb{N})_{#2}}
\def\MNS(#1){ M_{#1}(\mathbb{N})^{\pm}}

\newcommand{\ZZ}{\mathbb{Z}}

\newcommand{\ZG}{{{\mathbb{Z}}}_2}
\newcommand{\Nt}{{{\mathbb{N}}}_2}
\newcommand{\NN}{\mathbb{N}}

\def\MZ(#1){ M_{#1}(\ZG)}
\def\MNt(#1){ M_{#1}(\Nt)}
\def\NZ(#1){ (\NN|\ZG)^{#1}}
\def\NZST(#1,#2,#3){ (\NN|\ZG)^{#1}_{#2|#3}}
\def\NZS(#1,#2){ {\NN}^{#1}_{#2}}
\def\MNZ(#1,#2){ M_{#1}(\NN | \Nt)_{#2}}
\def\MNZN(#1){ M_{#1}(\NN | \Nt)}
\def\CMNZ(#1,#2,#3){\Lambda(#1,#2|#3)}
\def\CMN(#1,#2){\Lambda(#1,#2)}
\def\CMNP(#1,#2){\Lambda_{#1,#2}}
\def\MNZNS(#1){ \MNZN(#1)^{\pm}}
\def\SE#1{{#1}^{\bar{0}}}
\def\SO#1{{#1}^{\bar{1}}}
\def\SEE#1{{#1}^{\bar{0}}}
\def\SOE#1{{#1}^{\bar{1}}}
\def\SUP#1{\SE{#1}|\SO{#1}}
\def\SS(#1,#2){{#1}^{\ol{#2}}}
\def\SSE(#1,#2){{#1}^{\ol{#2}}}

\def\ESE(#1, #2, #3){ \SEE{#1}_{{#2},{#3}} }
\def\ESO(#1, #2, #3){ \SOE{#1}_{{#2},{#3}} }
\def\bs#1{\boldsymbol{#1}}

\def\Qqs(#1,#2){ \mathcal{Q}(#1,#2) }

\def\lcase#1{\MakeLowercase{#1}}

\newcommand{\SerA}{\f{H}^c_r}
\newcommand{\HCR}{\mathcal{H}^c_{r, R}}
\newcommand{\Heck}{\mathcal{H}_{r,R}}

\newcommand{\QqnrR}{\mathcal{Q}_{\lcase{q}}(\lcase{n},\lcase{r}; R)}

\def\ol#1{\overline{#1}}

\newcommand{\ep}{\epsilon}

\def\Qvs(#1){\mathcal{Q}_{\lcase{v}}{(\lcase{#1})}}

\def\Uvq(#1){U_{\lcase{v}}(\mathfrak{\lcase{q}}_{\lcase{#1}})}

\def\USN(#1){{\US[#1]}_{v}}

\def\SABJR(#1,#2,#3,#4){({#1}|{#2})[\bs{#3}, #4]}
\def\SABJRS(#1,#2,#3,#4){({#1}|{#2})[#3, #4]}
\def\SABJS(#1,#2,#3){({#1}|{#2})[#3]}
\def\SAJRS(#1,#2,#3){{#1}[#2, #3]}
\def\SAJS(#1,#2){{#1}[#2]}
\def\SABJ(#1,#2,#3){({#1}|{#2})[\bs{#3}]}
\def\SAJR(#1,#2,#3){{#1}[\bs{#2}, #3]}
\def\SAJ(#1,#2){{#1}[\bs{#2}]}
\def\ABJR(#1,#2,#3,#4){({#1}|{#2})(\bs{#3}, #4)}
\def\ABJRS(#1,#2,#3,#4){({#1}|{#2})(#3, #4)}
\def\ABJS(#1,#2,#3){({#1}|{#2})(#3)}
\def\AJRS(#1,#2,#3){{#1}(#2, #3)}
\def\AJS(#1,#2){{#1}(#2)}
\def\ABJ(#1,#2,#3){({#1}|{#2})(\bs{#3})}
\def\AJR(#1,#2,#3){{#1}(\bs{#2}, #3)}
\def\AJ(#1,#2){{#1}(\bs{#2})}

\def\STDUE(#1,#2){({#1}+E_{{#2},{#2}+1}-E_{{#2}+1,{#2}+1}|0)}
\def\STDUO(#1,#2){({#1}-E_{{#2}+1,{#2}+1}|E_{{#2},{#2}+1})}
\def\STDLE(#1,#2){({#1}-E_{{#2},{#2}}+E_{{#2}+1,{#2}}|0)}
\def\STDLO(#1,#2){({#1}-E_{{#2},{#2}}|E_{{#2}+1,{#2}})}
\def\STDDE(#1){(D_{#1}|0)}
\def\STDDO(#1,#2){({#1}-E_{{#2},{#2}}|E_{{#2},{#2}})}

\newcommand{\End}{\mathrm{End}}

\newcommand{\ro}{\mathrm{ro}}
\newcommand{\co}{\mathrm{co}}
\newcommand{\AP}{{A^+_{h,k}}}
\newcommand{\TDAP}{T_{d_{A^+_{h,k}}}}
\newcommand{\AM}{{A^-_{h,k}}}
\newcommand{\TDAM}{T_{d_{A^-_{h,k}}}}

\def\STEPX#1#2{{[\![{#1}]\!]}_{#2}}

\def\STEP#1{ {[\![{#1}]\!]}_{{q}} }
\def\STEPP#1{{[\![{#1}]\!]}_{{q}^2}}
\def\STEPPD#1{{[\![{#1}]\!]}_{{q},{q}^2}}
\def\STEPPDR#1{{[\![{#1}]\!]}_{{q}^2,{q}}}

\newcommand{\where}{\ \bs{|} \ }

\def\Hom{\mathrm{Hom}}

\newcommand{\spaceintv}{}
\def\intd(#1,#2,#3){\left[\begin{matrix}{#1};{#2}\\{#3}\end{matrix}\right]}
\def\intds(#1,#2){\left[\begin{matrix}{#1}\\{#2}\end{matrix}\right]}
\def\intdss(#1,#2){\intd({#1},{0},{#2})}
\def\diag{{\rm{diag}}}

\def\parity#1{p({#1})}

\def\TAIJ(#1,#2){ T^\lhd_{({#1}, {#2})} }
\def\TDIJ(#1,#2){ T^\rhd_{({#1}, {#2})} }
\def\rmText#1{}
\def\rmForm#1{}

\def\AK(#1,#2){ {\overleftarrow{\bf r}}^{#2}_{#1} }
\def\BK(#1,#2){ {\overrightarrow{\bf r}}^{#2}_{#1} }
\newcommand{\YK}{{\widetilde a}^{\bar1}_{h-1,k}}

\newcommand{\cbefore}{{\overleftarrow c}_{\!\!\Ad\;\;}^{h,k}}
\newcommand{\cafter}{{\overrightarrow c}_{\!\!\Ad}^{h+1,k}}
\newcommand{\ckbefore}{{\overleftarrow c}_{\!\!\Ad\;\;}^{h,k}}
\newcommand{\ckafter}{{\overrightarrow c}_{\!\!\Ad}^{h,k}}
\newcommand{\cmiddle}{c_{\Ad;k}^{h,h+1}}

\def\up{{\upsilon}}

\allowdisplaybreaks[1]

\def\bsS{{\boldsymbol{\mathcal S}}}
\def\sS{{\mathcal S}}

\def\rmO{{\mathrm O}}
\def\Ad{{A^{\!\star}}}
\def\Bd{{B^{\!\star}}}
\def\Cd{{C^{\!\star}}}
\def\Xd{{X^{\!\star}}}
\def\Md{{M^{\!\star}}}
\def\Nd{{N^{\!\star}}}
\def\Dd{D^{\star}}

\def\Ed{E^{\star}}
\def\Fd{F^{\star}}
\def\qSchR{{\mathcal S_q(n,r;R)}}
\def\urcm{A^{\text{\normalsize $\llcorner$}}}
\def\ulcm{A^{\text{\normalsize $\lrcorner$}}}
\def\llcm{A_{\text{\normalsize$\urcorner$}}}
\def\lrcm{A_{\text{\normalsize$\ulcorner$}}}
\def\sdpHk{{}_\text{\sc sdp}{\rm H}}
\def\sdpHk{{}_\text{\sc sdp}{\rm H}\overline{\textsc{k}}}
\def\sdpHe{{}_\text{\sc sdp}{\rm H}\overline{\textsc{e}}}
\def\sdpHf{{}_\text{\sc sdp}{\rm H}\overline{\textsc{f}}}
\def\Tk{{\rm T}\overline{\textsc{k}}}
\def\Te{{\rm T}\overline{\textsc{e}}}

\def\hahk{\widehat a_{h,k-1}}
\def\ha{\widehat a}
\def\fkm{\mathfrak m}
\def\ufm{\underline{\mathfrak m}}
\def\fkM{\mathfrak M}
\def\la{{\lambda}}

\title[Some multiplication formulas in queer $q$-Schur superalgebras]
{Some multiplication formulas in queer $q$-Schur superalgebras}

\author{Jie Du$^\star$, Haixia Gu$^{\star\star}$, Zhenhua Li$^*$, and Jinkui Wan$^{**}$}
\address{J. D., School of Mathematics and Statistics,
University of New South Wales, Sydney NSW 2052, Australia}
\email{j.du@unsw.edu.au}
\address{H. G., School of Science, Huzhou University, Huzhou, China}
\email{ghx@zjhu.edu.cn}
\address{Z. L., School of Mathematical Sciences, Xiamen University, Xiamen 361005, China}
\email{zhen-hua.li@qq.com}
\address{J. W., School of Mathematics and Statistics, Beijing Institute of Technology,
Beijing 100081, China. } \email{wjk302@hotmail.com}

\keywords{quantum supergroup, $q$-Schur superalgebra, Hecke-Clifford superalgebra, Schur--Olshanski duality}

\date{\today}

\subjclass[2010]{20B30, 20G43, 17A70, 20C08}
\thanks{$^\star$Partially supported by the Science FRG;\;\; $^{\star\star}$Partially supported by NSFC-12071129}
\thanks{*Partially supported by NSFC-11871404;\;\;**Partially supported by NSFC-12122101 \& NSFC-12071026}

\begin{document}
\maketitle

\begin{abstract} Building on the work \cite{DW2},  where some standard basis for the queer $q$-Schur superalgebra $\QqnrR$ is defined by a labelling set of matrices and their associated double coset representatives, we investigate the matrix representation of the regular module of $\QqnrR$ with respect to this basis. More precisely, we derive explicitly (resp., partial explicitly) the multiplication formulas of the basis elements by certain even (resp., odd) generators of a queer $q$-Schur superalgebra. 

These multiplication formulas are highly technical to derive, especially in the odd case. It requires to discover many multiplication (or commutation) formulas in the Hecke--Clifford algebra $\HCR$ associated with  the labelling matrices. 
For example, for a given such a labelling matrix $\Ad$, there are several  matrices $w(A)$, $\sigma(A), \widetilde A$, and $\widehat A$ associated with the base matrix $A$ of $\Ad$, where $w(A)$ is used to compute a reduced expression of the distinguished double coset representatives $d_A$, and the other matrices are used to describe the permutation $d_A$ and the SDP (commutation) condition between $T_{d_A}$ and generators of the Clifford subsuperalgebra.  

With these multiplication formulas, we will construct a new realisation of the quantum queer supergroup in a forthcoming paper \cite{DGLW}, and to give new applications to the integral Schur--Olshanski duality and its associated representation theory at roots of unity.
\end{abstract}

\maketitle

\tableofcontents

\spaceintv
\section{Introduction}\label{sec_introduction}
For a given finitely generated algebra $\mathcal A$ over a field, the regular representation of $_{\mathcal A}\mathcal A$ can be understood as the matrix representations of generators with respect to a certain given basis. If the basis contains the generators as a subset, computing such matrix representations amounts to the determination of the structure constants associated with the generators or to the explicit multiplication formulas of the basis  by the generators. When $\mathcal A$ is a member of the family of Schur and $q$-Schur algebras, which are fundamental objects in representation theory, such multiplication formulas are first established by Beilinson--Lusztig--MacPherson (BLM) in \cite{BLM} (see also \cite[\S13.5]{DDPW} for more details), and used to determine the regular representation (or a new realisation) of quantum linear groups.

The method in \cite{BLM} is geometric, using the convolution algebra interpretation of $q$-Schur algebras associated with certain partial flag varieties, and the multiplication formulas are derived directly from the convolution products. This geometric method has also been generalised to the $q$-Schur algebra of affine type $A$ \cite{GV, L} and, more recently, to $q$-Schur algebras arising from $i$-quantum groups and quantum symmetric pairs \cite{BW, BKLW} of type AIII, whose Satake diagrams have no black nodes (see also \cite{LNX} for a coordinate algebra approach to these $q$-Schur algebras), and to affine type $C$
\cite{FLLLW}. However, for quantum linear supergroups and $q$-Schur superalgebras, the geometric method cannot be applied. Thus, an algebraic method is developed in \cite{DG, DGZ}. In this approach, the multiplication formulas for $q$-Schur superalgebras, regarded as endomorphism algebras of $q$-permutation supermodules over Hecke algebras, are derived using certain multiplication (or commutation) formulas  in the Hecke algebras involving distinguished double coset representatives.

The aim of this paper is to establish similar fundamental multiplication formulas for queer $q$-Schur superalgebras arising from the Schur-Olshanski duality \cite{Ol}. We will use the definition of queer $q$-Schur superalgebras $\QqnrR$  introduced by the first and last authors \cite{DW2}.
As endomorphism algebras of a direct sum of certain cyclic modules (or queer $q$-permutation supermodules) over the Hecke-Clifford superalgebra $\HCR$, such a queer $q$-Schur superalgebra admits a standard basis $\{\phi_\Ad\}$ labelled by certain $n\times 2n$ matrices $\Ad=(A^{\bar0}|A^{\bar1})$. Here each such a matrix defines a distinguished double coset representative $d_A$ associated with the ``base matrix'' $A:=A^{\bar0}+A^{\bar1}$ in the symmetric group $\fS_r$ and an element $c_{\Ad}$ in the Clifford superalgebra, %whose product becomes the generator of a queer $q$-permutation supermodule,
and the standard basis element $\phi_\Ad$ is defined by an element $T_\Ad$ in $\HCR$ associated with $d_A$ and the element $c_{\Ad}$.
Continuing with this construction, we now compute the structure constants occurring in the multiplication $\phi_\Xd\phi_\Ad$, for an arbitrary $\Ad$, %if $\Xd$ is even, and certain $\Ad$ if $\Xd$ is odd,
where $\Xd$ is one of the following six matrices
$$\aligned
(\lambda|O),\;\; (\lambda + E_{h, h+1} -E_{h+1, h+1} |O), \;\;(\lambda-E_{h,h}+E_{h+1,h}|O)\;\;(\text{of parity }0);\\
 (\lambda-E_{h,h}|E_{h,h}),\;\;(\lambda-E_{h+1, h+1}| E_{h, h+1} ),\;\;(\lambda-E_{h,h}|E_{h+1,h})\;\;(\text{of parity }1).\endaligned$$
Here $\lambda\in\mathbb N^n$ represent diagonal matrices. These matrices correspond to the Chevalley type generators for the queer Lie superalgebra and its universal enveloping algebra.

However, unlike the quantum $\mathfrak{gl}_n$ or $\mathfrak{gl}_{m|n}$ case, there are no complete multiplication formulas for $\Xd$ being the odd matrices above. This subtlety was observed over 10 years ago! In order to seek new ideas and methods, we tested two experiments. First, can the universal enveloping algebra of the queer Lie superalgebra be approached by queer Schur superalgebras?  Second, can the quantum queer supergroups be approached by constructing its regular representation directly? Both questions had affirmative answers; see \cite{GLL} and \cite{DLZ}. The two works convinced us that a construction using queer $q$-Schur superalgebras must exist and so does the integral Schur--Olshanski duality.

Here is the new approach to tackle the problem. Instead of seeking complete multiplication formulas in the odd case, we focus on the ``head part'', which is determined by a so-called SDP condition, leaving out the big ``tail part''; see Theorems \ref{phidiag1}--\ref{philower1}. The SDP condition is about commuting certain generators in the Clifford superalgebra with a basis element $T_{d_A}$  in the Hecke algebra associated with the distinguished double coset representative $d_A$ defined by a matrix $A$; see Definition \ref{defn:SDP}. If a matrix satisfies this condition, then certain multiplication formulas in the odd case can be completely computed. Fortunately, these head part formulas do provide a path to solving this decade long problem.

In the forthcoming paper \cite{DGLW}, we will use the multiplication formulas (and partial formulas in the odd case) in this paper to give a new construction (a BLM type realisation) for the quantum queer supergroup ${\bf U}(\mathfrak q_n)$. Furthermore, this breakthrough will allow us to establish the integral Schur--Olshanski duality and to study the polynomial representation theory at roots of unity \cite{DWZ}. We also remark that,
as a by product,  all above given $\phi_\Xd$ generate the queer $q$-Schur algebra $\QqnrR$ and the multiplication formulas in \S\S5--6 provide the regular representation on the standard basis. See Remark \ref{regular}.

We organise the paper as follows. We start with the definitions of Hecke-Clifford superalgebras $\HCR$ and queer $q$-Schur superalgebras $\QqnrR$ in \S2, and discuss the construction of their standard bases.
In \S3, we establish several commutation formulas in $\HCR$ which involve distinguished double coset representatives and the $c$-elements aforementioned. The SDP condition is introduced and investigated in \S4. The main result of this section is Theorem \ref{CdA1} which characterises the SDP condition by the shape of the matrix $A$.
The fundamental multiplication formulas for the even (resp., odd) case  are established in \S5 (resp. in \S6).
In \S7, we present some special cases for the multiplication formulas and relations relative to the generators. Finally, in the last section, we outline some ongoing work on applications of these multiplication formulas, including the aforementioned new realisation of the quantum queer supergroup ${\bf U}(\mathfrak q_n)$.

{\bf Some Notations.}  Let $\ZZ$ (resp., $\NN$) be the set of integers (resp. non-negative integers). Let $\ZZ_2=\{\bar 0,\bar1\}$ denote the abelian group of order 2 and $\NN_2=\{0,1\}\subset\NN$.

For any positive integer $n$, let $[1,n]=\{1,2,\ldots,n\}$, and define $[i+1, i+n]$ similarly for any integer $i$.
Let $[\![n]\!]=[\![n]\!]_q=\frac{q^n-1}{q-1}$ denote the Gaussian integers and let $[\![n]\!]_{x,y}=[\![n]\!]_x-[\![n]\!]_y$.
The free $\ZZ$-module $\ZZ^n$ has standard basis $\bs{\ep}_{1},\bs{\ep}_{2},\ldots,\bs{\ep}_{n}$ and ``positive roots'' $\alpha_1,\ldots,\alpha_{n-1}$, where $\alpha_h=\bs{\ep}_{h}-\bs{\ep}_{h+1}$.

Let $\MN(n)$ be the set of $n\times n$-matrices over $\mathbb N$. For any $m\times n$ matrix $A$ over $\NN$, let $|A|$ denote the sum of entries of $A$.

\spaceintv
\section{The queer $q$-Schur superalgebra}\label{sec_qqschur}% and quantum queer supergroups}

Let $R$ be a commutative ring with 1.
The Clifford superalgebra {$\C_r$} is an associative superalgebra over {$R$} defined by odd generators {$c_1, \cdots , c_r$} and relations: for $1\leq i,j\leq r$,
\begin{equation}\label{cliff}
	c_i^2 = -1,  \quad c_i c_j = - c_j c_i \; (i \ne j).
\end{equation}

Let {$\fS_{r}$} be the symmetric group on {$r$} letters,
with Coxeter generators (or simple reflections) {$ s_{i}=(i, i+1)$}, for all {$1 \le i < r$} and length function $\ell$.
Let $y<w$ denote the Bruhat partial order relation on $\fS_r$.
The Hecke  algebra  {$\Heck={\mathcal H}(\fS_r)_R$} associated with {$\fS_{r}$} and $1\neq q\in R^\times$
is an algebra  over {${R}$}  generated by
{$T_{i} = T_{s_{i}}$} for all  {$1 \le i < r$},
subject to the relations: for $1 \le i, j \le r-1$,
\begin{equation}\label{Hecke}
	\aligned &(T_i - {q} )(T_i + 1) = 0, \quad \\
	&T_i T_j = T_j T_i\,
	 (|i-j| >1),\quad\\
	&T_i T_{i+1} T_i = T_{i+1} T_i T_{i+1}\; (i\neq r-1).\endaligned
\end{equation}

The {\it Hecke-Clifford superalgebra}  {$\HCR$}  is a  superalgebra  over {${R}$}  generated by even generators
{$T_1, \cdots, T_{r-1}$} and odd generators {$c_1, \cdots, c_r$},
subject to the relations \eqref{cliff}-\eqref{Hecke} and the extra relations: for $1\leq i\leq r-1, 1\leq j\leq r,$
\begin{equation}\label{Hecke-Cliff}
\aligned c_j T_i&=T_i c_j\,\quad(j\neq i, i+1), \quad\\
 c_{i+1} T_i&=T_i c_i, \quad\\
c_i T_i &=T_i c_{i+1} + ({q}-1)(c_i - c_{i+1}).\endaligned
\end{equation}
For any homogeneous element {$h \in \HCR$}, let {$\parity{h}$} be its parity, more precisely,  {$\parity{h} = \bar{0}$} if  {$h$} is even, or {$\parity{h} = \bar{1}$} if {$h$} is odd.

The following relations can be deduced directly from \eqref{Hecke}-\eqref{Hecke-Cliff} and will be used in the sequel: for any {$1\leq k\leq r-1$},
\begin{equation}\label{Tick-inv}
\begin{aligned}
(1)\;\;&{q} T_{k}^{-1} = T_{k} -  ({q} - 1),
\qquad\\
%&(2)\;\;   x_{\alpha} T_{k}^{-1} =  {q}^{-1} x_{\alpha} \quad  \mbox{ if }  s_{k} \in \fS_{\alpha},\\
(2)\;\;&T_{k} c_{k+1}
= {q} c_{k} T_{k}^{-1} + ({q} - 1) c_{k+1},
\qquad\\
(3)\;\;& c_{k}T_{k}
= {q} T_{k}^{-1} c_{k+1}  + ({q} - 1) c_{k}.
\end{aligned}
\end{equation}

For $w \in \fS_r$ with reduced expression $w=s_{i_1}s_{i_2}\cdots s_{i_k}$ and $\alpha\in {\NN}^r_2$, let
\begin{align*}
T_w = T_{i_1} \cdots T_{i_k},\qquad c^\alpha=c_1^{\alpha_1}c_2^{\alpha_2}\cdots c_r^{\alpha_r}.
\end{align*}
Then it is known (cf. \cite[Lemma 2.2]{DW1}) that both
%Here $w= s_{i_1} \cdots s_{i_k}$  is any reduced expression of $w$.
%Then, by \cite[Lemma 2.2]{DW1}, both
$$\mathcal B=\{c^\alpha T_w\mid w\in\fS_r,\alpha\in {\NN}^r_2\}\;\text{ and }\;
\mathcal B'=\{T_wc^\alpha\mid w\in\fS_r,\alpha\in {\NN}^r_2\}$$ form (standard) $R$-bases for $\HCR$, consisting of homogeneous elements.

\begin{lem} The transition matrix from $\mathcal B$ to $\mathcal B'$ has the lower triangular form
\begin{equation}
\left(\begin{matrix}
P_{1}&O&\cdots&O\\%&\cdots&O\\
{*}&P_2&\cdots&O\\%&\cdots&O\\
\vdots&\vdots&\ddots&\vdots\\%&%&O\\
{*}&{*}&\cdots&P_{r!}\\%&\cdots&O\\
%\vdots&\vdots&&\vdots&&O\\
%O&O&\cdots&O&\cdots&A_s
\end{matrix}\right),
\end{equation}
where each block $P_k$ ($k\in[1,r!]$) on the diagonal is a signed permutation matrix of size $2^r\times 2^r$. Moreover, each $P_k$ is a block diagonal of blocks
of sizes ${r\choose s}$, for all $s\in[1,r]$.
\end{lem}
\begin{proof} Define the place permutation of $\fS_r$ on ${\NN}^r_2$ by setting
$(\alpha_1,\ldots, \alpha_r).w=(\alpha_{w(1)},\ldots,\alpha_{w(r)})$. We have, by \cite[(2.4)]{DW2},
$$c^\alpha T_w=\pm T_w c^{\alpha.w}+\sum_{y<w,\beta\in\NN_2^r}f^{\alpha,w}_{\beta,y}T_yc^\beta$$
 for some $f^{\alpha,w}_{\beta,y}\in R$.
  For $1\leq s\leq r$, let $\NN_2^{r,s}=\{\alpha\in \NN_2^r\mid s=|\alpha|\}$. If $\alpha\in\NN_2^{r,s}$, the formula above can be further refined to
 \begin{equation*}\label{homog}
 c^\alpha T_w=\pm T_w c^{\alpha.w}+\sum_{y<w,\beta\in\NN_2^{r,s}}f^{\alpha,w}_{\beta,y}T_yc^\beta
 \end{equation*}
 (For a proof, use induction on $|\alpha|$.)
 Thus, if we order the basis $\mathcal B$ by linearly  refining the partial order defined by setting
 $$T_wc^\alpha<T_yc^\beta\iff w>y \text{ or }y=w, |\alpha|\leq|\beta|\;\; (\text{and in lexicographic order if }|\alpha|=|\beta|)$$
(and define $c^\alpha T_w<c^\beta T_y$ similarly), then the transition matrix has the desired form.
\end{proof}

For later use, we record the following special case:
 \begin{equation}\label{homog}
 c_j T_w= T_w c_{w^{-1}(j)}+\sum_{y<w,i\in[1,r]}f^{j,w}_{i,y}T_yc_i.
 \end{equation}

Define the set of  compositions of $r$ into $n$ parts:
\begin{equation}\label{Lanr}
\CMN(n,r)=\{\lambda=(\lambda_1,\lambda_2,\ldots,\lambda_n)\in \NN^n \where  r=|\lambda|:=\lambda_1+\cdots+\lambda_n \}.
\end{equation}
For {$ \lambda,\mu \in \CMN(n,r)$}, let {$\fS_{\lambda}$} be the standard Young subgroup of {$\fS_r$} corresponding to {$\lambda$}, and
let {$\D_{\lambda}$} be the distinguished (i.e. shortest) representatives of right cosets of {$\fS_{\lambda}$} in {$\fS_r$},
Then {$\D_{\lambda, \mu}:=\D_\lambda\cap\D_\mu^{-1}$} is the distinguished representatives of the {$\fS_{\lambda}$}-{$\fS_{\mu}$} double cosets in {$\fS_r$}.

Associated with {$\fS_{\lambda}$}, define the even element
\begin{align}\label{xlambda}
	x_{\lambda} = \Sigma_{\fS_{\lambda}}:= \sum_{w \in \fS_{\lambda}} T_w.
\end{align}
It is known that, for {$  s_i \in \fS_{\lambda} $}, we have
\begin{align}\label{Tixlambda}
	T_i x_{\lambda} =x_{\lambda}  T_i = {q} x_{\lambda},
\end{align}
and $x_\lambda\HCR$ is an ${R}$-free supermodule\footnote{By an $R$-free supermodule, we mean an $R$-module of the form $V_{\bar0}\oplus V_{\bar1}$ such that both $V_{\bar0}$ and $V_{\bar1}$ are $R$-free.}, as well as a right $\HCR$-supermodule.

\medskip

Given ${R}$-free supermodules  $V,W$,
let   $\Hom_{{R}}(V,W)$  be the set of  all ${R}$-linear  maps from $V$ to $W$.
We make $\Hom_{{R}}(V,W)$  into an $R$-free supermodule by declaring
${\Hom_{{R}}(V,W)}_{i}$
to be  the set of  homogeneous maps of parity $i\in \ZG$,
that is,
the linear maps $\theta:V\rightarrow W$ with $\theta(V_{j})\subset W_{{i}+{j}}$ for {$j\in\ZG$}.
{\it Note that it is typical in a superalgebra to write the expressions
which only make sense for homogeneous elements,
and the expected meaning for arbitrary elements
 is obtained by extending linearly from the homogeneous case.}

Following \cite{DW2},  the {\it queer $q$-Schur superalgebra} over {$R$} is defined to be

\begin{equation}\label{QqnrR}
	\QqnrR := \End_{\HCR}\Big(\bigoplus_{\lambda \in \CMN(n,r)} x_{\lambda} \HCR\Big).
	%\Qqnr=\mathcal Q_{{v}^2}(n,r;\mathcal Z).
\end{equation}
Here we view $T_R(n,r):=\bigoplus_{\lambda \in \CMN(n,r)} x_{\lambda} \HCR$ as a right $\HCR$-module and homomorphisms are $\HCR$-module homomorphisms, i.e., $R$-linear maps $\phi:T_R(n,r)\to T_R(n,r)$ satisfying $\phi(mh)=\phi(m)h$, for all $m\in T_R(n,r)$ and $h\in\HCR$.

\begin{rem}\label{sHom}Notice that  $\End_{R}(T_R(n,r))$ is an $R$-free supermodule. Thus,
we may also define the endomorphism superalgebra $ \End_{\HCR}^s(T_R(n,r))$
which consists of $\HCR$-supermodule homomorphisms defined by the rule
\begin{equation}\label{eq:twist}
 \Phi(mh)=(-1)^{\parity{\Phi}\parity{h}}\Phi(m)h,
\end{equation}
for all $m\in T_R(n,r)$ and (homogeneous)  $h\in\HCR,\Phi\in\End_{R}(T_R(n,r))$
(see  \cite[Def.~1.1]{CW}).
These superalgebras  over $R=\mathbb Q(\up)$ $(q=\up^2$) will be used in \cite{DGLW} for a new construction of the quantum queer supergroup. It can be proved that  if $\sqrt{-1}\in R$ then
$$\End_{\HCR}^s(T_R(n,r))\cong \End_{\HCR}(T_R(n,r))=\QqnrR.$$
Let $\widetilde{\boldsymbol{\mathcal Q}}(n,r):= \End_{\mathcal H^c_{r,\mathbb Q(\up)}}^s(T_{\mathbb Q(\up)}(n,r))$. Note that
$T_{\mathbb C(\up)}(n,r)$ can be identified as the tensor space considered in \cite{Ol}; see \cite[Lem.~3.1]{WW}.
\end{rem}

In order to describe a natural (standard) basis for $\QqnrR$, we need a few more notations.

Associated with {${\lambda} = ({\lambda}_1, \cdots , {\lambda}_{n}) \in \CMN(n,r)$}, define its {\it partial sum sequence},
%{$\widetilde{\lambda}_i = \sum_{k=1}^i \lambda_{k}$}
\begin{equation}\label{latilde}
\widetilde\lambda:=(\widetilde{\lambda}_1,\ldots,\widetilde{\lambda}_n),\text{ where }\widetilde{\lambda}_i = \sum_{k=1}^i \lambda_{k} \;\;(1\leq i\leq n), \text{ and }\widetilde{\lambda}_0=0.
\end{equation}

Let $\MN(n)$ be the set of $n\times n$-matrices over $\mathbb N$.
For any {$A=(a_{i,j}) \in \MN(n)$}, let
\begin{align}\label{nuA}
	\nu_{A} := (a_{1,1}, \cdots, a_{n,1}, a_{1,2}, \cdots, a_{n,2}, \cdots,  a_{1,n}, \cdots, a_{n,n})=(\nu_1,\nu_2,\ldots,\nu_{n^2}).
\end{align}
Define $\widetilde {\nu_A}=\nu_{\widetilde A}$ similarly, where %for some $1\leq i,j\leq n $, then we write the partial sum $\widetilde{\nu}_k$ as
\begin{align}\label{Atilde}
%\widetilde{\nu}_k=
\widetilde A=(\widetilde a_{i,j}),\text{ with }\;\widetilde{a}_{i,j} := \sum_{p=1}^{j-1}\sum_{u=1}^{n} a_{u,p} + \sum_{u=1}^{i} a_{u,j},
\text{ and set }  \widetilde a_{0,1}=0,\;
\widetilde{a}_{0,j} =\widetilde{a}_{n,j-1}\; (j\geq2).
\end{align}

Following \cite[Section 4]{DW2}, we introduce the following odd elements in $\HCR$: for $1\leq i\leq j\leq r$,
\begin{equation}\label{cqij}
	c_{{q},i,j} = {q}^{j-i}c_i  + \cdots + {q} c_{j-1} + c_j,\qquad
	c'_{{q},i,j} = c_i + {q} c_{i+1} + \cdots + {q}^{j-i}c_j.
\end{equation}

For {$\lambda= (\lambda_1, \cdots , \lambda_m)\in \mathbb N^m$} and {$\alpha \in \Nt^m $} such that $\alpha\leq\lambda$ (meaning $\alpha_i\leq\lambda_i $, for all $1\leq i\leq m$), recall the following elements introduced in \cite[\S4]{DW2}
\begin{equation}\label{eq_c_a}
\begin{aligned}
	&c^{\alpha}_{\lambda} = {(c_{{q}, 1, \widetilde{\lambda}_1})}^{\alpha_1}
				{(c_{ {q}, \widetilde{\lambda}_1+1 , \widetilde{\lambda}_2} )}^{\alpha_2}
				\cdots
				{(c_{{q}, \widetilde{\lambda}_{m-1}+1,  \widetilde{\lambda}_m} )}^{\alpha_m}, \\
	&(c^{\alpha}_{\lambda})' = {(c'_{{q}, 1, \widetilde{\lambda}_1})}^{\alpha_1}
				{(c'_{{q}, \widetilde{\lambda}_1 + 1,  \widetilde{\lambda}_2} )}^{\alpha_2}
				\cdots
				{(c'_{{q}, \widetilde{\lambda}_{m-1} + 1,  \widetilde{\lambda}_m })}^{\alpha_m} .
\end{aligned}
\end{equation}
Here we used the convention $(c_{q, \widetilde{\lambda}_{i-1}+1 , \widetilde{\lambda}_i} )^{\alpha_i}=1$ if $\lambda_i=0$ (so, $\alpha_i=0$).
By \cite[Lem.~4.1, Cor.~4.3]{DW2},
we have, for all $\alpha\leq\lambda$,
\begin{align}\label{xc}
	x_{\lambda} c_{\lambda}^{\alpha} = (c_{\lambda}^{\alpha})'  x_{\lambda}.
\end{align}

Let {$  \MNR(n,r) = \{ A=(a_{i,j}) \in \MN(n) \where  r=|A| \} $}, where {$|A|:=\sum_{i,j} a_{i,j} $}.
It is well known (see, e.g., \cite[Section 4.2]{DDPW}) that there is a bijection
\begin{equation}\label{jmath}
\aligned
	\mathfrak{j}: \MNR(n,r) &\to \{ (\lambda,  d, \mu) \where \lambda,\mu \in \CMN(n,r), d \in \D_{\lambda, \mu} \} \\
	A & \mapsto (\ro(A), d_A, \co(A)),
	\endaligned
\end{equation}
where, for $A=(a_{i,j})$,
\begin{equation}\label{roco}
\ro(A)=\Big(\sum_{j=1}^na_{1,j},\ldots,\sum_{j=1}^na_{n,j}\Big)\text{ and }\co(A)=\Big(\sum_{i=1}^na_{i,1},\ldots,\sum_{i=1}^na_{i,n}\Big).
\end{equation}
Call $\ro(A)$ (resp., $\co(A)$) the {\it row sum} (resp., {\it column sum}) vector of $A$. For the distinguished double coset representative $d_A$, see \eqref{d_A} below for its reduced expression and \eqref{map d_A} for the precise description as a permutation.

We are now ready to describe a basis for $\QqnrR$.

Recall {$\Nt = \{0, 1\} \subset \NN$} and  the following matrix sets introduced in \cite{DW2}:\footnote{The set $\MNZN(n)$ is denoted $M_n(\NN|\ZZ_2)$ in \cite{DW2}.}
\begin{equation}
\aligned\label{MNZ}
	& \MNZN(n) = \{ \Ad=(\SUP{A}) \where \SE{A} \in \MN(n), \SO{A} \in \MNt(n) \}, \\
	& \MNZ(n,r) = \{ \Ad=(\SUP{A}) \in \MNZN(n) \where  \SE{A} + \SO{A} \in \MNR(n,r) \}.
\endaligned
\end{equation}
If $\Ad =(\SE{A}|\SO{A})\in\MNZ(n,r)$ with $\SE{A}= (\SEE{a}_{i,j}),\SO{A}=(\SOE{a}_{i,j})$, then we often write
$\Ad=(\SEE{a}_{i,j}|\SOE{a}_{i,j})$. We call the matrix
\begin{equation}\label{Ad->A}
A=\lfloor\Ad\rfloor:=\SE{A}+\SO{A}=(a_{i,j})\in\MNR(n,r)\;, \text{ where }a_{i,j}=\SEE{a}_{i,j}+\SOE{a}_{i,j},
\end{equation}
the {\it base} (matrix) of $\Ad$. For later use in Section 6, we also let, for $\lambda,\mu\in\Lambda(n,r)$,
$$\MNZN(n)_{\lambda,\mu}=\{\Ad\in\MNZN(n)\mid \ro(A)=\lambda,\co(A)=\mu\}.$$

With the compositions $\nu=\nu _A$ %$=(a_{1,1},\ldots,a_{n,1},\ldots,a_{1,n},\ldots,a_{n,n})$ and
and $\alpha=\nu_{A^{\bar1}}$ %:=(\SOE{a}_{1,1},\ldots, \SOE{a}_{n,1},\ldots, \SOE{a}_{1,n},\ldots, \SOE{a}_{n,n})$
as defined
in \eqref{nuA}, we define
\begin{equation}\label{eq cA}
c_{\Ad} :=  c_{\nu}^{\alpha}\;\;\text{ and }\;\;T_{\Ad} :=  x_{\ro( A)} T_{d_{ A}} c_{\Ad}\sum _{{\sigma} \in \D_{{\nu}} \cap \fS_{\co( A)}} T_{\sigma}.
\end{equation}
For the simplicity of later use, we abbreviate the ``tail'' of the element $T_\Ad$ by setting
\begin{equation}\label{tail}
\Sigma_A:=\Sigma_{\D_{{\nu_A}} \cap \fS_{\co( A)}}=\sum _{{\sigma} \in \D_{{\nu_A}} \cap \fS_{\co( A)}} T_{\sigma}.
\end{equation}

We also set
\begin{equation}\label{p(A)}
c'_{\Ad} =  (c_{\nu}^{\alpha})',\qquad
p(\Ad):=p(T_\Ad)\equiv|\SO{A}|\text{ (mod 2)}.
\end{equation}

\begin{prop}[{\cite[Prop.~5.2, Th.~5.3]{DW2}\label{DW-basis}}]
(1) For any $\lambda,\mu\in\Lambda(n,r)$,
the intersection $x_\lambda\HCR\cap \HCR x_\mu$ is a free $R$-module with basis $\{T_\Ad\mid \Ad  \in\MNZ(n,r), \lambda=\ro( A), \mu=\co( A)\}$.

(2) The queer $q$-Schur superalgebra $\QqnrR$ is a free $R$-supermodule with basis $\{ \phi_{\Ad} \where  \Ad \in \MNZ(n,r)\}$, where {$\phi_{\Ad} $} is defined by
\begin{equation*}
\phi_{\Ad}: \bigoplus_{\lambda \in \CMN(n,r)} x_{\lambda} \HCR \longrightarrow \bigoplus_{\lambda \in \CMN(n,r)} x_{\lambda} \HCR,\;\;
 x_{\mu}h \longmapsto \delta_{\mu, \co( A)} T_{\Ad}h,\;\forall  \mu \in \CMN(n,r), h \in \HCR,
\end{equation*}
where $A$ is the base of $\Ad$. Moreover, $\phi_\Ad$ has the parity $p(\Ad)$.
\end{prop}
We have, for $\Ad,\Bd\in \MNZ(n,r)$ with base $A,B$, respectively,
\begin{equation}\label{co=ro}
\phi_\Bd\phi_\Ad \ne  0 \implies \co(B) = \ro( A).
\end{equation}
By \eqref{eq cA},  if we write $T_\Ad=x_{\ro( A)}h'$,
where $h'={T_{d _A}} c_{\Ad}\Sigma_A,$ %\sum _{{\sigma} \in \D_{{\nu}_{_A}} \cap \fS_{\co( A)}} T_{\sigma}$,  
then $h'$ is a homogeneous element in the superalgebra $\HCR$ with parity $p(h')=p(\Ad)$
and the structure constants $\gamma^\Md_{\Bd,\Ad}$
appearing in $\phi_\Bd\phi_\Ad = \sum_{\Md} \gamma^\Md_{\Bd,\Ad} \phi_\Md$
are determined by writing $T_\Bd h'\in x_{\ro(B)}\HCR\cap \HCR x_{\co( A)}$
as a linear combination of $T_\Md$'s. In other words,
\begin{equation}\label{eq_gB}
z=\phi_\Bd\phi_\Ad(x_{\co( A)})=T_\Bd h'= \sum_{\Md \in \MNZ(n,r)} \gamma_{\Bd,\Ad}^\Md T_\Md.
\end{equation}
Moreover, we have {$\parity{\Md} = \parity{\Ad} + \parity{\Bd}$}, for every {$\Md$} with $\gamma_{\Bd,\Ad}^\Md\neq0$.

\begin{cor}\label{q-Sch}
Let $\qSchR$ be the $R$-submodule spanned by $\{\phi_{(A|{\rmO})}\mid A\in\MNR(n,r)\}$. Then $\qSchR$ is a subalgebra of $\QqnrR$ (with the same identity). Moreover, $\qSchR$ is isomorphic to the $q$-Schur algebra (of type $A$).
\end{cor}

\section{Some commutation relations in the Hecke-Clifford superalgebra }\label{sec_basicformulas}

In this section, %$R$ denotes a commutative ring such that {$2 \in R^\times$} (the group of units in $R$).
we establish a number of commutation relations  in $\HCR$
which will be useful to compute the multiplication formulas
in the queer $q$-Schur superalgebra.
%(i.e., the quantum queer Schur superalgebra).

Let {$E_{i,j} \in \MN(n)$} denote the matrix unit with entry $1$ at {$(i,j)$} position and {$0$} elsewhere. For any {$A=(a_{i,j}) \in \MNR(n,r)$} and {$h,k \in[1, n]$, $h<n$}, define matrices in $\MNR(n,r)$:
\begin{equation}\label{Ahk}
\aligned
	 A^+_{h,k} &:= A + E_{h,k} - E_{h+1, k}, \text{ if } a_{h+1,k}>0;\\
	  A^-_{h,k} &:= A - E_{h,k} + E_{h+1, k}, \text{ if } a_{h,k}>0.
	  \endaligned
\end{equation}
 As shown in \cite[Prop. 3.6]{DGZ}, explicit relations between the distinguished double coset representatives $d_A$ and $d_{A^+_{h,k}}$, $d_{A^-_{h,k}}$ are given via their reduced expressions. In fact, such matrices occur naturally in the product $\phi_{\Xd}\phi_\Ad$ for the $\Xd$ mentioned in the introduction.

% We first extend them to similar relations in $\HCR$.

We further define the partial sums of row $h$ of $A$ by setting $\AK(h,1)(A)=0$, $\BK(h,n)(A)=0$, and
\begin{equation}\label{prsum}
\aligned
\AK(h,k)(A)&:= \sum^{k-1}_{u=1}a_{h, u},\;\text{ for }2\leq k\leq n+1;\\
\BK(h,k)(A)&:= \sum^{n}_{j=k+1}a_{h, j},\;\text{ for }0\leq k\leq n-1.
\endaligned
\end{equation}
Note that $\AK(h,n+1)(A)=\ro(A)_h$, the $h$-th component of $\ro(A)$, and $\BK(h,0)(A)=\ro(A)_h$.

For $A=(a_{i,j})\in \MN(n)$ with {$\lambda=\ro(A)$}, {$r=|A|$}, $\AK(h,k)=\AK(h,k)(A)$, let
\begin{equation}\label{Ahat}
\widehat A=(\widehat a_{i,j}), \;\text{ where }\;\widehat a_{i,j}=\widetilde \lambda_{i-1}+\AK(i,j+1),\;\text{ and set }\;\widehat a_{1,0}=0, \widehat a_{i,0}=\widehat a_{i-1,n} (i>1).
\end{equation}
Observe that the matrix $\widehat A$ and the matrix $\widetilde A$ defined in \eqref{Atilde} are related by
 $\widetilde{\,{}^t\!A\,}={}^t(\widehat A)$, where ${}^t\!B$ is the transpose of $B$.

The following corner sub-matrices will be useful later on.
\begin{defn}\label{corner} For {$A=(a_{i,j}) \in \MNR(n,r)$} and $h,k\in[1,n]$, deleting $h$th row and $k$th column divides the matrix into four submatrices in general. We will call them, respectively, {\it the upper left/right corner matrix} and {\it the lower left/right corner matrix at $(h,k)$}. We denote the upper left/right corner matrix at $(h,k)$ by $\ulcm_{h,k}/\urcm_{h,k}$, and the lower left/right corner matrix at $(h,k)$ by $\llcm^{h,k}/\lrcm^{h,k}$.

We also set $\urcm_{1,k}=0$, $\urcm_{h,n}=0$, $\llcm^{h,1}=0$, and  $\llcm^{n,k}=0$, etc.
\end{defn}
We will use the lower left corner matrices to characterise a SDP condition in Section 4.

We first describe a reduced expression of $d_A$; see  \cite[Lemma 3.2]{DGZ}. For $\lambda=\ro(A)$, $\mu=\co(A)$,  define the ``hook sums'' matrix $\sigma(A)=(\sigma_{i,j})$ of $A$, where, for $i,j\in[1,n]$,
\begin{align}\label{sigmaij}
{\sigma}_{i,j}
	%:=\sum_{t=1}^{j-1}  \sum_{u=1}^{n}a_{u,t}		+ \sum_{u \le i, t \ge j}  a_{u,t}
		:= \widetilde\mu_{j-1}+\sum_{u \le i, t \ge j}  a_{u,t}
	= \widetilde{\lambda}_{i}	 + \sum_{u > i, t < j}  a_{u,t}.
\end{align} (Recall the $\widetilde{\cdot}$ notation in \eqref{latilde} and \eqref{Atilde} and see Figure 1 below for a depiction of $\sigma_{h-1,k}$.)

For $i\in[2,n],j\in[1,n-1]$, define $w_{i,j}\in\fS_r$ by
 \begin{equation}\label{wij}
\begin{aligned}
w_{i,j}=\left\{
\begin{array}{ll}
1,&\text{ if }a_{i,j}=0, \text{ or } a_{i,j}>0 \text{ and }\sigma_{i-1,j} = \widetilde{a}_{i-1,j},\\
w_{i,j}^{(1)}w_{i,j}^{(2)}\cdots w_{i,j}^{(a_{i,j})},& \text{ if }a_{i,j}>0 \text{ and }\sigma_{i-1,j} > \widetilde{a}_{i-1,j},
\end{array}
\right.
\end{aligned}
\end{equation}
where $w_{i,j}^{(k)}=s_{\sigma_{i-1,j}+k-1} s_{\sigma_{i-1,j}+k-2 } \cdots s_{\widetilde{a}_{i-1,j}+k}$ for $k\in[1,a_{i,j}]$.

Note that, for $h\in[2,n],k\in[1,n-1]$, the two corner matrices $\llcm^{h,k}$ and $\urcm_{h,k}$ play interesting roles in linking the entries of $\sigma(A)$, $\widetilde A$, and $\widehat A$ in the following way:
\begin{equation}\label{2corner}
\aligned
\sigma_{h-1,k}&=\widetilde{a}_{h-1,k}+\sum_{s<h, t> k}a_{s,t}=\widetilde{a}_{h-1,k}+|\urcm_{h,k}|\\
&=\widetilde{\lambda}_{h-1}+\overleftarrow{\bf r}^{k}_h+\sum_{i>h,j<k}a_{i,j}=\widehat a_{h,k-1}+|\llcm^{h,k}|.
\endaligned
\end{equation}
\begin{multicols}{2}
We use a picture (Figure 1) to indicate the various entry positions $(h-1,k)=\bullet$, $(h,k-1)=*$ and $(h,k)=\circ$.
The two grey areas represent the submatrices $\urcm_{h,k}$ and $\llcm^{h,k}$. (Recall Definition \ref{corner} for corner matrices.)

\begin{center}\label{hatvstilde}
\begin{tikzpicture}[scale=.6]
\draw (0.5,3) -- (4,3);
\draw (2,2) -- (4,2);
\draw(0.5,.6) -- (2,.6);
\draw (0.5,3) -- (0.5,.6);
\draw (2,2) -- (2,.6);
\draw (4,2) -- (4,3);
\fill (2.4,2) node {$_\bullet$};
\fill (2,1.6) node {$_*$};
\fill (2.4,1.6) node {$_\circ$};
\fill (0.15,2) node {\tiny\it h-1};
\fill (2.4,3.3) node {\rotatebox[origin=c]{90}{$_k$}};
\fill (2,3.35) node {\rotatebox[origin=c]{90}{\tiny\it k-1}};
\fill (.2,1.6) node {\tiny\it h};
%\fill (2,.2) node {\tiny k-1};
\fill (4.6,2.5) node {$_{\urcm_{h,k}}$};
\fill (1.2,0) node {$_{\llcm^{h,k}}$};
\fill (3.5,-.3) node {\tiny (Figure 1)};
\draw [fill=lightgray] (2.8, 2) rectangle (4,3);
\draw [fill=lightgray] (2,1.2) rectangle (0.5,.5);
\end{tikzpicture}
\end{center}
\end{multicols}

Thus, each $w_{i,j}^{(k)}$ has length $\ell(w_{i,j}^{(k)})=|\urcm_{i,j}|$.
%$$\aligned
%&(s_{\sigma_{i-1,j}} s_{\sigma_{i-1,j} - 1} \cdots s_{\widetilde{a}_{i-1,j} +1})\\
%&(s_{\sigma_{i-1,j}+1} s_{\sigma_{i-1,j}} \cdots s_{\widetilde{a}_{i-1,j} +2})\\
%\hspace{1.5in}	\cdots		(s_{\sigma_{i-1,j}+a_{i,j}-1} s_{\sigma_{i-1,j}+a_{i,j}-2 } \cdots s_{\widetilde{a}_{i,j}})$
By \cite[Lem.~3.2]{DGZ}, we have
\begin{equation}\label{d_A}
\begin{aligned}
	d_A =
		(w_{2,1}w_{3,1}\cdots w_{n,1})
		(w_{2,2}w_{3,2}\cdots w_{n,2})
		\cdots
		(w_{2,n-1}w_{3,n-1}\cdots w_{n,n-1}),
\end{aligned}
\end{equation}
and (cf. \cite[Exer.~8.2]{DDPW})
$$\ell(d_A)=\sum_{1\leq i,j\leq n}a_{i,j}|\urcm_{i,j}|.$$

.
\begin{rem}\label{rem:dA=1}
(1) We remark that the condition $\sigma_{i-1,j} = \widetilde{a}_{i-1,j}$ (i.e., $\urcm_{i,j}=0$) in \eqref{wij} is trivially true for $i=1$ or $j=n$. Thus, we may extend the definition of $w_{i,j}$ to include $w_{1,j}=1=w_{i,n}$, for all $i,j\in[1,n]$. Thus,
we may depict $d_A$ as the $n\times n$ matrix $w(A)=(w_{i,j})_{1\leq i,j\leq n}$ and regard $d_A$ as the product down column 1 of $w(A)$, then column 2  of $w(A)$, and so on. We may identify $d_A$ as $w(A)$.

%Moreover, if $A'=(a_{i_k,j_l})_{1\leq k\leq a\atop 1\leq l\leqq b}$ is a submatrix of $A$, say $A'=\ulcm_{h,k}$, and assume $w(A')=(w_{i_k,j_k})_{1\leq k\leq a\atop 1\leq l\leqq b}$ is the corresponding submatrix of $w(A)$, then $w(A')$ may be regarded as a subword of $d_A$.

(2) Observe the following:
%Note that, by definition, we have
\begin{equation*}
\aligned d_A=1&
\iff\sigma_{i-1,j}-\widetilde{a}_{i-1,j}=0,\text{ for all  $2 \le i \le  n,  1 \le j \le n-1$ with $a_{i,j}>0$;}
 \\
&\iff \urcm_{i,j}=0,% \sum_{u \le i-1 , t\geq j+1 }a_{u,t} = 0,
\text{ for all
 $2 \le i \le  n,  1 \le j \le n-1$ with $a_{i, j}> 0$.}
 \endaligned
\end{equation*}

(3) For convenience of later use, if $w_{i,j}\neq 1$, we may display all the simple reflections $s_k$ in $w_{i,j}$ as follows:
\begin{equation}\label{full wij}\aligned
&s_{\sigma_{i-1,j}} s_{\sigma_{i-1,j}-1 } \cdots s_{\widetilde{a}_{i-1,j}+1}\\
&s_{\sigma_{i-1,j}+1} s_{\sigma_{i-1,j} } \cdots s_{\widetilde{a}_{i-1,j}+2}\\
&\cdots\cdots\\
&s_{\sigma_{i-1,j}+a_{i,j}-1} s_{\sigma_{i-1,j}+a_{i,j}-2 } \cdots s_{\widetilde{a}_{i-1,j}+a_{i,j}}.
\endaligned
\end{equation}
If $s_k$ appears in the display, then we write $s_k|w_{i,j}$. If $w$ is a product of some $w_{i,j}$
and $w\neq1$,
then $s_k|w$ is well defined.
For such a $w$, define
\begin{equation}\label{max}
\text{max}(w):=\text{max}\{k\mid k\in[1,r), s_k|w\},\quad
  \text{min}(w):=\text{min}\{k\mid k\in[1,r), s_k|w\}.
  \end{equation}
We see easily from \eqref{full wij} that if $w_{i,j}\neq 1$, then
\begin{equation}\label{bound}
\text{max}(w_{i,j})= \sigma_{i-1,j}+a_{i,j}-1\;\;\text{ and }\;\; \text{min}(w_{i,j})= \widetilde{a}_{i-1,j}+1.
\end{equation}
\end{rem}
Call $w$ a {\it segment} of $d_A$ if $w$ is a product from $w_{i,j}$ to $w_{i',j'}$ in \eqref{d_A}. In the sequel, we often write
$d_A=\overleftarrow w\cdot w\cdot \overrightarrow w$, where $\overleftarrow w$ (resp., $\overrightarrow w$) is the segment of $d_A$ before (resp., after) $w$.
\begin{lem}\label{mmin} (1) If $w=w_{i,j}\cdots w_{h,k}$ is a segment of $d_A$ and $w\neq1$, then
$\text{\rm min}(w)\geq \widetilde a_{i-1,j}+1.$

(2) If $y=w({\ulcm_{h+1,k+1}}):=(w_{1,1}w_{2,1}w_{3,1}\cdots w_{h,1})
\cdots
(w_{1,k}w_{2,k}w_{3,k}\cdots w_{h,k})$ and $w\neq 1$, then $\text{\rm max}(y)\leq \sigma_{h-1,k} +a_{h,k}-1$.
\end{lem}
\begin{proof} (1) is clear since $\text{min}(w)=\text{min}(w_{i',j'})=\widetilde a_{i'-1,j'}+1$, where $w_{i',j'}$ is the first non-identity $w_{s,t}$ appearing in $w$, and $\widetilde a_{i'-1,j'}\geq \widetilde a_{i-1,j}$ whenever $j'>j$ or $j'=j$ but $i'>i$.

(2) follows from the fact that $\sigma_{i-1,j}+a_{i,j}-1\leq\sigma_{h-1,k}+a_{h,k}-1$ whenever $i\leq h$ and $j\leq k$.
\end{proof}

We are now ready to look at some commutation relations in $\HCR$.
The first  involves the distinguished representatives $d_A$,  $d_{A^+_{h,k}}$, and
$d_{A^-_{h,k}}$. It tells how to turn an expression involving $T_{d_A}$ to an expression involving $T_{d_{A^+_{h,k}}}$ or $T_{d_{A^-_{h,k}}}$.

The following elements in {$\HCR$} will be  frequently used later on in this paper.
For any {$1 \le i , j \le r-1$},
define
\begin{equation}\label{Tij}
\begin{aligned}
\TAIJ(i,j) &=\begin{cases}  1 + T_{i} + T_{i} T_{i+1} + \cdots +  T_{i} T_{i+1}  \cdots  T_{j}, &\text{ if }i\leq j;\\
1,&\text{ otherwise,}\end{cases}\\%\;\quad\TAIJ(i,i-1)=1;\\
\TDIJ(j,i) &=\begin{cases}  1 + T_{j} + T_{j} T_{j-1} + \cdots +  T_{j} T_{j-1}  \cdots  T_{i},&\text{ if }j\geq i;\\
1,&\text{ otherwise.}\end{cases}
\end{aligned}
\end{equation}
For example, in $\mathcal H(\fS_{r+1})_R$, $\Sigma_{\mathcal D_{(1,r)}}=\TAIJ(1,r)$ and $\Sigma_{\mathcal D_{(r,1)}}=\TDIJ(r,1)$.
\begin{lem}\label{prop_pjshift}
	For {$A =(a_{i,j})\in\MN(n)$}, let
	{$\lambda = \ro(A)$}, $\AK(h,k)=\AK(h,k)(A)$, and $\BK(h,k)=\BK(h,k)(A)$. %Keep the above notations.

{\rm(1)}
 For each $1\leq h\leq n-1$ and $1\leq k\leq n$ such that $a_{h+1,k}>0$, we have%, for $\AK(h+1,k)\leq j<\AK(h+1,k+1)$,
 %all $j\in[0,\lambda_{h+1})$
\begin{align*}
&\sum_{j=\AK(h+1,k)}^{\AK(h+1,k+1) - 1}
		T_{\widetilde{\lambda}_{h}+1} T_{\widetilde{\lambda}_{h}+2} \cdots T_{\widetilde{\lambda}_{h}+j}
		{T_{d_A}}
=
		T_{\widetilde{\lambda}_{h}} T_{\widetilde{\lambda}_{h}-1}\cdots T_{\widetilde{\lambda}_{h}-\BK(h,k)+1}
		{{\TDAP}}\TAIJ(\widetilde{a}_{h,k}+1,\widetilde{a}_{h,k}+a_{h+1,k}-1).
		%\TAIJ( {\widetilde{a}_{h,k}+1}}, {a_{h+1, k} - 1})
		%(\sum_{j=\AK(h+1,k)}^{\AK(h+1,k+1) - 1}T_{\widetilde{a}_{h,k}+1} \cdots T_{\widetilde{a}_{h,k}+p_j}).
\end{align*}

{\rm(2)}
For each $1\leq h\leq n$ and $1\leq k\leq n$ such that $a_{h,k}>0$, we have
\begin{align*}
&\sum_{j=\BK(h,k)}^{\BK(h,k-1) - 1}
	T_{\widetilde{\lambda}_{h}-1} \cdots T_{\widetilde{\lambda}_{h}-j }
	T_{d_{A}}
	= T_{\widetilde{\lambda}_{h}} T_{\widetilde{\lambda}_{h}+1}\cdots T_{\widetilde{\lambda}_{h}+\AK(h+1,k)-1}  T_{d_{A^-_{h,k}}}
	\TDIJ(\widetilde{a}_{h,k}-1,\widetilde{a}_{h,k}-a_{h,k}+1).
	%(\sum_{j=\BK(h,k)}^{\BK(h,k-1) - 1}T_{\widetilde{a}_{h,k}-1} \cdots T_{\widetilde{a}_{h,k}-q_j}).
\end{align*}
%Here $p_j=j-\AK(h+1,k)$ and $q_j = j - \BK(h,k)$. (Note that, in case {$j=0$}, we set
%	{$T_{\widetilde{\lambda}_{h}+1} \cdots T_{\widetilde{\lambda}_{h}+j }=1$} and
%	{$T_{\widetilde{\lambda}_{h}-1} \cdots T_{\widetilde{\lambda}_{h}-j } = 1$}).
\end{lem}
\begin{proof}
We only prove {\rm(1)}, and the proof of {\rm(2)} is similar.
For each {$j \in  [\AK(h+1,k), \AK(h+1,k+1) - 1]$} and  {$p_j=j-\AK(h+1,k)$},
according to \cite[Proposition 3.6{\rm(1)}]{DGZ},
we have
\begin{align*}
&s_{\widetilde{\lambda}_{h}+1} s_{\widetilde{\lambda}_{h}+2} \cdots s_{\widetilde{\lambda}_{h}+  j} d_{A}
=
		s_{\widetilde{\lambda}_{h}} s_{\widetilde{\lambda}_{h}-1}\cdots
		s_{\widetilde{\lambda}_{h} - \BK(h,k) +1}
		{d_{A^+_{h,k}}}
		 s_{\widetilde{a}_{h,k}+1} \cdots s_{\widetilde{a}_{h,k}+p_j}.
\end{align*}
Since both sides are reduced expressions, it follows that
\begin{align*}
&T_{\widetilde{\lambda}_{h}+1} T_{\widetilde{\lambda}_{h}+2} \cdots T_{\widetilde{\lambda}_{h}+  j} T_{d_{A}}
=
		T_{\widetilde{\lambda}_{h}} T_{\widetilde{\lambda}_{h}-1}\cdots
		T_{\widetilde{\lambda}_{h} - \BK(h,k) +1}
		T_{d_{A^+_{h,k}}}
		 T_{\widetilde{a}_{h,k}+1} \cdots T_{\widetilde{a}_{h,k}+p_j}.
\end{align*}
Hence, (1) follows.%and then equation in {\rm(1)} is proved.
\end{proof}

The next relations show how to interchange the tails $\Sigma_A$ and $\Sigma_{A^+_{h,k}}$ or $\Sigma_A$ and $\Sigma_{A^-_{h,k}}$, where $A^+_{h,k}$ and $A^-_{h,k}$ are defined in \eqref{Ahk} and their tails in \eqref{tail}.
\begin{lem}\label{Tnsum}
	Let  {$ h \in [1,n-1] , k \in [1,n] $}, and {$A=(a_{i,j}) \in \MNR(n,r)$}.
	 Then we have
\begin{align*}
{\rm(1)} \quad &\TAIJ( {\widetilde{a}_{h,k}} + 1,  { {\widetilde{a}_{h,k}} + a_{h+1,k}-1})\Sigma_A
		%\sum_{\sigma\in\mathcal{D}_{\nu_A}\cap {\fS}_\mu} T_{\sigma}
		= \TDIJ( {{\widetilde{a}_{h,k}}}, {{\widetilde{a}_{h,k}} - a_{h,k} + 1})\Sigma_{A^+_{h,k}},\; \text{ if }\; a_{h+1,k} \ge 1;\\
		%\sum_{\sigma\in\mathcal{D}_{\nu^+}\cap {\fS}_\mu} T_{\sigma}, \text{ if }\; a_{h+1,k} \ge 1;\\
{\rm(2)} \quad 	&\TDIJ( {{\widetilde{a}_{h,k}} - 1} ,   {{\widetilde{a}_{h,k}} - a_{h,k} + 1})\Sigma_A
		%\sum_{\sigma\in\mathcal{D}_{\nu_A}\cap {\fS}_\mu} T_{\sigma}
		 = \TAIJ( {{\widetilde{a}_{h,k}}} , {{\widetilde{a}_{h,k}} + a_{h+1,k} - 1})\Sigma_{A^-_{h,k}},\; \text{ if }\; a_{h,k} \ge 1.\\
	%\sum_{\sigma\in\mathcal{D}_{\nu^-}\cap {\fS}_\mu} T_{\sigma}, \text{ if }\;a_{h,k}\geq1.
\end{align*}
\end{lem}
\begin{proof}Let $\mu=\co(A)=\co(A^+_{h,k})$, $\nu=\nu_A$, $\nu^+=\nu_{A^+_{h,k}}$, and $\fS_{\delta^+}=\fS_{\nu}\cap\fS_{\nu^+}$ with $\delta^+$ as given in \eqref{nu+nu-}. Then, for $a_{h+1,k} \ge 1$,
$$\mathcal D_{\delta^+}\cap\fS_{\mu}=(\mathcal D_{\delta^+}\cap\fS_{\nu})(\mathcal D_\nu\cap\fS_\mu)=
(\mathcal D_{\delta^+}\cap\fS_{\nu^+})(\mathcal D_{\nu^+}\cap\fS_\mu),$$
where
 $$\aligned
& \mathcal D_{\delta^+}\cap\fS_{\nu}=\{1, s_{{\widetilde{a}_{h,k}} + 1}, s_{{\widetilde{a}_{h,k}} + 1}s_{{\widetilde{a}_{h,k}} + 2}, \cdots ,   s_{{\widetilde{a}_{h,k}} + 1} \cdots s_{{\widetilde{a}_{h,k}} + a_{h+1,k}-1}\}\\
 &\mathcal D_{\delta^+}\cap\fS_{\nu^+}=\{1, s_{\widetilde{a}_{h,k}}, s_{\widetilde{a}_{h,k}} s_{\widetilde{a}_{h,k} - 1}, \cdots,
	 s_{\widetilde{a}_{h,k}} \cdots s_{\widetilde{a}_{h,k} - a_{h,k} + 1}\}.
 \endaligned$$
%Thus, we have
% (see, e.g., \cite[Prop.~3.4]{GLL} for more details),
%\begin{align*}
	%&\Big(\sum_{ \mathcal D_{\delta^+}\cap\fS_{\nu}}\tau\Big)
	%	\sum_{\sigma\in\mathcal{D}_{\nu}\cap \mathfrak{S}_\mu}\sigma= \sum_{\sigma'\in \mathcal D_{\delta^+}\cap\fS_{\mu}}\sigma'	  =\Big(\sum_{ \mathcal D_{\delta^+}\cap\fS_{\nu^+}}\tau\Big)
	%\sum_{\sigma\in\mathcal{D}_{\nu^+}\cap \mathfrak{S}_\mu}\sigma.
%\end{align*}
Hence, by length additivity,
$$\TAIJ( {\widetilde{a}_{h,k}} + 1,  { {\widetilde{a}_{h,k}} + a_{h+1,k}-1})\Sigma_A= \Sigma_{ \mathcal D_{\delta^+}\cap\fS_{\nu}}\Sigma_{\mathcal D_\nu\cap\fS_\mu}=\Sigma_{ \mathcal D_{\delta^+}\cap\fS_{\nu^+}}\Sigma_{\mathcal D_{\nu^+}\cap\fS_\mu}= \TDIJ( {{\widetilde{a}_{h,k}}}, {{\widetilde{a}_{h,k}} - a_{h,k} + 1})\Sigma_{A^+_{h,k}},$$
proving assertion (1).  If $B=A^+_{h,k}$, then $b_{h,k}>0$ and $B^-_{h,k}=A$. Hence, assertion (2) follows from (1).
%Equation {\rm(2)}  can be proved by replacing {$A$} with {${A^-_{h,k}}$}.
\end{proof}

%The following follows from definitions.

\begin{lem}\label{xTinverse}
For {$\alpha \in \CMN(n, r)$},
{$1\leq u < r$, and $ j \ge 0$}, assume {$s_{u+1}, \cdots , s_{u+j} \in \fS_{\alpha}$}. Then	$x_{\alpha} \TAIJ({u},  {u+j}) = x_{\alpha} \TDIJ({u+j},  {u})=\STEP{ j+2} x_{\alpha} $ if, in addition, $s_u\in\fS_\alpha$.
Moreover, the following holds:
%If $s_u\not\in\fS_\alpha$, then {\hlt{ TODO: No matter $s_u$ in $\fS_\alpha$ or not }}
\begin{align*}
x_{\alpha} T_{u}^{-1} T_{u+1}^{-1} \cdots T_{u+j}^{-1}
&=
	{q}^{-1 - j} x_{\alpha}T_{u}   T_{u+1} \cdots T_{u + j}
	- {q}^{- j-1 }  ({q} - 1)   x_{\alpha} \TAIJ( {u}, { u + j  - 1}).
\end{align*}

\end{lem}
\begin{proof} Assertion (1) is clear by \eqref{Tixlambda} and \eqref{Tij}, while assertion (2) can be proved by induction on $j$, noting \eqref{Tick-inv}(1).
\end{proof}

Finally, we look at some commutation relations involving odd generators $c_i$. Recall from \eqref{cqij} the elements $c_{q,i,j}$.

\begin{lem}\label{xcT}
For {$\alpha \in \CMN(n, r)$},
{$1\leq u<r$, and $j \ge 0$}, assume
{$s_{u+1}, \cdots , s_{u+j} \in \fS_{\alpha}$} if {$j\ge 1$}.
\begin{itemize}
\item[(1)]
For {$j \ge 0$}, we have in $\HCR$
$$x_{\alpha}c_u T_{u} T_{u+1} \cdots T_{u+j}={q}^{j+1}   x_{\alpha}  {T^{-1}(u,j+1)} c_{u+j+1}+
({q} - 1)  {q}^{j}  x_{\alpha}\sum_{k=0}^{j} {T^{-1}(u,k)} c_{u+k}$$
			%( c_{u} +  T_{u}^{-1}c_{u+1} + T_{u}^{-1} T_{u+1}^{-1} c_{u+2} + \cdots
			%+  T_{u}^{-1} T_{u+1}^{-1} \cdots  T_{u+j-1}^{-1} c_{u+j}  ).
where $T^{-1}(u, k)=
\begin{cases}
1, &\text{ if }k=0;\\
T_{u}^{-1} T_{u+1}^{-1} T_{u+2}^{-1} \cdots T_{u+k-1}^{-1},&\text{ if }k \in[1,j].
\end{cases}$

%In particular, if $s_u\in\fS_\alpha$,
%then \hlt{$x_{\alpha}c_u T_{u} T_{u+1} \cdots T_{u+j}=x_{\alpha} c_{q, u, u+j+1}$}s.
\item[(2)] Further, we have
\begin{align*}
&x_{\alpha}c_{u}  \TAIJ({u}, {u+j} )
 ={q}^{j+1} x_{\alpha} ( c_{u} + T_{u}^{-1} c_{u+1} +  T_{u}^{-1} T_{u+1}^{-1}   c_{u+2}
	+ \cdots
	+ T_{u}^{-1} T_{u+1}^{-1} T_{u+2}^{-1} \cdots T_{u+j}^{-1}  c_{u+j+1} ).
\end{align*}
If, in addition, $s_u\in\fS_\alpha$, then
$
x_{\alpha} c_{u} \TAIJ( {u}, {u+j})
= x_{\alpha} c_{q, u, u+j+1}.
$
\item[(3)] If {$s_{u},  s_{u-1} \cdots, s_{u-j} \in \fS_{\alpha}$},
then
$
x_{\alpha} c_{u+1} \TDIJ( {u}, {u-j})
= x_{\alpha} c_{q, u-j, u+1}.
$
\item[(4)] For $1\leq i\leq j\leq r$, $(c_{q,i,j})^2=-{\STEPX{j-i+1}{q^2}}$.
\end{itemize}
\end{lem}
\begin{proof}
We first prove (1).
Repeatedly applying \eqref{Tick-inv}(3) gives
\begin{align*}
x_{\alpha}c_u &T_{u} T_{u+1} \cdots T_{u+j} \\
&= x_{\alpha} ({q} T_{u}^{-1} c_{u+1}  + ({q} - 1) c_{u}) T_{u+1} \cdots T_{u+j}  \\
&=  {q}  x_{\alpha}  T_{u}^{-1} c_{u+1} T_{u+1} \cdots T_{u+j}   + ({q} - 1) {q}^{j}  x_{\alpha}  c_{u}   \\
&=  {q}  x_{\alpha}  T_{u}^{-1} ( {q} T_{u+1}^{-1} c_{u+2}
	+ ({q} - 1) c_{u+1} ) T_{u+2} \cdots T_{u+j}
	+ ({q} - 1) {q}^{j} x_{\alpha}  c_{u}  \\
&= {q}^2   x_{\alpha}  T_{u}^{-1} T_{u+1}^{-1} c_{u+2} T_{u+2} \cdots T_{u+j}
	+ ({q} - 1)  {q}^{j}  x_{\alpha}  (  c_{u}  +  T_{u}^{-1}c_{u+1} ) \\
& = \cdots \\
&=
		{q}^{j+1}   x_{\alpha}  T_{u}^{-1} T_{u+1}^{-1} T_{u+2}^{-1} \cdots T_{u+j}^{-1}  c_{u+j+1}\\
		& \qquad + ({q} - 1)  {q}^{j}  x_{\alpha}
			( c_{u} +  T_{u}^{-1}c_{u+1} + T_{u}^{-1} T_{u+1}^{-1} c_{u+2} + \cdots
			+  T_{u}^{-1} T_{u+1}^{-1} \cdots  T_{u+j-1}^{-1} c_{u+j}),
\end{align*}
proving (1). Now, by (1), we have
$$\aligned
x_{\alpha}c_{u}  \TAIJ({u}, {u+j} )&=x_\alpha c_u+\sum_{i=0}^jx_\alpha c_uT_u\cdots T_{u+i}\\
&=x_\alpha c_u+\sum_{i=0}^j\big(({q} - 1)  {q}^{i}  x_{\alpha}\sum_{k=0}^{i}
	{T^{-1}(u,k) } c_{u+k}+{q}^{i+1}   x_{\alpha}  { T^{-1}(u,i+1) } c_{u+i+1}\big)\\
&=\sum_{k=0}^j\big(q^k+\sum_{i=k}^j(q-1)q^i\big)x_\alpha
	 T^{-1}(u,k)  c_{u+k}+q^{j+1}x_\alpha  {T^{-1}(u,j+1)} c_{u+j+1}\\
&={q}^{j+1} x_{\alpha} ( c_{u} + T_{u}^{-1} c_{u+1} +  T_{u}^{-1} T_{u+1}^{-1}   c_{u+2}
	+ \cdots
	+ T_{u}^{-1} T_{u+1}^{-1} T_{u+2}^{-1} \cdots T_{u+j}^{-1}  c_{u+j+1} ),
\endaligned
$$
 since the telescoping sum $q^k+\sum_{i=k}^j(q-1)q^i=q^{j+1}$.
This proves the first assertion of (2). The second assertion in (2) is clear.

For assertion (3), $s_{u-i}\in\fS_{\alpha}$ implies $x_{\alpha}T_{u-i}=qx_{\alpha}$ for $0\leq i\leq j$.
Thus, by \eqref{Hecke-Cliff},
we obtain
\begin{align*}
x_{\alpha}c_{u+1}\TDIJ(u,u-j)&=x_{\alpha}c_{u+1}(1+T_u+T_uT_{u-1}+\cdots T_uT_{u-1}\cdots T_{u-j})\\
&=x_{\alpha}(c_{u+1}+T_uc_u+T_uT_{u-1}c_{u-1}+\cdots+T_uT_{u-1}\cdots T_{u-j}c_{u-j})\\
&=x_{\alpha}(c_{u+1}+qc_u+\cdots+q^{j+1}c_{u-j})=x_{\alpha} c_{q, u-j, u+1},
\end{align*}
proving (3).

Finally, (4) is straightforward by an induction.
\end{proof}

\section{Distinguished double coset representatives and the SDP condition}

Recall from \eqref{jmath} the distinguished double coset representative $d_A\in\fS_r$ associated with $A\in M_n(\NN)_r$. To compute some structure constants in \eqref{eq cA}, especially in the odd case, we need to impose some commutation relations between $T_{d_A}$ and certain $c$-elements. The SDP condition introduced in this section is convenient for the purpose.

This element $d_A$ has a reduced expression as given in \eqref{d_A}.
As a permutation, $d_A$ is given in \cite[p. 424]{Du} (or \cite[Ex. 8.2]{DDPW}). More precisely, for $A=(a_{i,j})\in \MN(n)$ with $\widetilde A=(\widetilde a_{i,j})$ and $\widehat A=(\widehat a_{i,j})$ as in \eqref{Atilde} and \eqref{Ahat}, the description of the pseudo-matrix $\dddot A$ associated with $A$ in loc. cit. (see Example \ref{pseudo} below) can be easily interpreted as follow:
\begin{equation}\label{map d_A}
d_A(\widetilde{a}_{h-1,k}+ p )=\widehat a_{h,k-1} + p ,
\end{equation}
for all $h,k\in[1,n] $ with $a_{h,k}>0$ and $p\in[1,a_{h,k}]$. In other words, $d_A$ is the permutation obtained by concatenating the permutations
\begin{equation}\label{a_ij}
\left(\begin{matrix}&
\widetilde{a}_{i-1,j}+1 &\widetilde{a}_{i-1,j}+2&\cdots&\widetilde{a}_{i-1,j}+a_{i,j}\\
&\widehat a_{i,j-1} +1
& \widehat a_{i,j-1}+2
&\cdots&\widehat a_{i,j-1}+a_{i,j}
\end{matrix}\right)
\end{equation}
for all $a_{i,j}>0$ via the ordering from top to bottom of column 1, then top to bottom of column 2, and so on.

\begin{exam} \label{pseudo}
Let $A={\tiny\begin{pmatrix} 0&3&2\\1&1&0\\2&0&0\end{pmatrix}}\in M_3(\NN)_{9}$. Then
$$d_A={\tiny\begin{pmatrix} 1&2&3&4&5&6&7&8&9\\6&8&9&1&2&3&7&4&5\end{pmatrix}}.$$
The images $d_A(1),d_A(2),\ldots,d_A(9)$ of $d_A$ (i.e., the second row in $d_A$) is obtained as follows: Form the pseudo-matrix
$\dddot{A}={\tiny\begin{pmatrix} -&(1,2,3)&(4,5)\\6&7&-\\(8,9)&-&-\end{pmatrix}}$
obtained by replacing each nonzero entry $a_{i,j}$ by the second row of \eqref{a_ij} --- a sequence of length $a_{i,j}$. In other words, $\dddot A$ is obtained by filling $1,2,\ldots,9$ along each row from left to right (using the sequence length $a_{i,j}$), and then down successive rows. Now, the sequence $d_A(1),\ldots, d_A(9)$ is obtained by reading the numbers from left to right inside each sequence associated with $a_{i,j}$ down column 1, then column 2, and so on.
\end{exam}

By specialising $q$ to 1, $\SerA=\HCR|_{q=1}$ is the Sergeev algebra (see \cite{Ser}),
which is the ``semi-direct (or smash) product'' $\C_r\rtimes R\fS_r$. (See \cite{BGJKW} for another such a kind of example.)
% then, for $\AK(h,k) =\AK(n,k)(\ABSUM{A})$, we have in $\SerA$ % and $\BK(h,k) =\BK(n,k)(\ABSUM{A})$,
Then, we have $wc_i=c_{w(i)}w$ in $\SerA$ and, in particular,
\begin{displaymath}
	  {d_A} c_{\widetilde{a}_{h-1,k}+ p}=c_{\widehat a_{h,k-1}+ p} {d_A}
\end{displaymath}
for all  $1\leq h,k\leq n$ with $a_{h,k}>0$ and $1\leq p\leq a_{h,k}$.
However, if we replace {$d_A$} by {$T_{d_A}$},  this commutation formula may not hold in the Hecke-Clifford superalgebra $\HCR$.
So we introduce the following definition.
\begin{defn}\label{defn:SDP}
Let $A=(a_{i,j})\in \MN(n)$.
%let {$\lambda=\ro(A)$}, {$r=|A|$}, $\AK(h,k)=\AK(h,k)(A)$, and $\widehat a_{h,k-1}=\lambda_{h-1}+\AK(h,k)$.
If $a_{h,k}>0$ and
\begin{displaymath}
 c_{\widehat a_{h,k-1}+p} {T_{d_A}}={T_{d_A}} c_{\widetilde{a}_{h-1,k} + p}
	 \qquad(\mbox{in }\HCR)
\end{displaymath}
for each {$p \in [1,a_{h,k}]$}, then {$A$} is said to
satisfy the {\it semi-direct product} (SDP) {\it condition} at {$(h, k)$}.

If {$A$} satisfies the  SDP condition at {$(h, k)$} for every {$k\in [1,n]$} (resp., $h\in[1,n]$) with {$a_{h,k}>0$},
then {$A$} is said to satisfy the SDP condition on the $h$th row (resp., $k$th column).
\end{defn}
Using the element defined in \eqref{cqij}, $A$ satisfied the SDP condition at $(h,k)$ implies
\begin{equation}\label{longSDP}
c_{q,\widehat a_{h,k-1}+1,\widehat a_{h,k-1}+a_{h,k}} {T_{d_A}}
	 = {T_{d_{A}}} c_{q,\widetilde{a}_{h-1,k} + 1,\widetilde{a}_{h-1,k} +a_{h,k}}.
	 \end{equation}

\begin{exam}\label{lem_dA1}For the given $A$ as above, if $d_A= 1$,
then, for  any   {$h,k \in [1, n]$}  satisfying {${a}_{h, k}>0$},
 the equation \eqref{map d_A} implies
{$\widehat a_{h,k-1} + p = \widetilde{a}_{h-1,k}+ p$}.
Hence, $A$ satisfies the SDP condition everywhere (i.e., at any {$(h,k)$} with ${a}_{h, k}>0$).

 In particular, if %$\Ad =(\SE{A}|\SO{A})$ satisfying that
 $ A=(a_{i,j})$ is one of the following matrices in $\MNR(n,r)$:
 $$%{\diag}(\lambda),\quad
{\diag}(\lambda) + \sum_{i=1}^{n-1} u_i E_{i, i+1},\quad
{\diag}(\lambda) + \sum_{i=1}^{n-1} u_i E_{i+1, i},$$
where $\lambda\in\NN^n$ and ${\bf u}=(u_1,\ldots,u_{n-1}) \in \NN^{n-1}$,
then {$d_A = 1$},
and $A$ satisfies the SDP condition everywhere.
%at {$(h,k)$}  for any {$h,k$} with {${a}_{h, k}>0$}.

In general, one checks easily by Remark \ref{rem:dA=1}(2) (cf. Example \ref{pseudo}) that $d_A=1$ if and only if the nonzero entries of $A$ are on a zig-zag line of the shape
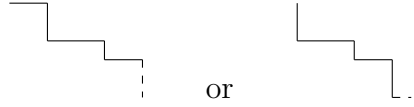
\begin{equation*}
\begin{tikzpicture}[scale=.5]
\draw (0,5) -- (1,5);
\draw (1,5) -- (1,4);
\draw(1,4) -- (2.5,4);
\draw(2.5,4) -- (2.5,3.5);
\draw(2.5,3.5) -- (3.5,3.5);
\draw[dashed](3.5,3.5) -- (3.5,2.5);
\end{tikzpicture}\qquad\text{or}\qquad
\begin{tikzpicture}[scale=.5]
%\draw (0,5) -- (1,5);
\draw (1,5) -- (1,4);
\draw(1,4) -- (2.5,4);
\draw(2.5,4) -- (2.5,3.5);
\draw(2.5,3.5) -- (3.5,3.5);
\draw(3.5,3.5) -- (3.5,2.5);
\draw[dashed](3.5,2.5) -- (4, 2.5);
\end{tikzpicture}
\end{equation*}
%In other words, the pseudo-matrix $\dddot A$ obtained by filling $1,2,\ldots,r$ along rows is the same pseudo-matrix obtained by filling $1,2,\ldots,r$ down columns.
%does not contain a $2\times 2$ submatrix whose (1,1),(1,2), and (2,1) entries are all nonzero. This is equivalent to say that $ A$ is a block diagonal matrix in $\MNR(n,r)$ with blocks of the form $U^{(k)}_{\lambda,\bf u}$ and $L^{(k)}_{\lambda,\bf u}$ on the diagonal.
\end{exam}

\begin{rem}The SDP condition on the base $A=A^{\bar0}+A^{\bar1}$ of $\Ad=(A^{\bar0}|A^{\bar1})$ makes the computation for some of the multiplication formulas in Sections 5 and 6 possible. Interestingly, these multiplication formulas associated with matrices that satisfy some SDP conditions play a key role for the establishment of the BLM type realisation for quantum queer supergroups.
\end{rem}

We now derive an equivalent condition for a matrix to satisfy the SDP condition at $(h,k)$. %We will use the lower left corner matrix $\llcm^{h,k}$ defined in Definition \ref{corner} to describe them.
%In the following Lemmas \ref{shiftonN}, and \ref{block form}--\ref{triang_aaa} and Corollaries \ref{shifton-l} and \ref{triang_aaa_cor}, we list the cases where a matrix satisfying the SPD conditions. We will see in \cite{DGLW} that these cases are sufficient for solving the BLM type realisation problem.

\begin{lem}\label{CT2}
(1) If $1\leq i\leq t\leq j<n$, then we have
\begin{equation*}
c_tT_jT_{j-1}\cdots T_i=T_jT_{j-1}\cdots T_ic_{t+1}+(q-1)T_j\cdots T_{t+1}T_{t-1}\cdots T_i(c_i-c_{t+1}).
\end{equation*}

(2) Let $A=(a_{i,j})\in \MN(n)$.
If $a_{h,k}>0$, then there are some {$f^{\cdots}_{w,i}\in R$} such that, for all $p\in[1,a_{h,k}]$, we have in $\HCR$
\begin{equation*}
	c_{\widehat a_{h,k-1}+p} {T_{d_A}}
	 =  {T_{d_{A}}} c_{\widetilde{a}_{h-1,k} + p}
	 + \sum_{w < {d_{A}},i\in[1,r]} f^{\cdots}_{w,i} T_w c_{i}.
\end{equation*}
\end{lem}
\begin{proof}
According to \eqref{Hecke-Cliff}, we have
\begin{align*}
c_tT_jT_{j-1}\cdots T_i&=T_j\cdots T_{t+1}c_tT_tT_{t-1}\cdots T_i\\
&=T_j\cdots T_{t+1}(T_tc_{t+1}+(q-1)(c_t-c_{t+1}))T_{t-1}\cdots T_i\\
&=T_j\cdots T_{t+1}(T_tc_{t+1})T_{t-1}\cdots T_i+T_j\cdots T_{t+1}((q-1)(c_t-c_{t+1}))T_{t-1}\cdots T_i\\
&=T_jT_{j-1}\cdots T_ic_{t+1}+(q-1)T_j\cdots T_{t+1}T_{t-1}\cdots T_i(c_i-c_{t+1}).
\end{align*}
Assertion (2) follows from \eqref{homog}.
\end{proof}
Using the notation in \eqref{cqij} and \eqref{Ahat}, Lemma \ref{CT2}(2) implies
\begin{equation}\label{longSDP2}
	c_{q,\widehat a_{h,k-1}+1,\widehat a_{h,k-1}+a_{h,k}} {T_{d_A}}
	 =
	 {T_{d_{A}}} c_{q,\widetilde{a}_{h-1,k} + 1,\widetilde{a}_{h-1,k} +a_{h,k}}
	 + \sum_{w < {d_{A}},i\in[1,r]} g^{\cdots}_{w,i} T_w c_{i}.
\end{equation}

Recall the lower left corner matrix $\llcm^{h,k}$ defined in Definition \ref{corner}.

\begin{thm}\label{CdA1}
Let $A\in \MN(n)$ and $h,k\in[1,n]$. %with $a_{h,k}> 0$ and $\ro(A)=\lambda$,
Then $A$ satisfies the SDP condition at $(h,k)$ if and only if $a_{h,k}>0$ and $a_{i,j}=0$, for $i>h$ and $j<k$ (equivalently, the lower left corner matrix $\llcm^{h,k}$ at $(h,k)$ is 0).
%Specially, if $a_{h,1}> 0$, $A$ satisfies SDP condition at position $(h,1)$.
\end{thm}
\begin{proof}
We first recall from \eqref{d_A} the reduced expression of $d_A$:
\begin{equation*}
\begin{aligned}
	d_A =
		(w_{2,1}w_{3,1}\cdots w_{n,1})
		(w_{2,2}w_{3,2}\cdots w_{n,2})
		\cdots
		(w_{2,n-1}w_{3,n-1}\cdots w_{n,n-1}) ,
\end{aligned}
\end{equation*}
where $w_{i,j}$, for $i\in[2,n],j\in[1,n-1]$, are defined in \eqref{wij}. Thus, if $w_{i,j}\neq1$, then $a_{i,j}>0$, $|\urcm_{i,j}|\neq0$, and $w_{i,j}=w_{i,j}^{(1)}w_{i,j}^{(2)}\cdots w_{i,j}^{(a_{i,j})}$ has length $a_{i,j}|\urcm_{i,j}|$, where, for $k\in[1,a_{i,j}]$,
 \begin{equation}\label{w-i,j,u}
 w_{i,j}^{(k)}=(s_{\sigma_{i-1,j}+k-1} s_{\sigma_{i-1,j}+k-2} \cdots s_{\widetilde{a}_{i-1,j} +k}).\end{equation}

Moreover, we also adopt the following notational scheme when breaking $d_A$ as a product of ``segments'':
\begin{equation}\label{d_A-in-3}
d_A=\overleftarrow{(d_A)}_{i,j}\cdot w_{i,j}\cdot \overrightarrow{(d_A)}_{i,j}=\overleftarrow{(d_A)}_{i,j}^{(k)}\cdot w_{i,j}^{(k)}\cdot \overrightarrow{(d_A)}_{i,j}^{(k)}
\end{equation}

We first prove that ``if'' part.

Suppose $a_{h,k}>0$ and $a_{i,j}=0$ for all $i>h$ and $j<k$, i.e., $\llcm^{h,k}=0$. We prove that $A$ satisfies SDP condition at $(h,k)$: $c_{\ha_{h,k-1}+p} {T_{d_A}}={T_{d_A}} c_{\widetilde{a}_{h-1,k} + p}$, for all $p\in[1,a_{h,k}]$.

By the hypothesis, we have $w_{i,j}=1$ for all $i>h$ and $j<k$. Thus, $d_A$ has the form
\begin{equation}
\begin{aligned}
d_A =&
		(w_{2,1}w_{3,1}\cdots w_{h,1})
		\cdots
		(w_{2,k-1}w_{3,k-1}\cdots w_{h,k-1}) (w_{2,k}w_{3,k}\cdots w_{h-1,k}\centerdot w_{h,k}\cdots w_{n,k})\\
		&
        (w_{2,k+1}w_{3,k+1}\cdots w_{n,k+1})\cdots (w_{2,n-1}w_{3,n-1}\cdots w_{n,n-1})
\end{aligned}
\end{equation}
Consider the maximal index defined in \eqref{max} for
$$w:=(w_{2,1}w_{3,1}\cdots w_{h,1})
\cdots
(w_{2,k-1}w_{3,k-1}\cdots w_{h,k-1}).$$
By Lemma \ref{mmin}(2), $m_1=\text{max}(w)<\sigma_{h-1,k-1}+a_{h,k-1}$. Similarly, $m_2=\text{max}(w_{2,k}w_{3,k}\cdots w_{h-1,k})<\sigma_{h-2,k}+a_{h-1,k}$.
Now, $\llcm^{h,k}=0$ implies
%Since $a_{i,j}=0$ for $i>h$ and $j<k$,
\begin{equation*}
\begin{aligned}
&\,\quad\sigma_{h-2,k}+a_{h-1,k}=\tilde{\lambda}_{h-2}+\sum_{j\leq k}a_{h-1,j}+\sum_{j<k}a_{h,j}=\tilde{\lambda}_{h-2}+\sum_{j\leq k}a_{h-1,j}+\overleftarrow{\bf r}^k_h\\
&\leq\tilde{\lambda}_{h-1}+\overleftarrow{\bf r}^k_h(=\sigma_{h-1,k-1}+a_{h,k-1})<\tilde{\lambda}_{h-1}+\overleftarrow{\bf r}^{k}_h+p=\hahk+p,\\
\end{aligned}
\end{equation*}
for all $p\in[1,a_{h,k}]$. Hence, $\hahk+p>\text{max}(m_1,m_2)+1$. This implies,
by the notation in \eqref{d_A-in-3}, that $c_{\hahk+p}T_{\overleftarrow{(d_A)}_{h,k}}=T_{\overleftarrow{(d_A)}_{h,k}}c_{\hahk+p}$ and so
\begin{equation}\label{xxx}
c_{\hahk+p}T_{d_A}=T_{\overleftarrow{(d_A)}_{h,k}}c_{\hahk+p} T_{w_{h,k}}T_{ \overrightarrow{(d_A)}_{h,k}}.
\end{equation}
% commutes with these $T_{w_{i,j}}$ for $j<k$ or $i<h$ and $j=k$.

We now commute $c_{\hahk+p}T_{w_{h,k}}$ for a fixed $p\in[1,a_{h,k}]$. The hypothesis $\llcm^{h,k}=0$ together with \eqref{2corner} gives
 $$\sigma_{h-1,k}=\tilde{\lambda}_{h-1}+\overleftarrow{\bf r}^{k}_h=\hahk.$$
 Recall that the $j$th factor of $w_{h,k}=w_{h,k}^{(1)}w_{h,k}^{(2)}\cdots w_{h,k}^{(a_{h,k})}$ has the form
 $$w_{h,k}^{(j)}=s_{\hahk+j-1}s_{\hahk+j-2}\cdots s_{\tilde{a}_{h-1,k}+j}.$$
If $j<p$,  the index of $s_i$ in the expression of $w_{h,k}^{(j)}$  is smaller than $\hahk+p-1$. So,
 $c_{\hahk+p}$ commutes with these $T_{w_{h,k}^{(j)}}$ .
If $j=p$, repeatedly applying \eqref{Hecke-Cliff} yields
 \begin{equation}
 \begin{aligned}
 c_{\hahk+p}T_{w_{h,k}^{(p)}}
 =c_{\hahk+p}T_{\hahk+p-1}\cdots T_{\widetilde a_{h-1,k}+p}
 =T_{w_{h,k}^{(p)}}c_{\tilde{a}_{h-1,k}+p}.
 \end{aligned}
 \end{equation}
Finally, if $j>p$, then $\tilde{a}_{h-1,k}+p<\tilde{a}_{h-1,k}+j$. Hence,
 $c_{\tilde{a}_{h-1,k}+p}$ commutes with $T_{w_{h,k}^{(j)}}$ and \eqref{xxx} becomes
 $c_{\hahk+p}T_{d_A}=T_{\overleftarrow{(d_A)}_{h,k}}T_{w_{h,k}}c_{\tilde{a}_{h-1,k}+p}T_{ \overrightarrow{(d_A)}_{h,k}}$.

It remains to prove that $c_{\tilde{a}_{h-1,k}+p}T_{ \overrightarrow{(d_A)}_{h,k}}=T_{ \overrightarrow{(d_A)}_{h,k}}c_{\widetilde{a}_{h-1,k}+p}$.  This is clear, since
 $$\overrightarrow{(d_A)}_{h,k}=w_{h+1,k}\cdots w_{n,k}
        (w_{2,k+1}w_{3,k+1}\cdots w_{n,k+1})\cdots (w_{2,n-1}w_{3,n-1}\cdots w_{n,n-1}),$$
and Lemma \ref{mmin}(1) implies
$$
\text{min}(\overrightarrow{(d_A)}_{h,k})\geq \widetilde a_{h,k}+1>\widetilde{a}_{h-1,k}+p.$$

 We now prove the ``only if'' part. This is equivalent to proving that
  if $a_{h,k}>0$ and $\llcm^{h,k}\neq0$, then $A$ does not satisfy the SDP condition at $(h,k)$.

 Suppose that there exists $a_{i,j}>0$ for some $i,j$ with $i>h$ and $j<k$. Then we must have $k>1$ (and $h<n$).
 Set
$$t=\mbox{min}\{j\mid a_{i,j}>0, i>h,j<k\},~l=\mbox{min}\{i\mid a_{i,t}>0,i>h\}.$$
 Then $a_{l,t}$ is the top left most nonzero entry in the corner matrix $\llcm^{h,k}$. Thus, $a_{i,j}=0$ for $i>h$ and $j<t$ or $h<i<l$ and $j=t$, and so, $w_{i,j}=1$ in these cases.
Hence, $d_A$ takes the form
\begin{equation}
\begin{aligned}
d_A =&
		(w_{2,1}w_{3,1}\cdots w_{h,1})
		\cdots
		(w_{2,t-1}w_{3,t-1}\cdots w_{h,t-1}) (w_{2,t}w_{3,t}\cdots w_{h,t} \centerdot w_{l,t}^{(1)}\cdots w_{n,t})\\
       & (w_{2,t+1}w_{3,t+1}\cdots w_{n,t+1})\cdots (w_{2,n-1}w_{3,n-1}\cdots w_{n,n-1})=\overleftarrow{(d_A)}_{l,t}^{(1)}\cdot w_{l,t}^{(1)}\cdot \overrightarrow{(d_A)}_{l,t}^{(1)}
\end{aligned}
\end{equation}
Now we derive the commutation formula for $ c_{\hahk+p} T_{d_A}$.

First,  we have
\begin{equation}\label{eq0}
 c_{\hahk+p}T_{\overleftarrow{(d_A)}_{l,t}^{(1)}}=T_{\overleftarrow{(d_A)}_{l,t}^{(1)}}c_{\hahk+p},
\end{equation} since, by Lemma \ref{mmin}(2) and noting $\llcm^{h,t}=0$,
$$\aligned
\text{max}(\overleftarrow{(d_A)}_{l,t}^{(1)})&\leq \sigma_{h-1,t}+a_{h,t}-1=\ha_{h,t-1}+\sum_{i>h,j<t}a_{i,j}+a_{h,t}-1=\ha_{h,t-1}-1<\hahk.
\endaligned$$

Second, we consider the commutation relation for $c_{\hahk+p} T_{w_{l,t}^{(1)}}$. %By \eqref{wij},
%According to the expression of $w_{i,j}$,
%$$w_{l,t}^{(1)}=s_{\sigma_{l-1,t}} \cdots s_{p'} \cdots s_{\widetilde{a}_{l-1,t} +1}$$
The hypothesis $a_{h,k}>0$ together with $a_{i,j}=0$, for $i>h$ and $j<t$ (and noting $l>h$) imply
$$\sigma_{l-1,t}=\tilde{\lambda}_{l-1}+\sum_{i>l-1,j<t}a_{i,j}=\tilde{\lambda}_{l-1}\geq \tilde{\lambda}_h\geq \tilde{\lambda}_{h-1}+\overleftarrow{\bf r}^{k}_h+p=\hahk+p=:p'.$$
On the other hand, $a_{l,t}>0$ together with $a_{i,t}=0$, for $h<i<l$, and $\llcm^{h,t}=0$ implies $\tilde{a}_{l-1,t}=\tilde{a}_{h,t}$ and
 $$\tilde{a}_{l-1,t}+1=\tilde{a}_{h,t}+1\leq \tilde{\lambda}_{h-1}+\overleftarrow{\bf r}^{k}_h+p=p'.$$
Hence, $\sigma_{l-1,t}\geq p'\geq \tilde{a}_{l-1,t}+1$ and so, $w_{l,t}^{(1)}=s_{\sigma_{l-1,t}} \cdots s_{p'} \cdots s_{\widetilde{a}_{l-1,t} +1}$ (i.e., $s_{p'}|w_{l,t}^{(1)}$).
Applying Lemma \ref{CT2}(1) yields
\begin{equation}\label{lt1}
\aligned c_{p'} T_{w_{l,t}^{(1)}}&=T_{w_{l,t}^{(1)}} c_{p'+1}+(q-1)T_{{w_{l,t}^{(1)}}\backslash s_{p'}}(c_{\tilde{a}_{l-1,t}+1}-c_{p'+1})\\
&=(q-1)T_{{w_{l,t}^{(1)}}\backslash s_{p'}}c_{\tilde{a}_{l-1,t}+1}+(T_{w_{l,t}^{(1)}}-(q-1)T_{{w_{l,t}^{(1)}}\backslash s_{p'}})c_{p'+1}.
\endaligned
\end{equation}
where ${w_{l,t}^{(1)}}\backslash s_{p'}$ ($p'=\hahk+p$) is the element obtained by deleting $s_{p'}$ from $w_{l,t}^{(1)}$.
%Next, we may use induction on $q\in[1,a_{l,t}]$ with
%$\sigma_{l-1,t}+q-1\geq \tilde{\lambda}_{h-1}+\overleftarrow{r}^{k}_h+p+q-1\geq \tilde{a}_{l-1,t}+q$
%Since $s_{\tilde{a}_{l-1,t}+2}$ has the minimal index among all $s_i$ in $w_{j,t}^{(2)}\cdots w_{l,j}^{(a_{l,t})}$

Finally, consider the multiplication \eqref{lt1}$\times T_{\overrightarrow{(d_A)}_{l,t}^{(1)}}$.
We must analyse the commutation relations for
$$(1)\quad c_{\tilde{a}_{l-1,t}+1}T_{\overrightarrow{(d_A)}_{l,t}^{(1)}},\qquad (2)\quad c_{p'+1}T_{\overrightarrow{(d_A)}_{l,t}^{(1)}}.$$

For (1), observe that the first factor of $\overrightarrow{(d_A)}_{l,t}^{(1)}$ is $w_{l,t}^{(2)}$. By Lemma \ref{mmin}(1),
%the definition of minimal index in \eqref{max} that
 $$\text{min}(\overrightarrow{(d_A)}_{l,t}^{(1)})\geq\tilde{a}_{l-1,t}+2>\tilde{a}_{l-1,t}+1.$$
 Thus,
 by \eqref{Hecke-Cliff}, we have
\begin{equation}\label{eq1}
c_{\tilde{a}_{l-1,t}+1}T_{\overrightarrow{(d_A)}_{l,t}^{(1)}}=T_{\overrightarrow{(d_A)}_{l,t}^{(1)}}c_{\tilde{a}_{l-1,t}+1}.
\end{equation}
For (2), since
$$\overrightarrow{(d_A)}_{l,t}^{(1)}=w_{l,t}^{(2)}\cdots w_{l,t}^{(a_{l,t})}w_{l+1,t}\cdots w_{n,t}
       (w_{2,t+1}w_{3,t+1}\cdots w_{n,t+1})\cdots (w_{2,n-1}w_{3,n-1}\cdots w_{n,n-1}),
       $$
and every nonidentity factor $w_{i,j}$  of $\overrightarrow{(d_A)}_{l,t}^{(1)}$ is a product of nonidentity $w_{i,j}^{(k)}$ ($k\in[1,a_{i,j}]$) satisfying $i>l$, if $j=t$, or $j>t$ and the inequality $\text{min}(w_{i,j}^{(k)})= \widetilde a_{i-1,j}+k\geq \widetilde a_{i-1,j}+1 >\widetilde a_{l-1,t}+1$ by \eqref{bound}. By deleting all the identity $w_{i,j}$ from $\overrightarrow{(d_A)}_{l,t}^{(1)}$, we may assume that
$\overrightarrow{(d_A)}_{l,t}^{(1)}=y_1y_2\cdots y_m$ is a product of those nonidentity $w_{i,j}^{(k)}$. Let $m_1=\text{min}(y_1)$ and $m=\text{min}(m_1, p'+1)$. Then $\text{min}(y_i)\geq \text{min}(y_1)$ for all $i\in[1,m]$  and $m>\widetilde{a}_{l-1,t}+1$.
Moreover, each $y_i$ has the form
\begin{equation}
y_i=s_bs_{b-1} \cdots s_a,\;\text{ where }\;b=\sigma_{i-1,j}+k-1\geq \widetilde a_{i-1,j}+k=a\geq m.
\end{equation}
Let $\mathcal H_{\geq m}$ be the subalgebra of $\HCR$ generated by $T_i, c_j$ with $m\leq i\leq r-1$ and $m\leq j\leq r$. Then
 $c_{p'+1}\in\mathcal H_{\geq m}$ and $T_{\overrightarrow{(d_A)}_{l,t}^{(1)}}\in \mathcal H_{\geq m}$. Hence, $ c_{p'+1}T_{\overrightarrow{(d_A)}_{l,t}^{(1)}}\in \mathcal H_{\geq m}$.
Consequently, by \eqref{homog},
\begin{equation}\label{eq2}
 c_{p'+1}T_{\overrightarrow{(d_A)}_{l,t}^{(1)}}
=T_{\overrightarrow{(d_A)}_{l,t}^{(1)}}c_{\tilde{a}_{h-1,k}+p}+\sum_{d'<\overrightarrow{(d_A)}_{l,t}^{(1)}, u\geq m}f_{d',u}T_{d'}c_u\;\;(f_{d',u}\in R).
 \end{equation}
Combining \eqref{eq0}--\eqref{eq2} gives
\begin{equation}\label{TLformula}
\begin{aligned}
c_{\hahk+p}T_{d_A}
=T_{d_A}c_{\tilde{a}_{h-1,k}+p}+(q-1)T_{d_{A}\backslash s_{\tilde{\lambda}_{h-1}+\overleftarrow{r}^{k}_h+p}}c_{\tilde{a}_{l-1,t}+1}+
\sum_{d< d_{A},\atop u'>\tilde{a}_{l-1,t}+1}g_{d,u'} T_dc_{u'}.
\end{aligned}
\end{equation}
Since the ``tail'' $(q-1)T_{d_{A}/s_{\tilde{\lambda}_{h-1}+\overleftarrow{r}^{k}_h+p}}c_{\tilde{a}_{l-1,t}+1}+
\sum_{d< d_{A},\atop u'>\tilde{a}_{l-1,t}+1}g_{d,u'} T_dc_{u'}$ is nonzero (as $q\neq 1$),
it follows that $A$ does not satisfy the SDP condition at $(h,k)$.
\end{proof}

\begin{cor}\label{shiftonN}
(1) Every $A=(a_{i,j})\in \MN(n)$ satisfies the SDP condition on the 1st column or $n$th row.

(2) If $l\in[1,n]$ and $A=(a_{i,j})\in \MN(n)$ satisfy $a_{i,j}=0$, for $i>l$ and $j<n$, then
$A$ satisfies SDP condition on the $l$-th row.
\end{cor}
%\begin{proof}(1) Suppose $1\leq k\leq n$ and $1\leq p<a_{n,k}$.
%By \eqref{sigmaij}  we observe that $\widetilde{\lambda}_{n-1}+\AK(n,k)+p+1>\sigma_{i-1,j}+a_{i,j}$ for any $2\leq i\leq n$ and $1\leq j\leq k-1$, or $2\leq i\leq n-1$ and $j=k$. Therefore, by \eqref{wij}, we obtain that $c_{\widetilde{\lambda}_{n-1}+\AK(n,k)+p+1}$ commutes with $T_{w_{i,j}}$ for each pair of $i,j$ satisfying $2\leq i\leq n$ and $1\leq j\leq k-1$, or $2\leq i\leq n-1$ and $j=k$. Meanwhile direct calculation show via \eqref{Hecke-Cliff} that the following holds \begin{align*}  c_{\widetilde{\lambda}_{n-1} +  \sum^{k-1}_{u=1}a_{n, u}+ p +1} T_{w_{n,k}} =T_{w_{n,k}} c_{\widetilde{a}_{n-1,k} +  p +1}. \end{align*} Furthermore, observe that $$ \widetilde{a}_{n-1,k} +  p +1<\widetilde{a}_{i-1,j}+1$$for all {$2\leq i\leq n$} and $k+1\leq j\leq n-1$, which implies that $c_{\widetilde{a}_{n-1,k} +  p +1}$ commutes with $T_{w_{i,j}}$ for each pair of $i,j$ satisfying  {$2 \leq i\leq n$} and $k+1\leq j\leq n-1$. Putting together the lemma is proved by \eqref{d_A}. \end{proof}

%%%ZZZ

%\newpage
\section{Multiplication formulas in {$\QqnrR$}: the even case}\label{even case}

We  are now ready to compute, for any $\Ad =  ({\SEE{a}_{i,j}} | {\SOE{a}_{i,j}})  \in\MNZ(n,r)$, the structure constants that appear in $\phi_\Xd\phi_\Ad$, where $\Xd$ is one of the following six matrices:
\begin{equation}\label{even-odd}
\aligned
\text{\bf the even case:}\;\;&(\lambda|O),\;\; (\lambda + E_{h, h+1} -E_{h+1, h+1} |O), \;\;(\lambda-E_{h,h}+E_{h+1,h}|O);\\
\text{\bf the\, odd\, case:}\;\;& (\lambda-E_{h,h}|E_{h,h}),\;\;(\lambda-E_{h+1, h+1}| E_{h, h+1} ),\;\;(\lambda-E_{h,h}|E_{h+1,h}).
\endaligned
\end{equation}
 Here $1\leq h\leq n-1$, and we regard any $\lambda \in \Lambda(n,r)$
 as a diagonal matrix $\lambda={\diag}(\lambda_1,\ldots,\lambda_n)$ in $\MN(n)$
 and write $\lambda+B:={\diag}(\lambda_1,\ldots,\lambda_n)+B$, for $B\in\MN(n)$.

We first deal with the even case in this section.

\begin{conv}\label{CONV}
The terms $\phi_\Md$ appearing in $\phi_\Xd\phi_\Ad$ have the form
$\Md=(\SE{A} \pm E|\SO{A} \pm E')$ for some {$E  \in \MN(n)$},   {$E'  \in M_n(\NN_2)$}. If $\Md\not\in \MNZN(n)$ (i.e., $\Md$ has a negative entry or $\SO{A} \pm E'$ has an entry $>1$), then we set $\phi_\Md$ to be 0. We make a similar convention for the elements $T_\Md$ in \eqref{eq cA} that defines $\phi_\Md$.
\end{conv}

The following notations will be used in the proofs of the multiplication formulas in this and next sections.

\begin{notn}[for the proofs in Sections \ref{even case} and \ref{odd case}]\label{rem_phi_short_0}
Let $\Ad = ({\SEE{a}_{i,j}} | {\SOE{a}_{i,j}})  \in \MNZN(n)$ and $A=\lfloor \Ad\rfloor$ its base matrix. For fixed $1\leq h,k\leq n$, we decompose the elements $c_\Ad, c'_\Ad$ defined in \eqref{eq cA} as a product of three segments:
\begin{align}\label{cA-decomp}
	%&c_A =  {\cbefore} \cdot  c_{\gamma} \cdot {\cafter},
	&c_\Ad =  {\cbefore} \cdot  {\cmiddle} \cdot {\cafter},\qquad c'_\Ad =  ({\cbefore})' \cdot  ({\cmiddle})' \cdot ({\cafter})'
\end{align}
where ${\cmiddle} := (c_{{q}, {\widetilde{a}}_{h-1,k}+1, {\widetilde{a}}_{h,k}})^{\SOE{a}_{h,k}}
		\cdot (c_{{q}, {\widetilde{a}}_{h,k}+1, {\widetilde{a}}_{h+1,k}})^{\SOE{a}_{h+1,k}},$
and
{$ {\cbefore} $} (resp., {$ {\cafter} $}) are the segment of the product before (resp., after) $\cmiddle$. The decomposition for $c'_\Ad$ is similar.

Recall the notations in  \eqref{Ahk} and \eqref{prsum} with $\lambda=\ro(A)$, %(and \eqref{Ahk})
and let {$\AK(h+1,k)=\AK(h+1,k)(A)$},  $\BK(h,k)=\BK(h,k)(A)$. Then
 $$\lambda_{(h)}^+:=\ro(\AP)=\lambda + \bs{\ep}_{h} - \bs{\ep}_{h+1}=\lambda+\alpha_h,\quad
  \lambda^-_{(h)}:=\co(\AM)=\lambda - \bs{\ep}_{h} + \bs{\ep}_{h+1}=\lambda-\alpha_h,$$ and
 {$\co(\AP)  = \co(\AM) = \co(A)$}.
Let %$\nu = \nu_{A} = (  \cdots ,a_{h,k},  a_{h+1,k}, \cdots )$ be as in \eqref{nuA} and let
\begin{equation}\label{nu+nu-}
\aligned
         \nu &= \nu_{A} = (\cdots  \cdots ,a_{h,k},  a_{h+1,k}, \cdots\cdots )\;\;\text{(as in \eqref{nuA}),}\\
	\nu^+ &= \nu_{\AP} = (  \cdots,a_{h,k}+1,  a_{h+1,k}-1, \cdots ), \quad
	\delta^+ = ( \cdots, a_{h,k}, 1, a_{h+1,k}-1, \cdots ),\\
	\nu^- &= \nu_{\AM} = (  \cdots ,  a_{h,k}-1, a_{h+1,k}+1, \cdots ), \quad
	\delta^-= ( \cdots ,a_{h,k}-1, 1, a_{h+1,k} , \cdots ).
	\endaligned
\end{equation}
Here the omitted components are identical. Clearly, $\fS_{\delta^\pm}\subset \fS_{\nu^\pm}$.
If {${\D_{\pm}} $} denote the sets of the shortest representatives  of the left cosets of  {$\fS_{\delta^\pm}$} in {$\fS_{\nu^\pm}$}, i.e.,
$$\D_+=\mathcal D_{\delta^+}^{-1}\cap\fS_{\nu^+},\quad \D_-=\mathcal D_{\delta^-}^{-1}\cap\fS_{\nu^-},$$
and set {$\widetilde a = \widetilde{a}_{h,k}$}, {$a = a_{h,k}$},  $b= a_{h+1,k}$, and $x_{\D_\pm}=\Sigma_{\D_\pm}$, then we have
\begin{equation}\label{subsubgroup}
\begin{aligned}
&	x_{\nu^+} = x_{\D_{+}} x_{\delta^+} = x_{\delta^+} \TDIJ({\widetilde a},  {\widetilde a-a+1}),\\
& x_{\nu^-} = x_{\D_{-}} x_{\delta^-} =x_{\delta^-} \TAIJ({\widetilde a},  {\widetilde a+b-1}),
\end{aligned}
\qquad x_{\nu} =  x_{\delta^+} \TAIJ( {\widetilde a+1}, {\widetilde a+b-1})
	=  x_{\delta^-} \TDIJ({\widetilde a-1},  {\widetilde a-a +1}).
\end{equation}
Let $'\!\nu^+ = \nu_{^t\!\AP}$ and $'\!\nu^- = \nu_{^t\!\AM}$ be defined as in \eqref{nuA} by using the transposes $^t\!\AP$, $^t\!\AM$ of $\AP$, $\AM$, respectively.
Then $\fS_{'\!\nu^\pm}\subseteq \fS_{\lambda_{(h)}^\pm}$ and we have (cf. \cite[Corollary 3.5]{DW2})
\begin{equation}\label{eq_ta+}
\aligned
	x_{{{\lambda}_{(h)}^{+}}} {T_{d_{\AP}}}
	&=  x_{{\lambda}_{(h)}^{+}\backslash{}'\!\nu^+}  {T_{d_{\AP}}}  x_{\nu^+}, \quad\text{where}\quad x_{{{\lambda}_{(h)}^{+}}\backslash{}'\! \nu^+}  := \sum_{u' \in {\D}^{-1}_{{}'\!\nu^+} \cap \fS_{{{\lambda}_{(h)}^{+}}}} T_{u'};\\
x_{{{\lambda}_{(h)}^{-}}} {\TDAM}
	&= x_{{{\lambda}_{(h)}^{-}}\backslash{}'\!\nu^-}  {\TDAM}  x_{\nu^-}, \quad\text{where}\quad
	x_{{{\lambda}_{(h)}^{-}}\backslash{}'\!\nu^-} := \sum_{u' \in {\D}^{-1}_{{}'\!\nu^-} \cap \fS_{{{\lambda}_{(h)}^{-}}}} T_{u'}.
\endaligned
\end{equation}

\end{notn}
The following displays the multiplication formulas for $\phi_\Xd\phi_\Ad$ in the queer $q$-Schur superalgebra $\QqnrR$, where  $X$ is one of the following matrices
\begin{equation}\label{EVEN}
\Dd_\mu:=(\mu|O),\;\;
 \Ed_{h,\lambda}:=(\lambda + E_{h, h+1}-E_{h+1, h+1}|O),\;\; \Fd_{h,\lambda}:=(\lambda-E_{h,h} + E_{h+1,h} |O).
 \end{equation}

 %Recall the notation in \eqref{stepd}.

\begin{thm}\label{phiupper0}
Let {$h \in [1,n-1]$} and {$\Ad =  ({\SEE{a}_{i,j}} | {\SOE{a}_{i,j}})   \in \MNZ(n,r)$}. Assume $A=(a_{i,j})$ is the base of $\Ad$ with {$a_{i,j}=\SEE{a}_{i,j} + \SOE{a}_{i,j}$},
  and
	 $\BK(h,k)=\BK(h,k)(A)$. Then, for $\lambda,\mu\in\Lambda(n,r)$ and $\varepsilon=\delta_{\lambda,\ro(A)} $, the following multiplication formulas hold in $\QqnrR$:
\begin{align*}
(1)\;\;& \phi_{\Dd_\mu} \phi_\Ad = \delta_{\mu,\ro(A)}\phi_{\Ad},\qquad
	 \phi_\Ad \phi_{\Dd_\mu}  =\delta_{\mu,\co(A)}\phi_{\Ad}.\\
(2)\;\;& \phi_{\Ed_{h,\lambda}}  \phi_\Ad = \varepsilon\sum_{k=1}^n \Big\{
	{q}^{\BK(h,k) + \SOE{a}_{h+1,k}}  \STEP{ \SEE{a}_{h,k} + 1}
		\phi_{(\SE{A} + E_{h,k}  - E_{h+1, k} | \SO{A} )}
	+
	% {\delta}_{1,\SOE{a}_{h+1, k}} {\delta}_{0,\SOE{a}_{h,k}}
	{q}^{\BK(h,k)}  \phi_{(\SE{A}  | \SO{A} + E_{h,k} - E_{h+1, k}  )}  \\
	&\hspace{6cm} +
	% {\delta}_{1,\SOE{a}_{h,k}} {\delta}_{1,\SOE{a}_{h+1,k}}
	{q}^{\BK(h,k)-1} \STEPPD{a_{h,k}+1}
	\phi_{(\SE{A} + 2E_{h,k} | \SO{A}  -E_{h,k} - E_{h+1,k} )}\Big
\}.\\
(3)\;\;& \phi_{\Fd_{h,\lambda}}  \phi_\Ad = \varepsilon\sum_{k=1}^n \Big\{
	 {q}^{\AK(h+1,k) }  \STEP{ \SEE{a}_{h+1, k} +1}
	\phi_{(\SE{A} - E_{h,k} + E_{h+1, k} | \SO{A} )}
	 +
	% {\delta}_{1,\SOE{a}_{h,k}} {\delta}_{0,\SOE{a}_{h+1,k}}
	{q}^{\AK(h+1,k)  + a_{h, k} - 1} \phi_{(\SE{A}  |\SO{A} - E_{h,k} + E_{h+1,k})} \\
	& \hspace{4.5cm} +
	% {\delta}_{1,\SOE{a}_{h,k}} {\delta}_{1,\SOE{a}_{h+1,k}}
	{q}^{\AK(h+1,k)  + a_{h, k} -2} \STEPPD{a_{h+1, k} +1}
	\phi_{(\SE{A} + 2E_{h+1,k} | \SO{A}  - E_{h,k} - E_{h+1,k})}
\Big\}.
\end{align*}
\end{thm}

\begin{proof} Formula (1) is a special case of \eqref{co=ro}, noting that $\phi_{(\lambda|O)}$ is the identity map on $x_\lambda\HCR$.
We now prove (2).

Let {$\Xd=\Ed_{h,\lambda}= (\lambda + E_{h, h+1} -E_{h+1, h+1} |O)$} whose base matrix $X=\lambda + E_{h, h+1} -E_{h+1, h+1} $.
Then {$\ro(X) =\lambda+ \alpha_h = {\lambda}_{(h)}^{+}$}, $\co(X)=\lambda$,
{$d_X = 1$}, {$T_{d_X} = 1$},  and {$c_\Xd  = 1$}.
Further, we have $\fS_{\nu_X}=\fS_{\lambda_{(h)}^+}\cap\fS_\lambda$ and
{$
	 \D_{{\nu}_{X}} \cap \fS_{{\lambda}}   = \{ 1, s_{{\widetilde{\lambda}_{h}+1}},
		s_{{\widetilde{\lambda}_{h}+1}} s_{{\widetilde{\lambda}_{h}+2}},
		\cdots,
		s_{{\widetilde{\lambda}_{h}+1}} s_{{\widetilde{\lambda}_{h}+2}} \cdots s_{{\widetilde{\lambda}_{h+1}-1}}
	\}
$}. Thus,
$$T_\Xd=x_{{{\lambda}_{(h)}^{+}}}\Sigma_{\D_{{\nu}_{X}} \cap \fS_{\lambda}}=x_{{{\lambda}_{(h)}^{+}}}\TAIJ( {\widetilde{\lambda}_{h}+1}, {\widetilde{\lambda}_{h+1}-1})=\sum_{{1\leq k\leq n}\atop {a_{h+1,k}>0}}x_{{{\lambda}_{(h)}^{+}}}\cdot
\sum_{j=\AK(h+1,k)}^{\AK(h+1,k+1)-1}
		T_{\widetilde{\lambda}_{h}+1} T_{\widetilde{\lambda}_{h}+2} \cdots  T_{\widetilde{\lambda}_{h}+j}
		.$$
Hence,  for $\lambda=\ro(A)$ and $\mu=\co(A)$, \eqref{eq_gB} with $\Bd=\Xd$ becomes
$$
z	:= x_{{{\lambda}_{(h)}^{+}}} \TAIJ( {\widetilde{\lambda}_{h}+1}, {\widetilde{\lambda}_{h+1}-1}) {T_{d_A}} c_{\Ad}
		\sum _{{\sigma} \in \D_{{\nu}_{_A}} \cap \fS_{\mu}} T_{\sigma}
	 = 	\sum_{{1\leq k\leq n}\atop {a_{h+1,k}>0}} z_k,
	 %\;\text{ with }\delta^<_{0,a_{h+1,k}}=\begin{cases}1,\text{ if }0<a_{h+1,k};\\0,\text{ otherwise,}
	% \end{cases}
$$
where, for each {$k \in [1, n]$} with $a_{h+1,k}>0$ (and $\Sigma_A:=\sum _{{\sigma} \in \D_{{\nu}_{_A}} \cap \fS_{\mu}} T_{\sigma}$),
\begin{equation}\label{z_k}
\aligned
	z_k
	     &=x_{{{\lambda}_{(h)}^{+}}}\Big(\sum_{j=\AK(h+1,k)}^{\AK(h+1,k+1)-1}
		T_{\widetilde{\lambda}_{h}+1} T_{\widetilde{\lambda}_{h}+2} \cdots  T_{\widetilde{\lambda}_{h}+j}
		{T_{d_A}} \Big) c_{\Ad}\Sigma_A.
		%\sum _{{\sigma} \in \D_{{\nu}_{_A}} \cap \fS_{\mu}} T_{\sigma}.
\endaligned
\end{equation}
Then, with {$\widetilde a= \widetilde{a}_{h,k}$}, {$a = a_{h,k}$},  and $b = a_{h+1,k}$, %and $\nu=\nu_A$,
\eqref{z_k} becomes, by Lemma  \ref{prop_pjshift},
\begin{align*}
	z_k
&=
		%\sum_{j=\AK(h+1,k)}^{\AK(h+1,k+1) - 1}
		x_{{{\lambda}_{(h)}^{+}}}
		T_{\widetilde{\lambda}_{h}} T_{\widetilde{\lambda}_{h}-1}\cdots T_{\widetilde{\lambda}_{h}-\BK(h,k)+1}
		{{\TDAP}}\TAIJ( {\widetilde a+1}, {\widetilde a+b-1})
		%T_{\widetilde{a}_{h,k}+1} \cdots T_{\widetilde{a}_{h,k}+p_j}
		c_{\Ad}\Sigma_A
		%\sum _{{\sigma} \in \D_{{\nu}} \cap \fS_{\mu}} T_{\sigma}
		\\
&=
	{q}^{\BK(h,k)} x_{{{\lambda}_{(h)}^{+}}}  {{\TDAP}}
	\TAIJ( {\widetilde a+1}, {\widetilde a+b-1})	c_{\Ad}\Sigma_A
	%\sum _{{\sigma} \in \D_{{\nu}} \cap \fS_{\mu}} T_{\sigma}
	\quad ( \mbox{since } s_{\widetilde{\lambda}_{h}}, s_{\widetilde{\lambda}_{h}-1}, \cdots,
		s_{\widetilde{\lambda}_{h}-\BK(h,k)+1} \in {\fS}_{{{\lambda}_{(h)}^{+}}})
	\\
&=
	{q}^{\BK(h,k)} x_{{{\lambda}_{(h)}^{+}}\backslash{}'\!\nu^+} {{\TDAP}}
	x_{{\D_{+}}} x_{\delta^+} \TAIJ( {\widetilde a+1}, {\widetilde a+b-1}) c_{\Ad}\Sigma_A
	%\sum _{{\sigma} \in \D_{{\nu}} \cap \fS_{\mu}} T_{\sigma}
	\quad
 {
	\mbox{($x_{\nu^+}=x_{{\D_{+}}} x_{\delta^+}$ by \eqref{eq_ta+}  and \eqref{subsubgroup})}
 }
	\\
&=
	{q}^{\BK(h,k)} x_{{{\lambda}_{(h)}^{+}}\backslash{}'\!\nu^+} {{\TDAP}}
	x_{{\D_{+}}}  (c_{\Ad})' x_{\nu}\Sigma_A
	%\sum _{{\sigma} \in \D_{{\nu}} \cap \fS_{\mu}} T_{\sigma}
	\quad
 {
	\mbox{($x_{\delta^+} \TAIJ( {\widetilde a+1}, {\widetilde a+b-1})  = x_{\nu}$ by  \eqref{subsubgroup} and by \eqref{xc})}
 }.
\end{align*}
By rewriting $x_{\nu}=x_{\delta^+} \TAIJ( {\widetilde a+1}, {\widetilde a+b-1})$ again and applying Lemma  \ref{Tnsum}(1),
we have
 \begin{equation}\label{zkupper}
	\begin{aligned}
&z_k=
	{q}^{\BK(h,k)} x_{{{\lambda}_{(h)}^{+}}\backslash{}'\!\nu^+} {{\TDAP}}
	x_{{\D_{+}}}   (c_{\Ad})'  x_{\delta^+}  \TDIJ({\widetilde a},  {\widetilde a-a+1})\Sigma_{A^+_{h,k}},\;\;
	\text{where }\Sigma_{A^+_{h,k}}=\sum _{{\sigma} \in \D_{\nu^+} \cap \fS_{\mu}} T_{\sigma},
	\\
&=
	{q}^{\BK(h,k)} x_{{{\lambda}_{(h)}^{+}}\backslash{}'\!\nu^+} {{\TDAP}}
	x_{{\D_{+}}}   ({\cbefore})' ({\cmiddle})' ({\cafter})' x_{\delta^+}  \TDIJ({\widetilde a},  {\widetilde a-a+1})\Sigma_{A^+_{h,k}}\\
	%\sum _{{\sigma} \in \D_{\nu^+} \cap \fS_{\mu}} T_{\sigma} \\
&=
	{(-1)}^{(\SOE{a}_{h,k} + \SOE{a}_{h+1,k}) \YK } {q}^{\BK(h,k)} x_{{{\lambda}_{(h)}^{+}}\backslash{}'\!\nu^+} {{\TDAP}}\cdot
	\underbrace{x_{{\D_{+}}}  ({\cmiddle})'  x_{\delta^+}  \TDIJ({\widetilde a},  {\widetilde a-a+1})}_{=\fkm^+}\cdot {\cbefore} {\cafter}\Sigma_{A^+_{h,k}}.
	%\sum _{{\sigma} \in \D_{\nu^+} \cap \fS_{\mu}} T_{\sigma} .
\end{aligned}
\end{equation}
Here, for the last equality, we used the following relations:
\begin{enumerate}[(a)]
\item
$({\cbefore})' ({\cmiddle})' ({\cafter})' ={(-1)}^{(\SOE{a}_{h,k} + \SOE{a}_{h+1,k}) \YK }({\cmiddle})' ({\cbefore})' ({\cafter})' $ (noting that $\YK$ is the partial sum at the $(h-1,k)$ position associated with $\nu_{A^{\bar 1}}$);
\item
$({\cbefore})' ({\cafter})' =(c_{\delta^+}^\alpha)'$ for some $\alpha\leq\delta^+$ with $\alpha_i=0,\;i\in[\widetilde a_{h-1,k}+1,\widetilde a_{h+1,k}]$,
and hence  $({\cbefore})' ({\cafter})' x_{\delta^+}=x_{\delta^+}({\cbefore})({\cafter}) $ by \eqref{xc}; %and, finally,
\item
$({\cbefore})({\cafter})$ commutes with $\TDIJ({\widetilde a},  {\widetilde a-a+1})$ since $T_ic_j=c_jT_i$ for all $j\neq i,i+1$.
\end{enumerate}

To complete the computation of $z_k$, we need to manipulate the ``middle part'' $\fkm^+$. This requires to consider the following four cases.
We record the following fact
\begin{equation}\label{3.3}
x_{\nu^+}  \TDIJ({\widetilde a},  {\widetilde a-a+1})=\STEP{a+1}x_{\nu^+},\;\;\text{since }s_{\widetilde a},s_{\widetilde a-1},\ldots,s_{\widetilde a-a+1}\in\fS_{\nu^+}.
\end{equation}

The first case is simple.

\medskip
\noindent
 {\bf Case 1 ({$\SOE{a}_{h,k} = 0$}, {$\SOE{a}_{h+1,k} = 0$}).}
In this case,  $(\cmiddle)' = 1$, $c_{\Ad} ={\cbefore} {\cafter} =c_{A^+_{h,k}}$ and
\begin{align*}
z_k
&=
	{q}^{\BK(h,k)} x_{{{\lambda}_{(h)}^{+}}\backslash{}'\!\nu^+} {{\TDAP}}
	x_{\nu^+}  \TDIJ({\widetilde a},  {\widetilde a-a+1}) {\cbefore} {\cafter}\Sigma_{A^+_{h,k}}
	%\sum _{\sigma\in \D_{\nu^+ \cap \fS_\mu}} T_{\sigma}
	\qquad \mbox{ (as $x_{{\D_{+}}}  x_{\delta^+}  = x_{\nu^+} $) }
	\\
&=
	{q}^{\BK(h,k)} \STEP{a+1} x_{{{\lambda}_{(h)}^{+}}} {{\TDAP}}   {\cbefore} {\cafter}\Sigma_{A^+_{h,k}}
	 %\sum _{{\sigma} \in \D_{\nu^+ } \cap \fS_{\mu}}  T_{\sigma}
	%\qquad  \mbox{(by Lemma \ref{xTinverse})}
	\\
&=
	{q}^{\BK(h,k)} \STEP{ \SEE{a}_{h,k} +1} T_{(\SE{A} + E_{h,k} - E_{h+1, k} | \SO{A})} .
\end{align*}

\medskip
\noindent
 {\bf Case 2 ({$\SOE{a}_{h,k} = 0$}, {$\SOE{a}_{h+1,k} = 1$}).}
Then $(\cmiddle)' = c'_{{q}, \widetilde a +1, \widetilde a +b}$ and
\begin{align*}
z_k &=
	{(-1)}^{\YK } {q}^{\BK(h,k)} x_{{{\lambda}_{(h)}^{+}}\backslash{}' \nu^+} {{\TDAP}}\cdot
	\underbrace{x_{{\D_{+}}}  (c_{{q}, \widetilde a +1, \widetilde a +b})'  x_{\delta^+}  \TDIJ({\widetilde a},  {\widetilde a-a+1})}_{=\fkm^+} \cdot{\cbefore} {\cafter}\Sigma_{A^+_{h,k}}.
	%\sum _{{\sigma} \in \D_{\nu^+} \cap \fS_{\mu}} T_{\sigma} .
\end{align*}
We now rewrite the middle part $\fkm^+$ using \eqref{subsubgroup}, \eqref{xc}, and \eqref{cqij}, etc. as follows:
\begin{align*}
	%x_{{\D_{+}}} (c_{{q}, \widetilde a +1, \widetilde a +b})'  &x_{\delta^+}  \TDIJ({\widetilde a},  {\widetilde a-a+1})
	\fkm^+&= x_{{\D_{+}}} ({q} (c_{{q}, \widetilde a +2, \widetilde a +b})' +  c_{\widetilde a+1}) x_{\delta^+}  \TDIJ({\widetilde a},  {\widetilde a-a+1})\quad (\text{by \eqref{cqij}})\\
	&= {q} x_{{\D_{+}}} x_{\delta^+} c_{{q}, \widetilde a +2, \widetilde a +b}  \TDIJ({\widetilde a},  {\widetilde a-a+1}) +  x_{{\D_{+}}} x_{\delta^+}c_{\widetilde a+1} \TDIJ({\widetilde a},  {\widetilde a-a+1})\quad(\text{by \eqref{xc}})  \\
	&= {q}  (c_{{q}, \widetilde a +2, \widetilde a +b})'x_{\nu^+}  \TDIJ({\widetilde a},  {\widetilde a-a+1}) +  x_{\nu^+}c_{{q}, \widetilde a-a+1, \widetilde a +1}\quad(\text{by \eqref{subsubgroup}, Lemma  \ref{xcT}(3)})\\
	%x_{{\D_{+}}} x_{\delta^+}c_{\widetilde a+1} \TDIJ({\widetilde a},  {\widetilde a-a+1})  \\
	&=
		{q} \STEP{a +1}  x_{\nu^+} c_{{q}, \widetilde a +2, \widetilde a +b}  +  x_{\nu^+}c_{{q}, \widetilde a-a+1, \widetilde a +1} \quad (\text{by \eqref{3.3}}).
\end{align*}
Notice that, if $a^{\bar 0}_{h+1,k}=0$ (thus, $b=1$), then the term $(c_{{q}, \widetilde a +2, \widetilde a +b})' =0$ (hence, $c_{{q}, \widetilde a +2, \widetilde a +b}=0$).
Hence, by noting ${\cbefore} c_{{q}, \widetilde a+2, \widetilde a +b}  {\cafter}=c_{\nu^+}^\alpha$ with $\alpha=\nu_{A^{\bar1}}$ and ${\cbefore} c_{{q}, \widetilde a-a+1, \widetilde a +1} {\cafter}=c_{\nu^+}^\beta$ with $\beta=\nu_{\SO{A} - E_{h+1, k}  + E_{h,k}}$,
\begin{align*}
z_k
&=
	{q}^{\BK(h,k)+1} \STEP{a+1}  x_{{{\lambda}_{(h)}^{+}}} {{\TDAP}}
	{\cbefore} c_{{q}, \widetilde a +2, \widetilde a +b}  {\cafter}\Sigma_{A^+_{h,k}}\\
	%\sum _{{\sigma} \in \D_{{\nu^+}} \cap \fS_{\mu}} T_{\sigma}  \\
	&\qquad + {q}^{\BK(h,k)} x_{{{\lambda}_{(h)}^{+}}} {{\TDAP}}
	{\cbefore} c_{{q}, \widetilde a-a+1, \widetilde a +1} {\cafter}\Sigma_{A^+_{h,k}}\\
	%\sum _{{\sigma} \in \D_{{\nu^+}} \cap \fS_{\mu}} T_{\sigma}  \\
&=
	{q}^{\BK(h,k)+1} \STEP{ \SEE{a}_{h,k} +1}  T_{(\SE{A} + E_{h,k} - E_{h+1, k} | \SO{A}) }
	+ {q}^{\BK(h,k)} T_{(\SE{A}  | \SO{A}  + E_{h,k} - E_{h+1, k} ) }.
\end{align*}
Here the last step follows from the facts that $(\SE{A} + E_{h,k} - E_{h+1, k} | \SO{A})$ and $(\SE{A}  | \SO{A}  + E_{h,k} - E_{h+1, k} )$ have the same base $A^+_{h,k}$, and
 the sign in $z_k$ is cancelled after swapping $ c_{{q}, \widetilde a +2, \widetilde a +b}$ (resp., $c_{{q}, \widetilde a-a+1, \widetilde a +1}$) with $\cbefore$.
(Note that,
by Convention \ref{CONV}, $T_{(\SE{A} + E_{h,k} - E_{h+1, k} | \SO{A}) }=0$ if $a^{\bar 0}_{h+1,k}=0$.)

\medskip
\noindent
 {\bf Case 3 ({$\SOE{a}_{h,k} = 1$}, {$\SOE{a}_{h+1,k} = 0$}).} Then we have {${\cmiddle} = c_{{q}, \widetilde a-a+1, \widetilde a}$} and
\begin{align*}
z_k &=
	{(-1)}^{\YK } {q}^{\BK(h,k)} x_{{{\lambda}_{(h)}^{+}}\backslash{}'\nu^+} {{\TDAP}}\cdot\underbrace{
	x_{{\D_{+}}}  (c_{{q}, \widetilde a-a+1, \widetilde a})'  x_{\delta^+}  \TDIJ({\widetilde a},  {\widetilde a-a+1})}_{=\fkm^+} \cdot{\cbefore} {\cafter}\Sigma_{A^+_{h,k}}.\\
	%\sum _{{\sigma} \in \D_{{\nu^+}} \cap \fS_{\mu}} T_{\sigma}.
\end{align*}
For the middle part $\fkm^+$, again using \eqref{subsubgroup}, \eqref{xc}, and \eqref{cqij},  one can deduce
\begin{align*}
 %{q} x_{{\D_{+}}}  (c_{{q}, \widetilde a-a+1, \widetilde a})' & x_{\delta^+}  \TDIJ({\widetilde a},  {\widetilde a-a+1})
 q\fkm^+&
	=  x_{{\D_{+}}}  x_{\delta^+} \cdot {q} c_{{q}, \widetilde a-a+1, \widetilde a} \cdot  \TDIJ({\widetilde a},  {\widetilde a-a+1}) \\
	&=  x_{\nu^+}  c_{{q}, \widetilde a-a+1, \widetilde a +1} \TDIJ({\widetilde a},  {\widetilde a-a+1}) -x_{\nu^+}  c_{\widetilde a+1} \TDIJ({\widetilde a},  {\widetilde a-a+1}) \\
	&=   (c_{{q}, \widetilde a-a+1, \widetilde a +1})'  x_{\nu^+} \TDIJ({\widetilde a},  {\widetilde a-a+1}) -x_{\nu^+}  c_{{q}, \widetilde a-a+1, \widetilde a +1} \;(\text{by Lemma \ref{xcT}(3)})\\
	&= ( \STEP{a+1} -1 )x_{\nu^+} c_{{q}, \widetilde a-a+1, \widetilde a +1}\;\;(\text{by \eqref{3.3}}).
\end{align*}
Substituing leads to
\begin{align*}
z_k
&=
	{q}^{\BK(h,k)-1} ( \STEP{a+1} -1 ) x_{{{\lambda}_{(h)}^{+}}} {{\TDAP}}
	{\cbefore}c_{{q}, \widetilde a-a+1, \widetilde a +1} {\cafter}\Sigma_{A^+_{h,k}}\\
	%\sum _{{\sigma} \in \D_{{\nu^+}} \cap \fS_{\mu}} T_{\sigma} \\
&=
	{q}^{\BK(h,k) }   \STEP{ \SEE{a}_{h,k} +1}
	T_{(\SE{A}+ E_{h,k}  -E_{h+1, k}  |\SO{A})} .
\end{align*}

\noindent
 {\bf Case 4 ({$\SOE{a}_{h,k} = 1$}, {$\SOE{a}_{h+1,k} = 1$}).}
Then {$(\cmiddle)' = c'_{{q},  \widetilde a-a+1,  \widetilde a} c'_{{q},  \widetilde a +1,  \widetilde a +b}$} and
\begin{align*}
z_k &=
	{q}^{\BK(h,k)} x_{{{\lambda}_{(h)}^{+}}\backslash{}'\nu^+} {{\TDAP}}\cdot\underbrace{
	x_{{\D_{+}}}  c'_{{q},  \widetilde a-a+1,  \widetilde a} c'_{{q},  \widetilde a +1,  \widetilde a +b}  x_{\delta^+}  \TDIJ({ \widetilde a},  { \widetilde a-a+1})}_{=\fkm^+}\cdot {\cbefore} {\cafter}\Sigma_{A^+_{h,k}}.
	%\sum _{{\sigma} \in \D_{{\nu^+}} \cap \fS_{\mu}} T_{\sigma}.
\end{align*}
This time, we handle the middle part $\fkm^+$ in a more tricky way: first swapping the two $c'$ factors, then commuting one with $x_{\delta^+}$, then separating $c_{\widetilde a+1}$ from the $c'$ factor and adding the tail $c_{\widetilde a+1}$ to the $q$ times the $c$-factor,
and finally expanding the product. In other words, we have
\begin{equation}\label{c'c'}
\aligned
 q c'_{{q}, \widetilde a-a+1, \widetilde a} c'_{{q}, \widetilde a +1, \widetilde a +b}  x_{\delta^+}&=-qc'_{{q}, \widetilde a +1, \widetilde a +b}c'_{{q}, \widetilde a-a+1, \widetilde a}  x_{\delta^+}=-c'_{{q}, \widetilde a +1, \widetilde a +b}x_{\delta^+}(qc_{{q}, \widetilde a-a+1, \widetilde a} )   \\
&= -  (c_{\widetilde a +1} +qc'_{{q}, \widetilde a +2, \widetilde a +b}) x_{\delta^+} (c_{{q}, \widetilde a-a+1, \widetilde a+1} -c_{\widetilde a+1}) \\
&= - x_{\delta^+} (qc_{{q}, \widetilde a +2, \widetilde a +b}+c_{\widetilde a +1}) (c_{{q}, \widetilde a-a+1, \widetilde a+1} -c_{\widetilde a+1}) \\
=	- x_{\delta^+} (qc_{{q}, \widetilde a +2, \widetilde a +b}&c_{{q}, \widetilde a-a+1, \widetilde a +1}
	- {q}  c_{{q}, \widetilde a +2, \widetilde a +b}  c_{\widetilde a+1}
	+ c_{\widetilde a+1} c_{{q}, \widetilde a-a+1, \widetilde a +1}
	 -  c_{\widetilde a+1} c_{\widetilde a+1} ).
\endaligned
\end{equation}

Now, substituting into $\fkm^+$, noting  $c_{\widetilde a +1}c_{{q}, \widetilde a-a+1, \widetilde a +1}=-  (c_{{q}, \widetilde a-a+1, \widetilde a +1} c_{\widetilde a +1}  + 2)$, $x_{\nu^+}   \TDIJ({\widetilde a},  {\widetilde a-a+1})=\STEP{a+1} x_{\nu^+} $, and repeatedly applying  \eqref{subsubgroup}, \eqref{xc} %\eqref{cqij}, Lemma \ref{xTinverse} and Lemma \ref{xcT}(3)
yield
\begin{align*}
 q\fkm^+
&=	- {q}  x_{\nu^+}  c_{{q}, \widetilde a +2, \widetilde a +b}c_{{q}, \widetilde a-a+1, \widetilde a +1} \TDIJ({\widetilde a },  {\widetilde a-a+1})
	+ {q}  x_{\nu^+}  c_{{q}, \widetilde a +2, \widetilde a +b}  c_{\widetilde a+1}  \TDIJ({\widetilde a },  {\widetilde a-a+1})  \\
	&\qquad -  x_{\nu^+} c_{\widetilde a+1} c_{{q}, \widetilde a-a+1, \widetilde a +1} \TDIJ({\widetilde a},  {\widetilde a-a+1})
	 +  x_{\nu^+}c_{\widetilde a+1} c_{\widetilde a+1}  \TDIJ({\widetilde a },  {\widetilde a-a+1})  \\
&=	- {q}  (c_{{q}, \widetilde a +2, \widetilde a +b})' (c_{{q}, \widetilde a-a+1, \widetilde a +1})' x_{\nu^+}   \TDIJ({\widetilde a},  {\widetilde a-a+1})
	+ {q} (c_{{q}, \widetilde a +2, \widetilde a +b})'  x_{\nu^+} c_{\widetilde a+1}  \TDIJ({\widetilde a},  {\widetilde a-a+1})  \\
	&\qquad + x_{\nu^+} (c_{{q}, \widetilde a-a+1, \widetilde a +1} c_{\widetilde a +1}  + 2) \TDIJ({\widetilde a},  {\widetilde a-a+1})
	 - x_{\nu^+}  \TDIJ({\widetilde a},  {\widetilde a-a+1})  \\
&=	- {q}  \STEP{a+1} x_{\nu^+}  c_{{q}, \widetilde a +2, \widetilde a +b} c_{{q}, \widetilde a-a+1, \widetilde a +1}
	+ {q} x_{\nu^+} c_{{q}, \widetilde a +2, \widetilde a +b}  c_{{q}, \widetilde a-a+1, \widetilde a +1}   \\
	&\qquad + (c_{{q}, \widetilde a-a+1, \widetilde a +1} )' x_{\nu^+} c_{\widetilde a +1}  \TDIJ({\widetilde a},  {\widetilde a-a+1})
	+ \STEP{a+1} x_{\nu^+} \;\;(\text{2nd term by Lemma } \ref{xcT}(3))  \\
&=	{q}( \STEP{a+1} - 1 ) x_{\nu^+}  c_{{q}, \widetilde a-a+1, \widetilde a +1} c_{{q}, \widetilde a +2, \widetilde a+b} + (\STEP{a+1} - {\STEPP{a+1}} ) x_{\nu^+},
\end{align*}
where the third term $=x_{\nu^+}(c_{{q}, \widetilde a-a+1, \widetilde a +1} )^2$  by  Lemma \ref{xcT}(3) and then apply Lemma \ref{xcT}(4).
Since ${q}( \STEP{a+1} - 1 )=q^2\STEP{a}$, we have then, in this case,
\begin{align*}
z_k
&=
	{q}^{\BK(h,k) + 1}  \STEP{a}
	x_{{{\lambda}_{(h)}^{+}}} {{\TDAP}}
	{\cbefore} c_{{q}, \widetilde a-a+1, \widetilde a +1} c_{{q}, \widetilde a +2, \widetilde a +b}  {\cafter}\Sigma_{A^+_{h,k}}\\
	%\sum _{{\sigma} \in \D_{{\nu^+}} \cap \fS_{\mu}} T_{\sigma} \\
	&\qquad + {q}^{\BK(h,k)-1} (\STEP{a+1} - {\STEPP{a+1}} )
	x_{{{\lambda}_{(h)}^{+}}} {{\TDAP}} {\cbefore} {\cafter}\Sigma_{A^+_{h,k}}\\
	%\sum _{{\sigma} \in \D_{{\nu^+}} \cap \fS_{\mu}} T_{\sigma} \\
&=
	{q}^{\BK(h,k) + 1}   \STEP{ \SEE{a}_{h,k}+1}
	T_{(\SE{A} + E_{h,k} - E_{h+1,k} | \SO{A})}
	  + {q}^{\BK(h,k)-1} \STEPPD{a_{h,k}+1}
	T_{( \SE{A} + 2E_{h,k} | \SO{A}  -E_{h,k} - E_{h+1,k} )} .
\end{align*}
%We now combine the four cases above into a single expression.
Let
\begin{align*}
z'_k&=
	{q}^{\BK(h,k) + \SOE{a}_{h+1,k}}  \STEP{ \SEE{a}_{h,k} + 1}
		T_{(\SE{A}  + E_{h,k} - E_{h+1, k}| \SO{A} )}
	+
	% {\delta}_{1,\SOE{a}_{h+1, k}} {\delta}_{0,\SOE{a}_{h,k}}
	{q}^{\BK(h,k)} T_{(\SE{A}  | \SO{A} + E_{h,k}- E_{h+1, k}   )}  \\
	&\hspace{5cm} +
	% {\delta}_{1,\SOE{a}_{h,k}} {\delta}_{1,\SOE{a}_{h+1,k}}
	{q}^{\BK(h,k)-1} \STEPPD{a_{h,k}+1}
	T_{(\SE{A} + 2E_{h,k} | \SO{A}  -E_{h,k} - E_{h+1,k} )}.
\end{align*}
Then, by the notational Convention \ref{CONV}, $z'_k=0$ if $a_{h+1,k}=0$. A case-by-case argument following the four cases above shows that $z'_k=z_k$ if $a_{h+1,k}>0$. This proves (2).

Finally,  we prove the formula in (3).
Observe that,
 for {$\Xd=\Fd_{h,\la}=( \lambda -E_{h, h} + E_{h+1, h}| O)$} with $\lambda=\ro(A)$,
 we have %$X=\lfloor\Xd\rfloor=\diag(\lambda) -E_{h, h} + E_{h+1, h}$ and
 \begin{equation}\label{Fbar}
 \aligned
 X&=\lfloor\Xd\rfloor=\diag(\lambda) -E_{h, h} + E_{h+1, h},\;\text{ and }\\
 \mathcal{D}_{\nu_X}\cap \mathfrak{S}_\lambda
&= \{ 1,  s_{\widetilde{\lambda}_h - 1}, s_{\widetilde{\lambda}_h - 1}s_{\widetilde{\lambda}_h-2},
	\cdots,
	s_{\widetilde{\lambda}_h-1}s_{\widetilde{\lambda}_h-2}\cdots s_{\widetilde{\lambda}_{h-1}+1}\}.
	\endaligned
\end{equation}
By
replacing the ($+$) notations by the ($-$) notations in various places in the proof of (2),
 \eqref{eq_gB} with $\Bd=\Xd$ becomes

\begin{align*}
z
&=
	x_{{{\lambda}_{(h)}^{-}}}
		 \TDIJ( {\tilde{\lambda}_{h}-1}, {\tilde{\lambda}_{h-1}+1}) {T_{d_A}} c_{\Ad}\Sigma_A
		%\sum _{{\sigma} \in \D_{{\nu}_{_A}} \cap \fS_{\mu}} T_{\sigma}
 =  \sum_{{1\leq k\leq n}\atop {a_{h,k}>0}}  z_k,
\end{align*}
where, for each {$k \in [1,n]$} with {$a_{h,k}>0$} and by Lemma  \ref{prop_pjshift}(2),
\begin{equation}\label{6.4.10}
\aligned
z_k&=
		\sum_{j=\BK(h,k)}^{\BK(h,k-1) - 1}
		x_{{{\lambda}_{(h)}^{-}}}
		T_{\tilde{\lambda}_{h}-1} T_{\tilde{\lambda}_{h}-2} \cdots T_{\tilde{\lambda}_{h}-j}
		{T_{d_A}} c_{\Ad}\Sigma_A\\
	&=x_{{{\lambda}_{(h)}^{-}}}
	T_{\tilde{\lambda}_{h}} T_{\tilde{\lambda}_{h}+1}
	\cdots T_{\tilde{\lambda}_{h}+\AK(h+1,k)-1}   \TDAM \TDIJ({\widetilde a-1},  {\widetilde a-a +1})
c_{\Ad}\Sigma_A\\
&={q}^{\AK(h+1,k)} x_{{{\lambda}_{(h)}^{-}}\backslash{}'\!\nu^-}  \TDAM   x_{\nu^-}
\TDIJ({\widetilde a-1},  {\widetilde a-a +1})c_{\Ad} \Sigma_A,
\endaligned
\end{equation}
since $s_{\tilde{\lambda}_{h}}, s_{\tilde{\lambda}_{h}+1},
	\cdots, s_{\tilde{\lambda}_{h}+\AK(h+1,k)-1}$ are all in $\fS_{\lambda^-_{(h)}}$.
(Recall from Notation \ref{rem_phi_short_0} that {$\widetilde a = \widetilde{a}_{h,k}$}, {$a = a_{h,k}$}, and {$b = a_{h+1,k}$}.)
Thus, by  \eqref{subsubgroup},
%Recall the notations in \eqref{cA-decomp},and  {$u = \tilde{a}_{h,k}$}, {$s = a_{h,k}$}, {$t = a_{h+1,k}$},
%applying equations \eqref{eq_ta+},   \eqref{subsubgroup}, \eqref{xc},and Lemma  \ref{Tnsum},
\begin{equation}\label{m^-}
\aligned
z_k
&=
	{q}^{\AK(h+1,k)} x_{{{\lambda}_{(h)}^{-}}\backslash{}'\!\nu^-}  \TDAM   x_{{\D_{-}}}
	x_{\delta^-} \TDIJ({\widetilde a-1},  {\widetilde a-a +1})c_{\Ad}\Sigma_A\\
&=
	{q}^{\AK(h+1,k)} x_{{{\lambda}_{(h)}^{-}}\backslash{}'\!\nu^-}  \TDAM   x_{{\D_{-}}}
	(c_{\Ad})' x_{\delta^-} \TDIJ({\widetilde a-1},  {\widetilde a-a +1})\Sigma_A\quad(x_{\delta^-} \TDIJ({\widetilde a-1},  {\widetilde a-a +1})=x_\nu, \eqref{xc})\\
&=
	{q}^{\AK(h+1,k)} x_{{{\lambda}_{(h)}^{-}}\backslash{}'\!\nu^-}  \TDAM   x_{{\D_{-}}}
	 (c_{\Ad})' x_{{\delta^-}} \TAIJ( {\widetilde a}, {\widetilde a+b-1})\Sigma_{A^-_{h,k}}\quad(\text{by Lemma  \ref{Tnsum}(2)})\\
&=
	{(-1)}^{(\SOE{a}_{h,k} + \SOE{a}_{h+1,k}) \YK } {q}^{\AK(h+1,k)}
	x_{{{\lambda}_{(h)}^{-}}\backslash{}'\!\nu^-}  \TDAM \cdot
	x_{{\D_{-}}}  ({\cmiddle})' x_{\delta^-} \TAIJ( {\widetilde a}, {\widetilde a+b-1})\cdot {\cbefore} {\cafter}\Sigma_{A^-_{h,k}}.
	%\sum _{{\sigma} \in \D_{{\nu}_{\AM}} \cap \fS_{\mu}} T_{\sigma}.
\endaligned
\end{equation}

For later use, we extract from \eqref{6.4.10} and \eqref{m^-} the element
 $${\frak M}^-:=x_{\nu^-}\TDIJ(\widetilde a-1,\widetilde a-a+1)
	c_{\Ad}
	\Sigma_A={(-1)}^{(\SOE{a}_{h,k} + \SOE{a}_{h+1,k}) \YK } \underbrace{x_{{\D_{-}}}  ({\cmiddle})' x_{\delta^-} \TAIJ( {\widetilde a}, {\widetilde a+b-1})}_{=\fkm^-}\cdot {\cbefore} {\cafter}\Sigma_{A^-_{h,k}}.$$
	Note that $z_k={q}^{\AK(h+1,k)}
	x_{{{\lambda}_{(h)}^{-}}\backslash{}'\!\nu^-}  \TDAM\fkM^-$.

We summarise the computation of $\fkm^-$, $\fkM^-$ and $z_k$ corresponding to the four values of $(x,y):=(a^{\bar1}_{h,k},a^{\bar1}_{h+1,k})$ in Table 1.
\begin{center}
{\bf (Table 1)}
\begin{tabular}{|c|c|}\hline
$(x,y)$&$\fkm^-$\qquad$\fkM^-$\qquad $z_k$\qquad\qquad\qquad\qquad\\\hline
$(0,0)$&$\aligned \fkm^-&=\STEP{b+1}x_{\nu^-}\\
\fkM^-&=\STEPX{b+1}{q}x_{\nu^-}c_{(A^{\bar{0}}-E_{h,k}+E_{h+1,k}|A^{\bar{1}})}\Sigma_{A^-_{h,k}}\\
z_k&={q}^{\AK(h+1,k)}\STEP{b+1}T_{(A^{\bar0}-E_{h,k}+E_{h+1,k}|A^{\bar1})}\endaligned$\\\hline
$(1,0)$&$\aligned \fkm^-&=\STEP{b+1}x_{\nu^-}c_{q,\tilde{a}-a+1,\tilde{a}-1}+q^{a-1}x_{\nu^-}c_{q,\tilde{a},\tilde{a}+b} \\
\fkM^-&=\STEPX{b+1}{q}x_{\nu^-}c_{(A^{\bar{0}}-E_{h,k}+E_{h+1,k}|A^{\bar{1}})}\Sigma_{A^-_{h,k}}+q^{a-1}x_{\nu^-}c_{(A^{\bar{0}}|A^{\bar{1}}-E_{h,k}+E_{h+1,k})}\Sigma_{A^-_{h,k}}\\
z_k&={q}^{\AK(h+1,k)}\STEP{b+1}T_{(A^{\bar0}-E_{h,k}+E_{h+1,k}|A^{\bar1})}+{q}^{\AK(h+1,k)+a-1}T_{(A^{\bar0}|A^{\bar1}-E_{h,k}+E_{h+1,k})}\endaligned$\\\hline
$(0,1)$&$\aligned \fkm^-&=\STEP{b}x_{\nu^-}(c_{q,\tilde{a},\tilde{a}+b}) \\
\fkM^-&=\STEPX{b}{q}x_{\nu^-}c_{(A^{\bar{0}}-E_{h,k}+E_{h+1,k}|A^{\bar{1}})}\Sigma_{A^-_{h,k}}\\
z_k&={q}^{\AK(h+1,k)}\STEP{b}T_{(A^{\bar0}-E_{h,k}+E_{h+1,k}|A^{\bar1})} \endaligned$\\\hline
$(1,1)$&$\aligned \fkm^-&=\STEP{b}x_{\nu^-}(c_{q,\tilde{a}-a+1,\tilde{a}-1})(c_{q,\tilde{a},\tilde{a}+b})-q^{a-2}\STEPPDR{b+1}x_{\nu^-} \\
\fkM^-&=\STEPX{b}{q}x_{\nu^-}c_{(A^{\bar{0}}-E_{h,k}+E_{h+1,k}|A^{\bar{1}})}\Sigma_{A^-_{h,k}}\!\!-\!q^{a-2}\STEPX{b\!+\!1}{q^2,q}x_{\nu^-}c_{(A^{\bar{0}}+2E_{h+1,k}|A^{\bar{1}}-E_{h,k}-E_{h+1,k})}\Sigma_{A^-_{h,k}}\\
z_k&={q}^{\AK(h+1,k)}\STEP{b}T_{(A^{\bar0}-E_{h,k}+E_{h+1,k}|A^{\bar1})}-{q}^{\AK(h+1,k)+a-2}\STEPPDR{b+1}T_{(A^{\bar0}+2E_{h+1,k}|A^{\bar1}-E_{h,k}-E_{h+1,k})}\endaligned$\\\hline
\end{tabular}
\end{center}

Let %by the notational Convention \ref{CONV},  we have
\begin{equation*}
\begin{aligned}
z_k'
&=%\sum_{1\leq k\leq m+n\atop a_{h,k}>0}^n \Big\{
	 {q}^{\AK(h+1,k) }  \STEP{ \SEE{a}_{h+1, k} +1}
	T_{(\SE{A} - E_{h,k} + E_{h+1, k} | \SO{A} )}
	 +
	{q}^{\AK(h+1,k)  + a_{h, k} - 1} T_{(\SE{A}  |\SO{A} - E_{h,k} + E_{h+1,k})} \\
	& \hspace{4.5cm} -
	{q}^{\AK(h+1,k)  + a_{h, k} -2} \STEPPDR{a_{h+1, k} +1}
	T_{(\SE{A} + 2E_{h+1,k} | \SO{A}  - E_{h,k} - E_{h+1,k})}.
%\Big\}.
\end{aligned}
\end{equation*}
Then, by the notational Convention \ref{CONV}, if $(a^{\bar1}_{h,k},a^{\bar1}_{h+1,k})=(0,0)$, then $b=a_{h+1,k}^{\bar 0}$ and the second and third summands in $z_k'$ are 0. Hence, $z_k'$ is the $z_k$ of line $(0,0)$ in Table 1.
The other three cases can be checked easily. This completes the proof of (3).
\end{proof}

%For later use, we extract the following from the proof.
Using the notational Convention \ref{CONV}, the element $\fkM^-$ can be unified as follows. This multiplication formulas in $\HCR$ will be used in the proof of Theorem \ref{philower1}(2).

\begin{cor}\label{fkM}
Maintain the notation in Theorem \ref{phiupper0}. We have in $\HCR$
\begin{equation}\label{F_A3}
\begin{aligned}
x_{\nu_{A^-_{h,k}}}&\TDIJ(\widetilde a_{h,k}-1,\widetilde a_{h,k}-a_{h,k}+1)
	c_{\Ad}
	\Sigma_A
=
	  \STEP{ \SEE{a}_{h+1, k} +1}
	x_{{\nu}_{{A}^-_{h,k}}}c_{(\SE{A} - E_{h,k} + E_{h+1, k} | \SO{A} )}\Sigma_{A^-_{h,k}}\\
&	 +
	{q}^{a_{h, k} - 1} x_{{\nu}_{ {A}^-_{h,k}}}c_{{(\SE{A}  |\SO{A} - E_{h,k} + E_{h+1,k})}} \Sigma_{A^-_{h,k}}\\	
&-{q}^{ a_{h, k} -2} \STEPPDR{a_{h+1, k} +1}
	x_{{\nu}_{ {A}^-_{h,k}}}c_{(\SE{A} + 2E_{h+1,k} | \SO{A}  - E_{h,k} - E_{h+1,k})}\Sigma_{A^-_{h,k}}.
\end{aligned}
\end{equation}
\end{cor}

\begin{rem}
If $A^{\bar 1}=0$, then the multiplication formulas coincide with those for quantum $\mathfrak{gl}_n$ (cf. \cite[Lem.~3.4]{BLM}).
\end{rem}

\section{Multiplication formulas in {$\QqnrR$}: the odd case}\label{odd case}
In this section, we  derive the multiplication formulas for $\phi_\Xd\phi_\Ad$  under the SDP condition, where $\Xd$ is one of the matrices \begin{equation}\label{ODD}
%\Dd_{\bar h,\lambda}:=
(\lambda-E_{h,h}|E_{h,h}),\;\;
 %\Ed_{\bar h,\lambda}:=
 (\lambda-E_{h+1, h+1}| E_{h, h+1} ),\;\;%\Fd_{\bar h,\lambda}:=
 (\lambda-E_{h,h}|E_{h+1,h}).
 \end{equation}
  We also describe the general case using the following order relation.

Recall from \cite{BLM} the partial order `{$\preceq$}' defined over {$ \MN(n)$}.
For {$A, B \in \MN(n)$}, we say {$ B \preceq A $} if the following conditions hold:
\begin{enumerate}
\item
{$\sum_{i \le s, j\ge t } b_{i,j} \le \sum_{i \le s, j\ge t } a_{i,j}$}
	 for all {$1\leq s<t\leq n $};
\item
{$\sum_{i \ge s, j\le t } b_{i,j} \le \sum_{i \ge s, j\le t } a_{i,j}$}
	 for all {$1\leq t<s\leq n $}.	
\end{enumerate}
In particular, {$B \prec A$} if one of the inequalities is strict.

For $A,B\in  \MN(n)_r$, if $d_A$, $d_B$ are the associated double coset representatives, then we have, by \cite[Lem.~13.20]{DDPW},
\begin{equation}\label{13.20}
d_A<d_B\implies A\prec B.
\end{equation}

For any $\mu\in\Lambda(n,r)$, denote by $\mathcal{H}_{\mu,R}^c$ the subsuperalgebra of $\mathcal{H}^c_{r,R}$  associated with the Young subgroup $\mathfrak{S}_\mu$ and the full Clifford superalgebra $\mathcal C_r$.
%$$
%X_{\lambda,\mu}:=\Big(\bigcup_{d\in\mathcal{D}_{\lambda,\mu}}\big(\{d\}\times (\mathcal{D}_{\nu(d)}\cap \mathfrak{S}_{\mu})\big)\Big)\times \mathbb{Z}_2^r.
%$$
\begin{lem}[{\cite[Cors 3.5 \& 3.8]{DW2}\label{lem:basis-permutation}}]\label{6.1}
For any given $\lambda,\mu\in\Lambda(n,r)$, the right (resp. left) $\mathcal{H}^{c}_{r,R}$-module $x_{\lambda}\mathcal{H}^c_{r,R}$ (resp. $\mathcal{H}^c_{r,R}x_\mu$ ) admits a basis $\{x_{\lambda}T_dc^{\alpha}T_{v}\mid d\in\mathcal{D}_{\lambda,\mu},\alpha\in \NN_2^r,v\in\mathcal{D}_{\nu(d)}\cap \mathfrak{S}_{\mu}\}$ (resp. $\{T_uc^\alpha T_d x_\mu\mid u\in\mathcal{D}^{-1}_{\nu(d^{-1})}\cap\mathfrak{S}_\lambda,\alpha\in\NN_2^r,d\in\mathcal{D}_{\lambda,\mu}\}$), where $\fS_{\nu(d)}=d^{-1}\fS_\lambda d\cap\fS_\mu$ and $\fS_{\nu(d^{-1})}=\fS_\lambda \cap d\fS_\mu d^{-1}$. Moreover
\begin{align}
x_{\lambda} T_d=\sum_{u'\in\mathcal{D}^{-1}_{\nu(d^{-1})}\cap \mathfrak{S}_{\lambda}}T_{u'}T_dx_{\nu(d)}\label{eq:X-lambda-Td}\\
\mathcal{H}_{\mu,R}^c=\bigoplus_{\sigma\in \mathcal{D}_{\nu(d)}\cap \mathfrak{S}_\mu}\mathcal{H}^c_{\nu(d),R}T_\sigma\label{Hmuc-decomp}
\end{align}
for each $d\in\mathcal{D}_{\lambda,\mu}$.
\end{lem}

Recall the notations in \eqref{prsum}, and Notation \ref{rem_phi_short_0}. We also apply \eqref{Atilde} to define the partial sum matrix $\widetilde{A}^{\bar 1}=(\widetilde a^{\bar1}_{i,j})$. %and {${\lambda} = \ro(A)$}.
Recall also the SDP condition  in Definition \ref{defn:SDP}.
\begin{thm}\label{phidiag1}
For {$h \in [1,n]$} and {$\Ad =  (A^{\bar0}|A^{\bar 1})=(a^{\bar0}_{i,j}|a^{\bar 1}_{i,j})   \in \MNZ(n,r)$}, let $A=A^{\bar0}+A^{\bar 1}$, {${\lambda} = \ro(A)$}, and $\BK(h,k)=\BK(h,k)(A)$. Let  $\Dd_{\bar h}=(\lambda-E_{h, h}| E_{h, h})$.
\begin{enumerate}
\item	Assume that {$A$}  satisfies the SDP condition on the $h$-th row.
	 Then we have in $\QqnrR$
\begin{align*}
\phi_{\Dd_{\bar h}} \phi_\Ad
	&=
\sum_{ k=1}^n
{(-1)}^{ {\SOE{\widetilde{a}}}_{h-1,k} } {q}^{\BK(h,k) }
\Big\{
	 %{\delta}_{0,\SOE{a}_{h,k}}
	 \phi_{(\SE{A} -E_{h,k}|\SO{A}+ E_{h,k})}
	% {\delta}_{1,\SOE{a}_{h,k}}
	- {\STEPP{a_{h,k}}}
	\phi_{(\SE{A} +  E_{h,k}| \SO{A} - E_{h,k})}
\Big\}=:\sdpHk.
\end{align*}
\item In general, we have
$$\phi_{\Dd_{\bar h}} \phi_\Ad=\sdpHk+ \sum_{
 			\substack{\Bd \in \MNZ(n, r) \\
 				 {\lfloor \Bd \rfloor} \prec { {A} }
 			}
 			}  f^{\Dd_{\bar h},\Ad}_{\Bd} \phi_{\Bd}\;\;(f^{\Dd_{\bar h},\Ad}_{\Bd}  \in R).$$
\end{enumerate}
\end{thm}
\begin{proof}
	Let  {$\Xd=\Dd_{\bar h}=({\lambda-E_{h, h}| E_{h, h}})$}. Then $X=\lfloor\Dd_{\bar h}\rfloor=\diag(\lambda)$.
	Clearly  {$\ro(X) = \lambda$}, {$d_X = 1$},
	{$\D_{{\nu}_{_X}} \cap \fS_{\lambda}  = \{ 1\} $},
	and
	{$ c_\Xd = c_{{q}, \widetilde{\lambda}_{h-1}+1, \widetilde{\lambda}_{h}} $}.
Then, for $\Bd=\Xd$ and $\mu=\co(A)$,   \eqref{eq_gB} becomes
\begin{equation}\label{z elt}
\aligned
z&
:=\phi_\Xd\phi_\Ad(x_\mu)=   x_{\lambda}   c_{{q}, \widetilde{\lambda}_{h-1}+1, \widetilde{\lambda}_{h}}
		{T_{d_A}} c_{\Ad}\Sigma_A\\
		%\sum _{{\sigma} \in \D_{{\nu}_{_A}} \cap \fS_{\mu}} T_{\sigma} \\
&=   x_{\lambda}
		( {q}^{{\lambda}_{h} - 1} c_{\widetilde{\lambda}_{h-1}+1}
		+ {q}^{{\lambda}_{h} - 2} c_{\widetilde{\lambda}_{h-1}+2}+
		 \cdots+
		 {q} c_{\widetilde{\lambda}_{h}- 1}
		+  c_{\widetilde{\lambda}_{h}}
		)
		{T_{d_A}} c_{\Ad}\Sigma_A\\
		%\sum _{{\sigma} \in \D_{{\nu}_{_A}} \cap \fS_{\mu}} T_{\sigma} \\
&=	 \sum_{{1\leq k\leq n}\atop {a_{h,k}>0}}
	 x_{\lambda}  \Big(\sum_{j= {\AK(h,k) }+1 }^{ {\AK(h,k+1) }}
	 {q}^{{\lambda_{h}}-j} c_{ \widetilde{\lambda}_{h-1} +j} \Big){T_{d_A}} c_{\Ad}\Sigma_A\\
	 &=  \sum_{{1\leq k\leq n}\atop {a_{h,k}>0}}  x_{\lambda}
		{q}^{{\lambda_h}-\AK(h,k+1)}
		(c_{q,\ha_{h,k-1}+1,\ha_{h,k-1}+a_{h,k}} {T_{d_A}} )c_{\Ad}\Sigma_A\;\text{ (as $\widetilde{\lambda}_{h-1}+\AK(h,k)=\ha_{h,k-1}$)},
	%\sum _{{\sigma} \in \D_{{\nu}_{_A}} \cap \fS_{\mu}} T_{\sigma}.
\endaligned
\end{equation}
%(Here we used the index change $p = j- {\AK(h,k) } $. Recall also
where the tail $\Sigma_A$ of $T_\Ad$ is defined in \eqref{tail}.)
Applying the SDP condition \eqref{longSDP} may turn $z$ into a summation:
\begin{align*}
z	=  \sum_{{1\leq k\leq n}\atop {a_{h,k}>0}}{q}^{{\lambda_h}-\AK(h,k+1)}x_{\lambda}({T_{d_A}} c_{q,\widetilde a_{h-1,k}+1,\widetilde a_{h-1,k}+a_{h,k}})
    c_{\Ad}\Sigma_A = \sum_{ 1\leq k\leq n, a_{h,k}>0} z_k,
\end{align*}
where, for each $1\leq k\leq n$ with $a_{h,k}>0$,
\begin{equation}\label{zz}
z_k={q}^{{\lambda_h} -  {\AK(h,k+1) }}  x_{\lambda}  {T_{d_A}}\cdot\underbrace{
	c_{{q}, \widetilde{a}_{h-1,k} + 1 , \widetilde{a}_{h-1,k} + a_{h,k} }c_\Ad}_{=\ufm}\cdot\Sigma_A.
\end{equation}
If we decompose $c_\Ad = {\ckbefore} c^{\SOE{a}_{h,k}}_{{q}, \widetilde{a}_{h-1,k} + 1 , \widetilde{a}_{h-1,k} + a_{h,k} } {\ckafter}$, similar to \eqref{cA-decomp}, then the middle part
$$\ufm=  {(-1)}^{\YK} ({\ckbefore}c_{{q}, \widetilde{a}_{h-1,k} + 1 , \widetilde{a}_{h-1,k} + a_{h,k} } c^{\SOE{a}_{h,k}}_{{q}, \widetilde{a}_{h-1,k} + 1 , \widetilde{a}_{h-1,k} + a_{h,k} } {\ckafter}).$$

If  {$\SOE{a}_{h,k} = 0$}, then $\ufm=  {(-1)}^{\YK}c_\Md$, where $\Md=( \SE{A} -E_{h,k}|\SO{A}+ E_{h,k} )$. Since $\Ad$ and $\Md$ have the same base so that  $d_A=d_M,\Sigma_A=\Sigma_M$, it follows that
\begin{align*}
z_k
=  {(-1)}^{\YK} {q}^{{\lambda_h} -  {\AK(h,k+1) }}  x_{\lambda}  {T_{d_M}}
	c_\Md\Sigma_M
	%\sum _{{\sigma} \in \D_{{\nu}_{_A}} \cap \fS_{\mu}} T_{\sigma}
	=
 {(-1)}^{\YK} {q}^{{\lambda_h} -  {\AK(h,k+1) }} T_{( \SE{A} -E_{h,k}|\SO{A}+ E_{h,k} )} .
\end{align*}

If {$\SOE{a}_{h,k} = 1$}, then
$$\aligned
\ufm&={(-1)}^{\YK}({\ckbefore}(c_{{q}, \widetilde{a}_{h-1,k} + 1 , \widetilde{a}_{h-1,k} + a_{h,k} })^2  {\ckafter})\\
&=
(-1)^{\YK}({-\STEPP{a_{h,k}})}({\ckbefore} {\ckafter})\quad(\text{Lemma \ref{xcT}(4)}) \\
&=(-1)^{\YK+1}{\STEPP{a_{h,k}}}c_\Nd,\;\text{where }\Nd=(\SE{A} + E_{h,k}|\SO{A}- E_{h,k} ).
\endaligned$$
Now, $\Ad$ and $\Nd$ have the same base so that $d_A=d_N,\Sigma_A=\Sigma_N$.
 Hence, \eqref{zz} becomes
\begin{align*}
z_k
= {(-1)}^{\YK+1} {q}^{{\lambda_h} -  {\AK(h,k+1) }}  {\STEPP{a_{h,k}}} x_{\lambda}  {T_{d_N}}
	c_\Nd\Sigma_N=
 {(-1)}^{\YK+1} {q}^{{\lambda_h} -  {\AK(h,k+1) }}  {\STEPP{a_{h,k}}}   T_{ (\SE{A} + E_{h,k}|\SO{A}- E_{h,k} )}.
\end{align*}
Finally, putting $z_k'={(-1)}^{ {\SOE{\widetilde{a}}}_{h-1,k} } {q}^{\BK(h,k) }
\big(
	 %{\delta}_{0,\SOE{a}_{h,k}}
	 T_{(\SE{A} -E_{h,k}|\SO{A}+ E_{h,k})}
	% {\delta}_{1,\SOE{a}_{h,k}}
	- {\STEPP{a_{h,k}}}
	T_{(\SE{A} +  E_{h,k}| \SO{A} - E_{h,k})}
\big),$
 one checks easily $z'_k=z_k$, by Convention \ref{CONV}. This proves (1).% that a case-by-case argument  proves  the formula.

  We now prove the general case (2). Applying Lemma \ref{CT2}(2) to \eqref{z elt} yields
 $$\aligned
 z&=\sum_{{1\leq k\leq n}\atop {a_{h,k}>0}} x_{\lambda}
		\Big( \sum_{p=1 }^{a_{h,k}}
		{q}^{{\lambda_h}-\AK(h,k)-p}
			 ( {T_{d_A}}  c_{\tilde{a}_{h-1,k} + p }
			 	+  \sum_{w<d_A,j\in[1,r]}
	 		 f^{\cdots}_{w,j}
			T_{w}c_j)\Big)
			 c_{\Ad}
	\Sigma_A\\
& =\sdpHk+\Tk
\endaligned
$$
where $\sdpHk=\sdpHk(x_\mu)$, $\Sigma_A=\sum_{v\in\mathcal D_{\nu_{\!A}}\cap\fS_\mu}$ and
$$\Tk= \sum_{ 1\leq k\leq n, a_{h,k}>0}\sum_{p=1 }^{a_{h,k}}
			 \sum_{w<d_A,j\in[1,r]}
	 		 f^{\cdots}_{w,j} {q}^{{\lambda_h}-\AK(h,k)-p}
		x_{\lambda}
			T_{w}c_j
			 c_{\Ad}
	\Sigma_A.$$
	
% According to \cite[Prop.5.2]{DW2},

Since every $w$ in the expression of $\Tk$ satisfies $w<d_A$ and may be written as $w=xdy$ with $w\in\fS_\lambda$, $d\in\mathcal D_{\lambda,\mu}$, and $y\in\fS_\mu$, it follows that $d<d_A$. Also, $c_j c_{\Ad}\Sigma_A\in\mathcal{H}^c_{\mu,R}$. By \cite[Cor. 3.8]{DW2}  (or Lemma \ref{6.1}), $\Tk$ is a linear combination of basis elements of the form $x_\lambda T_dc^\alpha T_v$ for $x_\lambda\HCR$, where $d\in\mathcal D_{\lambda,\mu}$ with $d<d_A$, $\alpha\in\NN_2^r$ and $v\in\mathcal D_{\nu_A}\cap\fS_\mu$.

On the other hand, since both $z$ and $\sdpHk$ are in $x_\lambda \HCR \cap \HCR x_\mu$, it follows that
$$\Tk=z-\sdpHk \in x_\lambda \HCR \cap \HCR x_\mu.$$
 Thus, by Proposition \ref{DW-basis}(1),
$\Tk=\sum_{\Md\in M_n(\NN|\NN_2)_{\lambda,\mu}}f_\Md^{\Dd_{\bar h},\Ad} T_\Md$ ($f_\Md^{\Dd_{\bar h},\Ad}\in R$), which is also a linear combination of the same basis for $x_\lambda\HCR$.

Equating coefficients and using \cite[Lem.~13.20]{DDPW} or \eqref{13.20} conclude that
\begin{equation*}\label{y_k diag}
\begin{aligned}\Tk
&=\sum_{\Bd\in \MNZ(n,r),d_{\lfloor\Bd\rfloor}<d_{\lfloor \Ad\rfloor} }f^{\Dd_{\bar h},\Ad}_{\Bd}T_{\Bd}=\sum_{
 			\substack{\Bd \in \MNZ(n, r) \\
 				 {\lfloor \Bd\rfloor} \prec { {A} }
 			}
 			}  f^{\Dd_{\bar h},\Ad}_{\Bd} \phi_{\Bd}(x_{\co(\Bd)}),
\end{aligned}
\end{equation*}
proving (2).
\end{proof}

We now tackle the remaining odd cases.
The proof is parallel to the even cases in Section \ref{even case} under the SDP condition, but with extra care of the $c$-elements. The general case are even more complicated.
\begin{thm}\label{phiupper1}
	Let {$h \in [1,n-1]$}  and {$\Ad =(A^{\bar0}|A^{\bar 1})=  ({\SEE{a}_{i,j}} | {\SOE{a}_{i,j}})   \in \MNZ(n,r)$} with base $A=A^{\bar0}+A^{\bar 1}$,
	 {${\lambda} = \ro(A)$}, and $\BK(h,k)=\BK(h,k)(A)$. Let $\Ed_{\bar h}=(\lambda-E_{h+1, h+1}| E_{h, h+1})$.
	% suppose {$A$}  satisfies the AAA condition on $h$,
\begin{enumerate}
\item	Suppose that, for every $k\in[1,n]$ such that $a_{h+1,k}>0$, $A$ satisfies the SDP condition at $(h,k)$ if $a_{h,k}>0$ and satisfies  $\llcm^{h,k}=0$ if $a_{h,k}=0$.\footnote{This hypothesis is equivalent to that, for every $k\in[1,n]$, $A_{h,k}^+$ satisfies the SDP condition at $(h,k)$.}
% for all $k\in[1,n]$ with $a_{h+1,k}>0$.}
%Assume that, for each {$k \in [1,n]$} with {$a_{h+1,k} \ge 1$},
%	{${A}^+_{h,k}$}  satisfies the SDP condition at $(h,k)$.
	Then
	we have in $\QqnrR$
\begin{equation}\label{oddUP}
\aligned
	\phi_{\Ed_{\bar h}} \phi_\Ad
	&= \sum_{k=1}^n \Big \{
	% {\delta}_{0,\SOE{a}_{h, k}}
	{(-1)}^{{\SOE{\widetilde{a}}}_{h-1,k}} {q}^{\BK(h,k) + \SOE{a}_{h+1,k} } \phi_{(\SE{A} - E_{h+1, k}| \SO{A}  + E_{h,k} )} \\
	& +
	% {\delta}_{1,\SOE{a}_{h+1, k}}
	{(-1)}^{{\SOE{\widetilde{a}}}_{h-1,k} + 1 - \SOE{a}_{h,k} }
	{q}^{\BK(h,k)  } \STEP{ \SEE{a}_{h,k} +1}
		\phi_{(\SE{A} + E_{h,k}| \SO{A} - E_{h+1, k} )}  \\
	&  +
	% {\delta}_{1,\SOE{a}_{h,k}}
	{(-1)}^{{\SOE{\widetilde{a}}}_{h-1,k} } {q}^{\BK(h,k)-1 + \SOE{a}_{h+1,k}} \STEPPDR{a_{h,k}+1}
		\phi_{(\SE{A} + 2 E_{h,k}  -E_{h+1, k} |\SO{A} - E_{h,k})}
\Big\}\\
&=:\sdpHe.
\endaligned
\end{equation}
\item In general, we have
$$\phi_{\Ed_{\bar h}} \phi_\Ad=\sdpHe+ \sum_{
 			\substack{\Bd \in \MNZ(n, r) \\
 				 {\exists k,\, \lfloor \Bd\rfloor} \prec { {A^+_{h.k}} }
 			}
 			}  f^{\Ed_{\bar h},\Ad}_{\Bd} \phi_{\Bd}.$$
			\end{enumerate}
\end{thm}

\begin{proof}
Let {$\Xd=\Ed_{\bar h}= (\lambda-E_{h+1, h+1} | E_{h, h+1})$}. Then its base $X=\lambda+E_{h, h+1}-E_{h+1, h+1}$ which is the same base of the matrix $\Xd$ in Theorem  \ref{phiupper0}{\rm (2)}.
Thus, we have
{$\ro(X) = {{{\lambda}_{(h)}^{+}}}  $},
 {$c_\Xd = c_{\widetilde{\lambda}_{h}+1}$}.
Similar to Theorem \ref{phiupper0}(2),
the equation \eqref{eq_gB} with $\Bd=\Xd$ (and $a=a_{h,k}, b=a_{h+1,k}$)  becomes
\begin{equation}\label{z elt2}
\aligned
z& = \sum_{{1\leq k\leq n}\atop {a_{h+1,k}>0}}
		x_{{{\lambda}_{(h)}^{+}}}    c_{\widetilde{\lambda}_{h}+1}\Big(\sum_{j=\AK(h+1,k)}^{\AK(h+1,k+1) - 1}
		T_{\widetilde{\lambda}_{h}+1} T_{\widetilde{\lambda}_{h}+2} \cdots T_{\widetilde{\lambda}_{h}+j}
		{T_{d_A}}\Big) c_{\Ad}\Sigma_A\\
		&=\sum_{{1\leq k\leq n}\atop {a_{h+1,k}>0}}  x_{{{\lambda}_{(h)}^{+}}}  c_{\widetilde{\lambda}_{h}+1}
		\big(T_{\widetilde{\lambda}_{h}} T_{\widetilde{\lambda}_{h}-1}\cdots T_{\widetilde{\lambda}_{h}-\BK(h,k)+1}  {{\TDAP}}
		\TAIJ( {\widetilde a+1}, {\widetilde a+b-1})\big)c_{\Ad}\Sigma_A\;\;(\text{Lemma }\ref{prop_pjshift})\\
		&=\sum_{{1\leq k\leq n}\atop {a_{h+1,k}>0}} x_{{{\lambda}_{(h)}^{+}}}
		T_{\widetilde{\lambda}_{h}} T_{\widetilde{\lambda}_{h}-1}\cdots T_{\widetilde{\lambda}_{h}-\BK(h,k)+1}
		(c_{\widetilde{\lambda}_{h}-\BK(h,k) + 1}
		{{\TDAP}})\TAIJ( {\widetilde a+1}, {\widetilde a+b-1})
		%T_{\widetilde{a}_{h,k}+1} \cdots T_{\widetilde{a}_{h,k}+p_j}
		c_{\Ad}\Sigma_A.
		%\sum _{{\sigma} \in \D_{{\nu}_{_A}} \cap \fS_{\mu}} T_{\sigma}.
\endaligned
\end{equation}

Observe that $\widetilde{\lambda}_{h}-\BK(h,k) + 1=\widetilde{\lambda}_{h-1}+\AK(h,k+1) + 1=\hahk+a_{h,k} + 1=\widehat a_{h,k-1}+a_{h,k}^+$,
where $a_{i,j}^+$ is the $(i,j)$-entry of $A^+_{h,k}$ (thus, $a_{h,k-1}^+=a_{h,k-1}$ and $a_{h,k}^+=a_{h,k} + 1$). Since the hypothesis in (1) together with Theorem \ref{CdA1} implies that $A^+_{h,k}$ satisfies
 the SDP condition at $(h,k)$, for all $k\in[1,n]$ with $a^+_{h,k}>0$, it follows that $$c_{\widetilde{\lambda}_{h}-\BK(h,k) + 1}
		{{\TDAP}}={{\TDAP}}c_{\widetilde a_{h-1,k}^++a_{h,k}^+}={{\TDAP}}c_{\widetilde a+1}.$$

Hence, since
		$x_{{{\lambda}_{(h)}^{+}}}
		T_{\widetilde{\lambda}_{h}} T_{\widetilde{\lambda}_{h}-1}\cdots T_{\widetilde{\lambda}_{h}-\BK(h,k)+1}={q}^{\BK(h,k)} x_{{{\lambda}_{(h)}^{+}}} $, we have
%  that  %for each {$k \in [1,n]$} with {$a_{h+1,k} \ge 1$},
%	{${A}^+_{h,k}$}  satisfies the SDP condition at $(h,k)$ implies
\begin{align*}
z=\sum_{{1\leq k\leq n}\atop {a_{h+1,k}>0}}z_k,\;\;
\text{ where }z_k
&= {q}^{\BK(h,k)} x_{{{\lambda}_{(h)}^{+}}} {{\TDAP}} c_{\widetilde a +1} \TAIJ( {\widetilde a+1}, {\widetilde a+b-1}) c_\Ad\Sigma_A,
	%\sum _{{\sigma} \in \D_{{\nu}_{_A}} \cap \fS_{\mu}} T_{\sigma} .
\end{align*}
and,
similar to  the even case, using an argument for turning $\Sigma_A$ to $\Sigma_{A^+_{h,k}}$ right above \eqref{zkupper} together with the fact $x_{\delta^+}c_{\widetilde a+1}=c_{\widetilde a+1}x_{\delta^+}$ gives rise to
%applying \eqref{eq_ta+}, \eqref{subsubgroup}, Lemma  \ref{Tnsum}, and  \eqref{xc} yields
\begin{equation}\label{C2zk}
\aligned
z_k&
=
	{(-1)}^{(\SOE{a}_{h,k} + \SOE{a}_{h+1,k}) \YK } {q}^{\BK(h,k)}
	x_{{{\lambda}_{(h)}^{+}}\backslash{}'\!\nu^+} {{\TDAP}}\cdot\underbrace{
	x_{{\D_{+}}}  c_{\widetilde a +1} ({\cmiddle})'  x_{\delta^+}   \TDIJ({\widetilde a},  {\widetilde a-a+1})}_{=\ufm^+}\cdot {\cbefore} {\cafter}\Sigma_{A^+_{h,k}}.
	%\!\!  \sum _{{\sigma} \in \D_{\nu^+} \cap \fS_{\mu}} \!\! \!\! T_{\sigma} .
\endaligned
\end{equation}

%Notice the factor {$c_{u+1}$}, which causes the difference from the even case,
%we discuss the  four cases.
Similar to the even situation, we need to analyse the middle part $\ufm^+$ in four cases.

\medskip
\noindent
{\bf Case 1.}  {$\SOE{a}_{h,k} = 0$}, {$\SOE{a}_{h+1,k} = 0$}.
In this case we have   {$({\cmiddle})' = 1$} and, by Lemma \ref{xcT}(3),
$$\ufm^+=x_{{\D_{+}}}  c_{\widetilde a +1} x_{\delta^+}  \TDIJ({\widetilde a},  {\widetilde a-a+1})= x_{\nu^+} c_{\widetilde a +1} \TDIJ({\widetilde a},  {\widetilde a-a+1})= x_{\nu^+} c_{q, \widetilde a-a+1,\widetilde a +1}.$$
Thus, if $\Md=(\SE{A} - E_{h+1, k}| \SO{A}  + E_{h,k} )$, then $\Md$ has the base $M=A^+_{h,k}$ and $\ro(M)=\lambda^+_{(h)}$. Moreover, we have
 $c_\Md={\cbefore}  c_{{q}, \widetilde a-a+1, \widetilde a +1}  {\cafter}$. Hence,
\begin{align*}
z_k &=
	{q}^{\BK(h,k)} x_{{{\lambda}_{(h)}^{+}}\backslash{}'\!\nu^+} {{\TDAP}} x_{\nu^+} c_{q, \widetilde a-a+1,\widetilde a +1}
	%x_{\nu^+} c_{\widetilde a +1} \TDIJ({\widetilde a},  {\widetilde a-a+1})
	{\cbefore} {\cafter}\Sigma_{A^+_{h,k}}\\
	 %\sum _{{\sigma} \in \D_{\nu^+} \cap \fS_{\mu}}  T_{\sigma} \\
&=
	{(-1)}^{\YK} {q}^{\BK(h,k)} x_{{{\lambda}_{(h)}^{+}}} d_M
	 {\cbefore}  c_{{q}, \widetilde a-a+1, \widetilde a +1}  {\cafter}\Sigma_{M}\\
	%\sum _{{\sigma} \in \D_{\nu^+} \cap \fS_{\mu}}  T_{\sigma} \\
&=
	{(-1)}^{\YK} {q}^{\BK(h,k)} T_{(\SE{A} - E_{h+1, k}| \SO{A}  + E_{h,k} )}.
\end{align*}

\medskip
\noindent
{\bf Case 2.}
 {$\SOE{a}_{h,k} = 0$}, {$\SOE{a}_{h+1,k} = 1$}.
Then {$({\cmiddle})' = (c_{{q}, \widetilde a +1, \widetilde a +b})'=c_{\widetilde a+1}+q(c_{{q}, \widetilde a +2, \widetilde a +b})'$} and
\begin{align*}
	%&	x_{{\D_{+}}} c_{u +1}  (c_{{q}, u +1, u +t})'  x_{\delta}   \TDIJ({u},  {u-s+1}) \\
	\ufm^+&= {q} x_{{\D_{+}}} c_{\widetilde a +1} (c_{{q}, \widetilde a +2, \widetilde a +b})'x_{\delta^+} \TDIJ({\widetilde a},  {\widetilde a-a+1})  +  x_{{\D_{+}}} c_{\widetilde a +1}^2 x_{\delta^+}  \TDIJ({\widetilde a },  {\widetilde a-a+1}) \\
	&= {q} x_{\nu^+}  c_{\widetilde a +1}  \TDIJ({\widetilde a},  {\widetilde a-a+1})c_{{q}, \widetilde a +2, \widetilde a +b}  -  x_{\nu^+}  \TDIJ({\widetilde a },  {\widetilde a-a+1}) \\
	&= {q} x_{\nu^+}   c_{{q}, \widetilde a-a+1, \widetilde a +1} c_{{q}, \widetilde a +2, \widetilde a +b} -  \STEP{a+1}  x_{\nu^+}.
\end{align*}
%\begin{align*}
%z_k &=
%	{(-1)}^{\YK } {q}^{\BK(h,k)} x_{{{\lambda}_{(h)}^{+}}\backslash{}'\!\nu^+} {{\TDAP}}
%	x_{{\D_{+}}}  c_{u +1} (c_{{q}, u +1, u +t})'  x_{\delta}  \TDIJ({u},  {u-s+1})  {\cbefore} {\cafter}
%	 \sum _{{\sigma} \in \D_{\nu^+} \cap \fS_{\mu}}  T_{\sigma} .
%\end{align*}
%Meanwhile by \eqref{cqij}, \eqref{subsubgroup}, Lemma \ref{xcT}(2) and Lemma \ref{xcT}(3), we deduce that
Substituting into \eqref{C2zk} and noting that the above $\Md$ and $\Nd=(\SE{A} + E_{h,k}| \SO{A} - E_{h+1, k} )$ have the same base $A^+_{h,k}$ and
$$c_{{q}, \widetilde a-a+1, \widetilde a +1}  c_{{q}, \widetilde a +2, \widetilde a +b} \cdot {\cbefore} {\cafter}= {\cbefore} c_{{q}, \widetilde a-a+1, \widetilde a +1}  c_{{q}, \widetilde a +2, \widetilde a +b} {\cafter}=c_\Md,$$ we obtain
\begin{align*}
z_k
&=
	{(-1)}^{\YK } {q}^{\BK(h,k)+1} x_{{{\lambda}_{(h)}^{+}}\backslash{}'\!\nu^+} {{\TDAP}}
	x_{\nu^+}   c_{{q}, \widetilde a-a+1, \widetilde a +1}  c_{{q}, \widetilde a +2, \widetilde a +b} \cdot {\cbefore} {\cafter}\Sigma_{A^+_{h,k}}\\
	% \sum _{{\sigma} \in \D_{\nu^+} \cap \fS_{\mu}}  T_{\sigma} \\
&\qquad - {(-1)}^{\YK } {q}^{\BK(h,k)}  \STEP{a+1}  x_{{{\lambda}_{(h)}^{+}}\backslash{}'\!\nu^+} {{\TDAP}}
 	x_{\nu^+} {\cbefore} {\cafter}\Sigma_{A^+_{h,k}}\\
	%\sum _{{\sigma} \in \D_{\nu^+} \cap \fS_{\mu}}  T_{\sigma}  \\
&=	 %{\delta}_{0,\SOE{a}_{h,k}}
	{(-1)}^{\YK } {q}^{\BK(h,k)+1}
	T_{(\SE{A} - E_{h+1, k}| \SO{A}  + E_{h,k})}  +
	{(-1)}^{\YK +1 } {q}^{\BK(h,k)} \STEP{ \SEE{a}_{h,k} +1}
	T_{(\SE{A} + E_{h,k}| \SO{A} - E_{h+1, k} )} .
\end{align*}

\medskip
\noindent
{\bf Case 3.}
  {$\SOE{a}_{h,k} = 1$}, {$\SOE{a}_{h+1,k} = 0$}.
It follows {${\cmiddle} = c_{{q}, \widetilde a -a+1, \widetilde a }$} and the middle part has the form
\begin{align*}
z_k &=
	{(-1)}^{\YK } {q}^{\BK(h,k)} x_{{{\lambda}_{(h)}^{+}}\backslash{}'\!\nu^+} {{\TDAP}}
	\underbrace{x_{{\D_{+}}}  c_{  \widetilde a+1}  (c_{{q},  \widetilde a-a+1,  \widetilde a})'  x_{\delta}   \TDIJ({ \widetilde a},  { \widetilde a-a+1})}_{\ufm^+} {\cbefore} {\cafter}\Sigma_{A^+_{h,k}}.
	 %\sum _{{\sigma} \in \D_{\nu^+} \cap \fS_{\mu}} T_{\sigma} .
\end{align*}
Using \eqref{cqij}, \eqref{subsubgroup}, Lemma \ref{xcT}(2)\&(3), we have
\begin{align*}
q\ufm^+&={q} x_{{\D_{+}}}  c_{\widetilde a  +1} (c_{{q}, \widetilde a -a+1, \widetilde a })'  x_{\delta^+}  \TDIJ({\widetilde a },  {\widetilde a -a+1})  \\
	&=  -x_{{\D_{+}}}   x_{\delta^+}  {q} c_{{q}, \widetilde a -a+1, \widetilde a } c_{\widetilde a  +1} \TDIJ({\widetilde a },  {\widetilde a -a+1}) \\
	&=  	- x_{\nu^+}  c_{{q}, \widetilde a -a+1, \widetilde a  +1} c_{\widetilde a  +1} \TDIJ({\widetilde a },  {\widetilde a -a+1})
		+ x_{\nu^+}  (c_{\widetilde a +1})^2  \TDIJ({\widetilde a },  {\widetilde a -a+1}) \\
	&=  	- ( c_{{q}, \widetilde a -a+1, \widetilde a  +1} )' x_{\nu^+}  c_{\widetilde a  +1} \TDIJ({ \widetilde a},  {\widetilde a -a+1})
		- x_{\nu^+}  \TDIJ({\widetilde a },  {\widetilde a -a+1}) \\
	&=  	- x_{\nu^+}  (c_{{q}, \widetilde a -a+1, \widetilde a  +1})^2%  c_{{q}, \widetilde a -a+1, \widetilde a  +1}
		- \STEP{a+1} x_{\nu^+} \quad(\text{by Lemma \ref{xcT}(3)}) \\
	&=   ( {\STEPP{a+1}} - \STEP{a+1} ) x_{\nu^+}, \quad(\text{by Lemma \ref{xcT}(4)}).
\end{align*}
Hence, taking $\Nd=(\SE{A}+ 2 E_{h,k}  -E_{h+1, k}  |\SO{A} - E_{h,k})$ which has base $N=A^+_{h,k}$ and $c_\Nd={\cbefore} {\cafter}$, we have
\begin{align*}
z_k&=
	{(-1)}^{\YK } {q}^{\BK(h,k)-1} x_{{{\lambda}_{(h)}^{+}}\backslash{}'\!\nu^+} {{\TDAP}}\big(( {\STEPP{a+1}} - \STEP{a+1} ) x_{\nu^+}\big)
	%x_{{\D_{+}}}  c_{u +1}  (c_{{q}, u-s+1, u})'  x_{\delta}   \TDIJ({u},  {u-s+1})
	 {\cbefore} {\cafter}\Sigma_{A^+_{h,k}}\\
&=
	{(-1)}^{\YK } {q}^{\BK(h,k)-1}	( {\STEPP{a_{h,k}+1}}  - \STEP{a_{h,k}+1} )
	 T_{(\SE{A}+ 2 E_{h,k}  -E_{h+1, k}  |\SO{A} - E_{h,k})}.
\end{align*}

\medskip
\noindent
{\bf Case 4.}  {$\SOE{a}_{h,k} = 1$}, {$\SOE{a}_{h+1,k} = 1$}.
Then {(${\cmiddle})' =c'_{{q}, \widetilde a-a+1, \widetilde a} c'_{{q}, \widetilde a +1, \widetilde a +b} $} and \eqref{C2zk} takes the form
\begin{equation}\label{last z_k}
z_k =
	{q}^{\BK(h,k)-1} x_{{{\lambda}_{(h)}^{+}}\backslash{}'\!\nu^+} {{\TDAP}}(q\ufm^+) {\cbefore} {\cafter}\Sigma_{A^+_{h,k}},
	%\sum _{{\sigma} \in \D_{\nu^+} \cap \fS_{\mu}}  T_{\sigma},
\end{equation}
where $q\ufm^+=x_{{\D_{+}}}  c_{\widetilde a +1} c'_{{q}, \widetilde a-a+1, \widetilde a} c'_{{q}, \widetilde a +1, \widetilde a +b}  x_{\delta^+}  \TDIJ(\widetilde{a},  \widetilde{a}-a+1)$. Multiplying \eqref{c'c'} by $c_{\widetilde a +1} $ to the left yields
$$
\aligned
 qc_{\widetilde a +1}  c'_{{q}, \widetilde a-a+1, \widetilde a} c'_{{q}, \widetilde a +1, \widetilde a +b}  x_{\delta^+}=	x_{\delta^+} (qc_{{q}, \widetilde a +2, \widetilde a +b}c_{\widetilde a +1} c_{{q}, \widetilde a-a+1, \widetilde a +1}
	+{q}  c_{{q}, \widetilde a +2, \widetilde a +b}
	+  c_{{q}, \widetilde a-a+1, \widetilde a +1}
	 -   c_{\widetilde a+1} )
\endaligned
$$
Substituting into $q\ufm^+$ with a similar argument right after \eqref{c'c'},
%With the fact $c_{u +1} c_{{q}, u-s+1, u +1}= -c_{{q}, u-s+1, u +1} c_{u +1}  -2$
% and  \eqref{subsubgroup}, Lemma \ref{xcT}, %(2) as well as Lemma \ref{xcT}(3),
 we have
\begin{align*}
%&{q} x_{{\D_{+}}}  c_{u +1} (c_{{q}, u-s+1, u} c_{{q}, u +1, u +t})'  x_{\delta}  \TDIJ({u},  {u-s+1}) \\
%&= - x_{{\D_{+}}}  c_{u +1} (c_{{q}, u +1, u +t})' x_{\delta} {q} c_{{q}, u-s+1, u}  \TDIJ({u},  {u-s+1})  \\
%&=	- x_{{\D_{+}}}  c_{u +1} {q} (c_{{q}, u +2, u +t})' x_{\delta} c_{{q}, u-s+1, u +1} \TDIJ({u},  {u-s+1})
%	+ x_{{\D_{+}}}  c_{u +1} {q} (c_{{q}, u +2, u +t})'  x_{\delta} c_{u+1}  \TDIJ({u},  {u-s+1})  \\
%	&\qquad - x_{{\D_{+}}}  c_{u +1} c_{u+1} x_{\delta} c_{{q}, u-s+1, u +1} \TDIJ({u},  {u-s+1})
%	 + x_{{\D_{+}}}  c_{u +1} c_{u+1} x_{\delta} c_{u+1}  \TDIJ({u},  {u-s+1})  \\
%&= 	{q} x_{{\D_{+}}}   x_{\delta} c_{{q}, u +2, u +t} c_{u +1} c_{{q}, u-s+1, u +1} \TDIJ({u},  {u-s+1})
%	+ {q} x_{{\D_{+}}}   x_{\delta} c_{{q}, u +2, u +t} \TDIJ({u},  {u-s+1})  \\
%	&\qquad + x_{{\D_{+}}}   x_{\delta} c_{{q}, u-s+1, u +1} \TDIJ({u},  {u-s+1})
%	 - x_{{\D_{+}}}   x_{\delta} c_{u+1}  \TDIJ({u},  {u-s+1})  \\
q\ufm^+&=	{q} x_{\nu^+} c_{{q}, \widetilde a +2, \widetilde a +b} (-c_{{q}, \widetilde a-a+1, \widetilde a +1} c_{\widetilde a +1}  -2) \TDIJ({\widetilde{a}},  {\widetilde a-a+1})
	+ {q} x_{\nu^+} c_{{q}, \widetilde a +2, \widetilde a +b} \TDIJ({\widetilde a},  {\widetilde a-a+1})  \\
	&\qquad + x_{\nu^+} c_{{q}, \widetilde a-a+1, \widetilde a +1} \TDIJ({\widetilde a},  {\widetilde a-a+1})
	 - x_{\nu^+} c_{\widetilde a+1}  \TDIJ({\widetilde a},  {\widetilde a-a+1})  \\
&=	 -{q} x_{\nu^+} c_{{q}, \widetilde a +2, \widetilde a +b} c_{{q}, \widetilde a-a+1, \widetilde a+1} c_{\widetilde a +1}  \TDIJ({\widetilde a},  {\widetilde a-a+1})
	- {q} x_{\nu^+} c_{{q}, \widetilde a +2, \widetilde a +b}  \TDIJ({\widetilde a},  {\widetilde a-a+1}) \\
	&\qquad  + x_{\nu^+} c_{{q}, \widetilde a-a+1, \widetilde a +1} \TDIJ({\widetilde a},  {\widetilde a-a+1})
	- x_{\nu^+} c_{\widetilde a+1}  \TDIJ({\widetilde a},  {\widetilde a-a+1})  \\
&=	-{q} c'_{{q}, \widetilde a +2, \widetilde a +b}c'_{{q}, \widetilde a-a+1, \widetilde a +1} x_{\nu^+}  c_{{q}, \widetilde a-a+1, \widetilde a +1}
	- {q} c'_{{q}, \widetilde a +2, \widetilde a +b}  \STEP{a+1} x_{\nu^+}  \\
	&\qquad  + c'_{{q}, \widetilde a-a+1, \widetilde a +1}\STEP{a+1}  x_{\nu^+}
	- x_{\nu^+} c_{{q}, \widetilde a-a+1, \widetilde a +1}   \\
&=	{q} ({\STEPP{a+1}} - \STEP{a+1}) x_{\nu^+} c_{{q}, \widetilde a +2, \widetilde a +b}
	+(\STEP{a+1} -1)  x_{\nu^+} c_{{q}, \widetilde a-a+1, \widetilde a +1}.
\end{align*}
Then \eqref{last z_k} becomes after substitution
\begin{align*}
z_k&=
	{(-1)}^{\YK} {q}^{\BK(h,k)} ({\STEPP{a+1}} - \STEP{a+1}) x_{{{\lambda}_{(h)}^{+}}} {{\TDAP}}
	 {\cbefore} c_{{q}, \widetilde a +2, \widetilde a +b} {\cafter}\Sigma_{A^+_{h,k}}\\
	%\sum _{{\sigma} \in \D_{\nu^+} \cap \fS_{\mu}} T_{\sigma} \\
	& \qquad +
	{(-1)}^{\YK} {q}^{\BK(h,k)-1} (\STEP{a+1} -1) x_{{{\lambda}_{(h)}^{+}}} {{\TDAP}}
	  {\cbefore} c_{{q}, \widetilde a-a+1, \widetilde a +1} {\cafter}\Sigma_{A^+_{h,k}}.
	%\sum _{{\sigma} \in \D_{\nu^+} \cap \fS_{\mu}}  T_{\sigma} \\
\end{align*}
Let $\Bd=(b^{\bar0}_{i,j}|b^{\bar1}_{i,j})=( \SE{A}+2E_{h,k} - E_{h+1, k} | \SO{A}  - E_{h,k} )$. Then its base $B=(b_{i,j})=A^+_{h,k}$ and $b^{\bar1}_{h,k}=0,b^{\bar1}_{h+1,k}=1$, and $b_{h+1,k}=a_{h+1,k}-1=b-1$. Thus,
$\widetilde b_{h,k}+1=\widetilde a+2$, and $\widetilde b_{h+1,k}=\widetilde a+b$. Hence, $c_\Bd={\cbefore} c_{{q}, \widetilde a +2, \widetilde a +b} {\cafter}$. Similarly, putting $\Cd=( \SE{A} + E_{h,k} | \SO{A}  - E_{h+1,k} )$, we have $c_\Cd={\cbefore} c_{{q}, \widetilde a-a+1, \widetilde a +1} {\cafter}$. Consequently,
\begin{align*}
z_k&={(-1)}^{\YK} {q}^{\BK(h,k)} ({\STEPP{a_{h,k} + 1}} - \STEP{ a_{h,k} +1})
	T_{( \SE{A}+2E_{h,k} - E_{h+1, k} | \SO{A}  - E_{h,k} ) } \\
	& \; +
	{(-1)}^{\YK} {q}^{\BK(h,k)} \STEP{ \SEE{a}_{h,k} + 1}
	T_{ ( \SE{A} + E_{h,k} | \SO{A}  - E_{h+1,k} ) }\;\;  (\text{noting $\STEP{a+1} -1=q\STEP{a_{h,k}}$}).
\end{align*}
Finally,
by using Convention \ref{CONV},  one checks easily $\sum_{1\leq k\leq n\atop a_{h+1,k}>0}z_k$ is the image of the right hand side of \eqref{oddUP} at $x_{\co(A)}$. This proves (1).

In general, by Lemma \ref{CT2}(2), $c_{\widetilde{\lambda}_{h}-\BK(h,k) + 1}
		{{\TDAP}}={{\TDAP}}c_{\widetilde a_{h,k}+1}+\sum_{w<d_{A^+_{h,k}},j\in[1,r]}f_{w,j}^{\cdots}T_wc_j$.
Thus,
\eqref{z elt2} takes the form
$$\aligned
z&=\sum_{{1\leq k\leq n}\atop {a_{h+1,k}>0}} {q}^{\BK(h,k)} x_{{{\lambda}_{(h)}^{+}}}
\Big(\TDAP c_{\tilde{a}_{h,k}+1}+\sum_{w<d_{A^+_{h,k}},j\in[1,r]}f_{w,j}^{\cdots}T_wc_j\Big)\TAIJ( {\widetilde a+1}, {\widetilde a+b-1})
		%T_{\widetilde{a}_{h,k}+1} \cdots T_{\widetilde{a}_{h,k}+p_j}
		c_{\Ad}
		\Sigma_A\\
&=\sdpHe+\Te,%\sum_{{1\leq k\leq n}\atop {a_{h+1,k}>0}} y_k,
\endaligned
$$
where $\sdpHe=\sdpHe(x_{\mu})$ ($\mu=\co(A)$)
\begin{equation}\label{Te1}
\Te=\sum_{{1\leq k\leq n}\atop {a_{h+1,k}>0}}\sum_{w<d_{A^+_{h,k}},j\in[1,r]}
		f_{w,j}^{\cdots}{q}^{\BK(h,k)} x_{{{\lambda}_{(h)}^{+}}}
T_wc_j\TAIJ( {\widetilde a+1}, {\widetilde a+b-1})
		%T_{\widetilde{a}_{h,k}+1} \cdots T_{\widetilde{a}_{h,k}+p_j}
		c_{\Ad}
		\Sigma_A.
\end{equation}
Since $\TAIJ({\widetilde a+1},, {\widetilde a+b-1})c_{\Ad} \Sigma_A\in \mathcal{H}^c_{\mu,R}$ and every $w<d_{A^+_{h,k}}$ has the form
$w=xdy$ with $x\in\fS_{\lambda_{(h)}^+}$, $d\in\mathcal D_{\lambda_{(h)}^+,\mu}$ and $y\in\fS_\mu$, it follows that $\Te$ is a linear combination of the basis elements $x_{\lambda_{(h)}^+}T_dc^\alpha T_v$ for $x_{\lambda_{(h)}^+}\HCR$  (see \cite[Cor. 3.8]{DW2} or Lemma \ref{6.1}) where $d<d_{A^+_{h,k}}$, for some $k$.

On the other hand, since both $z$ and $\sdpHe$ are in $x_{\lambda_{(h)}^+}\mathcal{H}^c_{r,R}\cap\mathcal{H}^c_{r,R}x_\mu$, it follows that $\Te\in x_{\lambda_{(h)}^+}\mathcal{H}^c_{r,R}\cap\mathcal{H}^c_{r,R}x_\mu.$ Hence, by Proposition \ref{DW-basis}(1),%\cite[Prop. 5.2]{DW2},
\begin{equation}\label{Te2}
\Te=\sum_{\Md\in M_n(\NN|\NN_2)_{\lambda_{(h)}^+,\mu}}f^{\Ed_{\bar h},\Ad}_\Md T_\Md.
\end{equation}
Writing this sum as a linear combination of the same basis for $x_{\lambda_{(h)}^+}\HCR$ and equating coefficients show that every
$d_{\lfloor\Md\rfloor}<d_{A^+_{h,k}}$, for some $k$. Now assertion (2) follows from the fact $d_A<d_B\implies A\prec B$; see \cite[Lem. 13.20]{DDPW} or \eqref{13.20}.
% a case-by-case argument  proves the result.
\end{proof}

We now deal with the most complicated case.
\begin{thm}\label{philower1}
Let {$h \in [1,n-1]$}  and {$\Ad =(A^{\bar0}|A^{\bar 1})=  ({\SEE{a}_{i,j}} | {\SOE{a}_{i,j}})   \in \MNZ(n,r)$} with $A=A^{\bar0}+A^{\bar 1}$,
	 {${\lambda} = \ro(A)$}, and $\BK(h+1,k)=\BK(h+1,k)(A)$. Let $\Fd_{\bar h}=(\lambda-E_{h, h}| E_{h+1, h})$.
\begin{item}
\item[(1)] If $A$ satisfies the SDP condition at $(h,k)$, for all $k\in[1,n]$ with $a_{h,k}>0$, then we have
\begin{align*}
\phi_{\Fd_{\bar h}} \phi_\Ad
 &=\sum_{{1\leq k\leq n}\atop {a_{h,k}>0}} \Big\{
	% {\delta}_{0,\SOE{a}_{h+1, k}}
	{(-1)}^{{\SOE{\widetilde{a}}}_{h-1,k} + \SOE{a}_{h,k}} 	\phi_{(\SE{A} - E_{h,k}| \SO{A}+ E_{h+1, k})} \\
	& \qquad +
	% {\delta}_{1,\SOE{a}_{h+1, k}}
	{(-1)}^{{\SOE{\widetilde{a}}}_{h-1,k} + \SOE{a}_{h,k}+1}
	{q}^{-1}  \STEPPDR{a_{h+1, k} +1}
	\phi_{(\SE{A} - E_{h,k} + 2E_{h+1, k} | \SO{A}  -E_{h+1, k})} \\
	& \qquad +
	% {\delta}_{1,\SOE{a}_{h,k}}
	{(-1)}^{{\SOE{\widetilde{a}}}_{h-1,k}+1} {q}^{a_{h,k}-1}
	\STEP{ \SEE{a}_{h+1,k}+1}
	\phi_{(\SE{A} + E_{h+1,k} | \SO{A}  - E_{h,k} )}
\Big\}=:\sdpHf.
\end{align*}

\item[(2)] In general, let
\begin{align*}
\text{\rm HH}\overline{\textsc f}&=\sum_{{1\leq k\leq n}\atop {a_{h,k}>0}}\sum_{l=1}^{k-1}(-1)^{\tilde{a}^{\bar1}_{h,l}}~\cdot{q}^{\overrightarrow{r}^{l}_{h+1}{-\overrightarrow{r}_{h+1}^{k-1}}}~\Big\{\STEP{ \SEE{a}_{h+1, k} +1}\cdot\\
&\quad\quad\quad\qquad\;(\phi_{(\SE{A} - E_{h,k} + E_{h+1, k} -E_{h+1,l}| \SO{A}+E_{h+1,l})}-{\STEPP{a_{h+1,l}}}\phi_{(\SE{A} - E_{h,k} + E_{h+1, k} +E_{h+1,l}| \SO{A}-E_{h+1,l})})\\
&\quad\quad\quad\quad+{q}^{a_{h, k} - 1}(\phi_{{(\SE{A} -E_{h+1,l} |\SO{A} - E_{h,k} + E_{h+1,k}+E_{h+1,l})}}-{\STEPP{a_{h+1,l}}}\phi_{{(\SE{A} +E_{h+1,l} |\SO{A} - E_{h,k} + E_{h+1,k}-E_{h+1,l})}})\\	
&\quad\quad\quad\quad+{q}^{ a_{h, k} -2} \STEPPDR{a_{h+1, k} +1}(
	\phi_{(\SE{A} + 2E_{h+1,k}-E_{h+1,l} | \SO{A}  - E_{h,k} - E_{h+1,k}+E_{h+1,l})}\\
&\quad\quad\quad\quad\quad\quad\quad\quad\quad\quad-{\STEPP{a_{h+1,l}}}\phi_{(\SE{A} + 2E_{h+1,k}+E_{h+1,l} | \SO{A}  - E_{h,k} - E_{h+1,k}-E_{h+1,l})})
\Big\}.
\end{align*}
Then $$\phi_{\Fd_{\bar h}} \phi_{\Ad}=\sdpHf+(q-1)\text{\rm HH}\overline{\textsc f}+\sum_{\Bd\in \MNZ(n,r)\atop \exists k, \lfloor \Bd\rfloor\prec{A}^-_{h,k}}f^{\Fd_{\bar h},\Ad}_{\Bd}\phi_{\Bd}.$$
\end{item}
\end{thm}

\begin{proof} We want to compute $z:=\phi_{\Xd} \phi_A(x_\mu)$ with $\Xd=\Fd_{\bar h}=(\lambda-E_{h, h}| E_{h+1, h})$ and $\mu=\co(A)$. Since $X=\lfloor
\Fd_{\bar h}\rfloor$ is the same matrix $X$ as given in \eqref{Fbar}, and $c_\Xd=c_{\tilde\lambda_h}$, it follows from \eqref{Fbar} that
$\sum _{{\sigma} \in \D_{{\nu}_{_X}} \cap \fS_{\lambda}} T_{\sigma}= \TDIJ( {\tilde{\lambda}_{h}-1}, {\tilde{\lambda}_{h-1}+1})$, and thus,
\begin{equation}\label{F1_A_z}
\begin{aligned}
z%&=\phi_{(\lambda-E_{h, h}| E_{h+1, h})} \phi_A(x_\mu)\\
&= x_{{{\lambda}_{(h)}^{-}}} T_{d_X} c_\Xd  \TDIJ( {\tilde{\lambda}_{h}-1}, {\tilde{\lambda}_{h-1}+1})
	 \cdot
	( {T_{d_A}} c_{\Ad}  \Sigma_A) \\
 &=
 \sum_{{1\leq k\leq n}\atop {a_{h,k}>0}}x_{{{\lambda}_{(h)}^{-}}}  c_{\tilde{\lambda}_{h}}
\Big( \sum_{j=\BK(h,k)}^{\BK(h,k-1) - 1}
		 T_{\tilde{\lambda}_{h}-1} \cdots T_{\tilde{\lambda}_{h}-j }  {T_{d_A}}\Big) c_{\Ad}
		\Sigma_A\\
&= \sum_{{1\leq k\leq n}\atop {a_{h,k}>0}}
	%\sum_{j=\BK(h,k)}^{\BK(h,k-1) - 1}
	x_{{{\lambda}_{(h)}^{-}}}c_{\tilde{\lambda}_{h}}\big(
	T_{\tilde{\lambda}_{h}} T_{\tilde{\lambda}_{h}+1}
	\cdots T_{\tilde{\lambda}_{h}+\AK(h+1,k)-1}   \TDAM \TDIJ(\widetilde a-1,\widetilde a-a+1)\big)
	%T_{\tilde{a}_{h,k}-1} \cdots T_{\tilde{a}_{h,k}-q_j}
	c_{\Ad}
	\Sigma_A\;\;(a=a_{h,k}),
\end{aligned}
\end{equation}
by Lemma \ref{prop_pjshift}(2). For each $k\in[1,n]$ with $a_{h,k}>0$,
let
\begin{equation}\label{642}
z_k=x_{{{\lambda}_{(h)}^{-}}} c_{\tilde{\lambda}_{h}}
	T_{\tilde{\lambda}_{h}}T_{\tilde{\lambda}_{h}+1}
	\cdots T_{\tilde{\lambda}_{h}+\AK(h+1,k)-1}  \TDAM \TDIJ(\widetilde a-1,\widetilde a-a+1)
	c_{\Ad}
	\Sigma_A.
\end{equation}

The hypothesis that $A$ satisfies the SDP condition at $(h,k)$ implies by Theorem \ref{CdA1} that $\llcm^{h,k}=0$.
Thus, the lower left corner matrix of $A^-_{h,k}=(a_{i,j}^-)$ at $(h,k)$ (and hence, at $(h+1,k)$) is zero. This implies, on the one hand, $a_{h+1,j}=0$  for $j<k$, forcing $\overleftarrow{r}^k_{h+1}=0$ and so, $T_{\tilde{\lambda}_{h}}\cdots T_{\tilde{\lambda}_{h}+\AK(h+1,k)-1}=1$. On the other hand,
since $a^-_{h+1,k}=a_{h+1,k}+1>0$, it follows that $A^-_{h,k}$ satisfies the SDP condition at $(h+1,k)$, by Theorem \ref{CdA1} again.
Thus, as $\tilde\lambda_h=\ha_{h,n}=\ha_{h+1,k-1}=\ha^-_{h+1,k-1}+1$, we have
\begin{equation}\label{SDPzk}
\aligned
&\quad c_{{\tilde{\lambda}_{h}}}\TDAM=c_{\ha^-_{h+1,k-1}+1}\TDAM=\TDAM c_{\widetilde a^-_{h,k}+1}=\TDAM c_{\widetilde a_{h,k}}, \;\;\text{ and }\\
z_k &=
	 x_{{{\lambda}_{(h)}^{-}}} c_{{\tilde{\lambda}_{h}}}
	   \TDAM
	 \TDIJ(\widetilde a-1,\widetilde a-a+1)
	c_{\Ad}
	\Sigma_A
=%\sum_{{1\leq k\leq n}\atop {a_{h,k}>0}}
	 x_{{{\lambda}_{(h)}^{-}}}
	   \TDAM c_{\tilde{a}_{h,k}}
	 \TDIJ(\widetilde a-1,\widetilde a-a+1)
	c_{\Ad}
	\Sigma_A.
\endaligned
\end{equation}
Hence, by Lemma \ref{SDP-FA},
\begin{equation}\label{A.1}
\begin{aligned}
z=\sum_{1\leq k\leq n \atop a_{h,k}>0}z_k&= \sum_{1\leq k\leq n \atop a_{h,k}>0}\Big\{
	{(-1)}^{{\SOE{\widetilde{a}}}_{h-1,k} + \SOE{a}_{h,k}}	T_{(\SE{A} - E_{h,k}| \SO{A}+ E_{h+1, k})} \\
	& \qquad +
	{(-1)}^{{\SOE{\widetilde{a}}}_{h-1,k} + \SOE{a}_{h,k}+1}
	{q}^{-1}  \STEPPDR{a_{h+1, k} +1}
	T_{(\SE{A} - E_{h,k} + 2E_{h+1, k} | \SO{A}  -E_{h+1, k})} \\
	& \qquad +
	{(-1)}^{{\SOE{\widetilde{a}}}_{h-1,k}+1} {q}^{a_{h,k}-1}
	\STEP{ \SEE{a}_{h+1,k}+1}
	T_{(\SE{A} + E_{h+1,k} | \SO{A}  - E_{h,k} )}
\Big\}=\sdpHf(x_\mu),
\end{aligned}
\end{equation}
proving (1).

We now prove the general case (2).

By applying Lemma \ref{xcT} {\rm(1)} to $x_{{{\lambda}_{(h)}^{-}}} c_{\tilde{\lambda}_{h}}
	T_{\tilde{\lambda}_{h}}T_{\tilde{\lambda}_{h}+1}
	\cdots T_{\tilde{\lambda}_{h}+\AK(h+1,k)-1}$, \eqref{642} becomes
\begin{equation}
\begin{aligned}	
z_k &=
% \sum_{{1\leq k\leq n}\atop {a_{h,k}>0}}
	\Big(
		{q}^{\AK(h+1,k) }  x_{{{\lambda}_{(h)}^{-}}}  T_{\tilde{\lambda}_{h}}^{-1} T_{{\tilde{\lambda}_{h}}+1}^{-1} T_{{\tilde{\lambda}_{h}}+2}^{-1} \cdots T_{{\tilde{\lambda}_{h}}+\AK(h+1,k)-1}^{-1}  c_{{\tilde{\lambda}_{h}}+\AK(h+1,k) }\\
		& \qquad  \qquad + ({q} - 1)  {q}^{\AK(h+1,k)-1}    x_{{{\lambda}_{(h)}^{-}}}
			\big( c_{\tilde{\lambda}_{h}} +  T_{\tilde{\lambda}_{h}}^{-1}c_{{\tilde{\lambda}_{h}}+1} + T_{\tilde{\lambda}_{h}}^{-1} T_{{\tilde{\lambda}_{h}}+1}^{-1} c_{{\tilde{\lambda}_{h}}+2}  \\
		&\quad\;+ \cdots
			+  T_{\tilde{\lambda}_{h}}^{-1} T_{{\tilde{\lambda}_{h}}+1}^{-1} \cdots  T_{{\tilde{\lambda}_{h}}+\AK(h+1,k)-2}^{-1} c_{{\tilde{\lambda}_{h}}+\AK(h+1,k)-1}\big)
	\Big)  \cdot
	   \TDAM
	 \TDIJ(\widetilde a-1,\widetilde a-a+1)
	c_{\Ad}
	\Sigma_A\\
&=\Big(x_{{{\lambda}_{(h)}^{-}}}  c_{{\tilde{\lambda}_{h}}+\AK(h+1,k) }+
({q} - 1) x_{{{\lambda}_{(h)}^{-}}}
			({q}^{\AK(h+1,k)-1} c_{\tilde{\lambda}_{h}} + {q}^{\AK(h+1,k)-2}c_{{\tilde{\lambda}_{h}}+1} +{q}^{\AK(h+1,k)-3} c_{{\tilde{\lambda}_{h}}+2}  \\
		& \qquad  \qquad \qquad + \cdots
			+  c_{{\tilde{\lambda}_{h}}+\AK(h+1,k)-1})\Big)\cdot
	   \TDAM
	 \TDIJ(\widetilde a-1,\widetilde a-a+1)
	c_{\Ad}
	\Sigma_A\\
&=\Big(x_{{{\lambda}_{(h)}^{-}}}  c_{\ha_{h+1,k-1}}+
({q} - 1) x_{{{\lambda}_{(h)}^{-}}}c_{q,\tilde\lambda_h,\ha_{h+1,k-1}-1}
			\Big)\cdot
	   \TDAM
	 \TDIJ(\widetilde a-1,\widetilde a-a+1)
	c_{\Ad}
	\Sigma_A=Z_k+N_k,
\end{aligned}
\end{equation}
where
\begin{equation}\label{F_A2}
\begin{aligned}
&Z_k= x_{{{\lambda}_{(h)}^{-}}} c_{\ha_{h+1,k-1}}  \TDAM \TDIJ(\widetilde a-1,\widetilde a-a+1)
	c_{\Ad}
	\Sigma_A\\
&N_k=(q-1)
			x_{{{\lambda}_{(h)}^{-}}}c_{q,\tilde\lambda_h,\ha_{h+1,k-1}-1}
			%( {q}^{\AK(h+1,k)-1}  c_{\tilde{\lambda}_{h}} +  {q}^{\AK(h+1,k)-2}  c_{{\tilde{\lambda}_{h}}+1}\cdots +  c_{{\tilde{\lambda}_{h}}+\AK(h+1,k)-1})
		\TDAM
	\TDIJ(\widetilde a-1,\widetilde a-a+1)
	c_{\Ad}
	\Sigma_A.
\end{aligned}
\end{equation}

Consider $Z_k$ first.
Since $\ha_{h+1,k-1}=\ha^-_{h+1,k-1}+1$ and $\widetilde a_{h,k}=\widetilde a^-_{h,k}+1$, by Lemma \ref{CT2}(2),
 \begin{equation}\label{zk1(1)}
 c_{\ha_{h+1,k-1}}\TDAM= \TDAM c_{\widetilde{a}_{h,k}}+\sum_{w<d_{A_{h,k}^-},j\in[1,r]}f_{w,j}^{\cdots} T_w c_j.
 \end{equation}
Thus, $Z_k= (\diamondsuit)_k+(\clubsuit)_k$, where
$$
\begin{aligned}
(\diamondsuit)_k&:=x_{{{\lambda}_{(h)}^{-}}}\TDAM c_{\widetilde{a}_{h,k}}\TDIJ(\widetilde a-1,\widetilde a-a+1)
	c_{\Ad}
	\Sigma_A,\\
(\clubsuit)_k&:=
	x_{{{\lambda}_{(h)}^{-}}} \Big(\sum_{w<d_{A_{h,k}^-},j\in[1,r]}f_{w,j}^{\cdots} T_w c_j\Big)
	%(\sum_{d<d_{A_{h,k}^-},\gamma\in\mathbb{Z}^r_2}g_{\gamma,d}T_d c^\gamma)
	\TDIJ(\widetilde a-1,\widetilde a-a+1)
	c_{\Ad}
	\Sigma_A.
\end{aligned}
$$
Note that, according to \eqref{SDPzk}, $(\diamondsuit)_k$ is the $k$th summand of the head term $\sdpHf(x_{\co(A)})$ in (1).
% equal to the $z_k$ when $A^-_{h,k}$ satisfies SDP condition at $(h,k)$ in assertion (1).

%We now consider $N_k$. Set $A^-_{h,k}=B=(b^{\bar0}_{i,j}|b^{\bar 1}_{i,j})$ and $b_{i,j}=b^{\bar0}_{i,j}+b^{\bar1}_{i,j}$. then $b_{h,k}=a_{h,k}-1$, $b_{h+1,k}=a_{h+1,k}+1$, and $b_{i,j}=a_{i,j}$ for others.

%Let ${\lambda'}=\lambda^-_{(h)}=\ro(A^-_{h,k})$, so $\lambda'_{h}=\lambda_h-1$, $\lambda'_{h+1}=\lambda_{h+1}+1$ and $\tilde{\lambda'}_h=\tilde{\lambda}_h-1$, $\overleftarrow{\bf r}^l_{h+1}(A)=\overleftarrow{\bf r}^l_{h+1}(A^-_{h,k})$ for $l\leq k$, and $\overrightarrow{\bf r}^l_{h+1}(A)+1=\overrightarrow{\bf r}^l_{h+1}(A^-_{h,k})$ for $l<k$.
Suppose $\AK(h+1,k)=a_{h+1,l_1}+a_{h+1,l_2}+
\cdots+a_{h+1,l_s}$, where $a_{h+1,l_i}>0$ and $1\leq l_1<l_2<\cdots <l_s<k$. Then
$$\aligned
c_{q,\widetilde\lambda_h,\ha_{h+1,k-1}-1}&={q}^{\AK(h+1,k)-1} c_{\tilde{\lambda}_{h}} + {q}^{\AK(h+1,k)-2}c_{{\tilde{\lambda}_{h}}+1} +{q}^{\AK(h+1,k)-3} c_{{\tilde{\lambda}_{h}}+2} + \cdots
			+  c_{{\tilde{\lambda}_{h}}+\AK(h+1,k)-1}\\
&=\sum_{i=1}^s{q}^{\AK(h+1,k)-\AK(h+1,l_i+1)}c_{q,\ha^-_{h+1,l_i-1}+1,\ha^-_{h+1,l_i}}.
\endaligned$$
Thus, by \eqref{longSDP2}
$$\aligned
N_k&=(q-1)
			x_{{{\lambda}_{(h)}^{-}}}\Big(\sum_{i=1}^s{q}^{\AK(h+1,k)-\AK(h+1,l_i+1)}c_{q,\ha^-_{h+1,l_i-1}+1,\ha^-_{h+1,l_i}}\Big)		\TDAM
	\TDIJ(\widetilde a-1,\widetilde a-a+1)
	c_{\Ad}
	\Sigma_A\\
	&=(q-1)\sum_{i=1}^s{q}^{\AK(h+1,k)-\AK(h+1,l_i+1)}
			x_{{{\lambda}_{(h)}^{-}}}\Big(\TDAM c_{(q,\widetilde{a}_{h,l_i-1}^-+1,\widetilde{a}_{h+1,l_i}^-)}+\sum_{w<d_{A^-_{h,k}},j\in[1,r]}g^{\cdots}_{w,j}T_w c_j
			%c_{q,\ha^-_{h+1,l_i-1}+1,\ha^-_{h+1,l_i}}		\TDAM
			\Big)
	\TDIJ(\widetilde a-1,\widetilde a-a+1)
	c_{\Ad}
	\Sigma_A\\
&=(q-1)((\diamondsuit\diamondsuit)_k+(\clubsuit\clubsuit)_k),
	\endaligned$$
where (noting $\widetilde{a}_{h,l_i-1}^-=\widetilde{a}_{h,l_i-1}$ and $\widetilde{a}_{h,l_i}^-=\widetilde{a}_{h,l_i}$)
$$
\begin{aligned}
(\diamondsuit\diamondsuit)_k&:=\sum_{1\leq l\leq k-1\atop a_{h+1,l}>0}{q}^{\AK(h+1,k)-\AK(h+1,l+1)}x_{{\lambda}_{(h)}^{-}}\TDAM c_{(q,\tilde{a}_{h,l}+1,\tilde{a}_{h+1,l})}
	\TDIJ(\widetilde a-1,\widetilde a-a+1)
	c_{\Ad}
	\Sigma_A\\
(\clubsuit\clubsuit)_k&:=\sum_{1\leq l\leq k-1\atop a_{h+1,l}>0}\sum_{w<d_{A^-_{h,k}},j\in[1,r]}
{q}^{\AK(h+1,k)-\AK(h+1,l+1)}g^{\cdots}_{w,j}x_{{\lambda}_{(h)}^{-}}T_w c_j
	\TDIJ(\widetilde a-1,\widetilde a-a+1)
	c_{\Ad}
	\Sigma_A\\
	&=\sum_{w<d_{A^-_{h,k}},j\in[1,r]}f^{\cdots}_{w,j}x_{\lambda^-_{(h)}} T_w c_j \TDIJ(\widetilde a-1,\widetilde a-a+1)
	c_{\Ad}
	\Sigma_A.
\end{aligned}
$$
Similar to the computations in \eqref{6.4.10} and \eqref{m^-}, the element $(\diamondsuit\diamondsuit)_k$ has the form
\begin{equation}\label{F_A4}
\begin{aligned}
(\diamondsuit\diamondsuit)_k&=\sum_{1\leq l\leq k-1\atop a_{h+1,l}>0}{q}^{\AK(h+1,k)-\AK(h+1,l+1)}x_{{\lambda}_{(h)}^{-}}\TDAM c_{(q,\tilde{a}_{h,l}+1,\tilde{a}_{h+1,l})}\TDIJ(\widetilde a-1,\widetilde a-a+1)
	c_{\Ad}
	\Sigma_A\\
&=\sum_{1\leq l\leq k-1\atop a_{h+1,l}>0}%\sum_{\sigma\in\mathcal{D}_{\nu(d^{-1}_{A^-_{h,k}})}\cap \mathfrak{S}_{{\lambda}_{(h)}^{-}}}T_\sigma
{q}^{\AK(h+1,k)-\AK(h+1,l+1)}x_{{\lambda}_{(h)}^{-}\backslash{}'\!\nu^-} \TDAM x_{\nu^-}c_{(q,\tilde{a}_{h,l}+1,\tilde{a}_{h+1,l})}\TDIJ(\widetilde a-1,\widetilde a-a+1)
	c_{\Ad}
	\Sigma_A\\
&=\sum_{1\leq l\leq k-1\atop a_{h+1,l}>0}{q}^{\AK(h+1,k)-\AK(h+1,l+1)}x_{{\lambda}_{(h)}^{-}\backslash{}'\!\nu^-}
%\sum_{\sigma\in\mathcal{D}_{\nu(d^{-1}_{A^-_{h,k}})}\cap \mathfrak{S}_{{\lambda}_{(h)}^{-}}}T_\sigma
\TDAM c'_{(q,\tilde{a}_{h,l}+1,\tilde{a}_{h+1,l})}
x_{\nu^-}\TDIJ(\widetilde a-1,\widetilde a-a+1)
	c_{\Ad}
	\Sigma_A,
\end{aligned}
\end{equation}
where $\nu^-={\nu}_{{A}^-_{h,k}}$; see  \eqref{nu+nu-}. By Corollary \ref{fkM},
\begin{equation}\label{F_A3}
\begin{aligned}
x_{\nu^-}\TDIJ(\widetilde a-1,\widetilde a-a+1)
	c_{\Ad}
	\Sigma_A&=
	  \STEP{ \SEE{a}_{h+1, k} +1}
	x_{{\nu}^-}c_{(\SE{A} - E_{h,k} + E_{h+1, k} | \SO{A} )}\Sigma_{A^-_{h,k}}\\
&	 +
	{q}^{a_{h, k} - 1} x_{\nu^-}c_{{(\SE{A}  |\SO{A} - E_{h,k} + E_{h+1,k})}} \Sigma_{A^-_{h,k}}\\	
&+{q}^{ a_{h, k} -2} \STEPPDR{a_{h+1, k} +1}
	x_{\nu^-}c_{(\SE{A} + 2E_{h+1,k} | \SO{A}  - E_{h,k} - E_{h+1,k})}\Sigma_{A^-_{h,k}}.
\end{aligned}
\end{equation}
Let $\Bd$ be one of the following three matrices occurring in \eqref{F_A3}
$$(\SE{A} - E_{h,k} + E_{h+1, k} | \SO{A} ),\; (\SE{A}  |\SO{A} - E_{h,k} + E_{h+1,k}),\; (\SE{A} + 2E_{h+1,k} | \SO{A}  - E_{h,k} - E_{h+1,k}),$$  and write $\Bd=(B^{\bar0}|B^{\bar1})=(b_{i,j}^{\bar0}|b_{i,j}^{\bar1})$. Then
 we have $\lfloor \Bd\rfloor=A^-_{h,k}$, $\nu_{\lfloor \Bd\rfloor}=\nu_{A^-_{h,k}}=\nu^-$, and $c_{\Bd}=c^{\nu_{B^{\bar 1}}}_{\nu^-}$. Note that, for $1\leq l<k$, $$b^{\bar0}_{h+1,l}=a^{\bar0}_{h+1,l}, b^{\bar1}_{h+1,l}=a^{\bar1}_{h+1,l}, \tilde{b}_{h,l}=\tilde{a}_{h,l}, \tilde{b}_{h+1,l}=\tilde{a}_{h+1,l}.$$
Thus,
 \begin{equation}\label{F_A6}
\begin{aligned}
&\quad\,c'_{(q,\tilde{a}_{h,l}+1,\tilde{a}_{h+1,l})}x_{\nu^-}c_{\Bd}
=x_{\nu^-}c_{(q,\tilde{a}_{h,l}+1,\tilde{a}_{h+1,l})}c_{\Bd}
=x_{\nu_{\lfloor \Bd\rfloor}}c_{(q,\tilde{b}_{h,l}+1,\tilde{b}_{h+1,l})}c_{\Bd}\\
&=\begin{cases}(-1)^{\tilde{a}^{\bar1}_{h,l}}x_{\nu_{\lfloor \Bd\rfloor}}c_{(\SEE{B}-E_{h+1,l}|\SOE{B}+E_{h+1,l})}, &\mbox{ if } b^{\bar1}_{h+1,l}=a^{\bar1}_{h+1,l}=0;\\
-(-1)^{\tilde{a}^{\bar1}_{h,l}}{\STEPP{a_{h+1,l}}} x_{\nu_{\lfloor \Bd\rfloor}}c_{(\SEE{B}+E_{h+1,l}|\SOE{B}-E_{h+1,l})}, &\mbox{ if } b^{\bar1}_{h+1,l}=a^{\bar1}_{h+1,l}=1.
\end{cases}
\end{aligned}
\end{equation}
%Since $\sum_{\sigma\in\mathcal{D}_{\nu(d^{-1}_{A^-_{h,k}})}\cap \mathfrak{S}_{{\lambda}_{(h)}^{-}}}T_\sigma \TDAM x_{{\nu}_{ {A}^-_{h,k}}}=x_{{{\lambda}_{(h)}^{-}}}\TDAM $, reminding that ${{\lambda}_{(h)}^{-}}=\lambda^-_{(h)}$ and due
Substituting  \eqref{F_A3} and \eqref{F_A6} into \eqref{F_A4} yields
\begin{equation}\label{F_A5}
\begin{aligned}
&(\diamondsuit\diamondsuit)_k=\sum_{l=1}^{k-1}(-1)^{\tilde{a}^{\bar1}_{h,l}}~\cdot{q}^{\AK(h+1,k)-\AK(h+1,l+1)}\\
	  &\Big\{\STEP{ \SEE{a}_{h+1, k} +1}(T_{(\SE{A} - E_{h,k} + E_{h+1, k} -E_{h+1,l}| \SO{A}+E_{h+1,l})}-{\STEPP{a_{h+1,l}}}T_{(\SE{A} - E_{h,k} + E_{h+1, k} +E_{h+1,l}| \SO{A}-E_{h+1,l})})\\
&+{q}^{a_{h, k} - 1}(T_{{(\SE{A} -E_{h+1,l} |\SO{A} - E_{h,k} + E_{h+1,k}+E_{h+1,l})}}-{\STEPP{a_{h+1,l}}}T_{{(\SE{A} +E_{h+1,l} |\SO{A} - E_{h,k} + E_{h+1,k}-E_{h+1,l})}})\\	
&+{q}^{ a_{h, k} -2} \STEPPDR{a_{h+1, k} +1}(
	T_{(\SE{A} + 2E_{h+1,k}-E_{h+1,l} | \SO{A}  - E_{h,k} - E_{h+1,k}+E_{h+1,l})}\\
&\quad\quad\quad\quad\quad\quad\quad\quad\quad\quad-{\STEPP{a_{h+1,l}}}T_{(\SE{A} + 2E_{h+1,k}+E_{h+1,l} | \SO{A}  - E_{h,k} - E_{h+1,k}-E_{h+1,l})})
\Big\},
\end{aligned}
\end{equation}
which is the $k$th term of $\text{\rm HH}\overline{\textsc f}(x_{\co(A)})$.

%On the other hand, rewrite $$(\clubsuit\clubsuit)_k=\sum_{w<d_{A^-_{h,k}}}f^{\cdots}_{w,j}x_\lambda T_w c_j \TDIJ(\widetildea-1,\widetilde a-a+1)c_{\Ad}\Sigma_A.$$

%According to \cite[Prop. 5.2]{DW2}, $z\in x_{\lambda_h^-}\mathcal{H}^c_r\cap \mathcal{H}^c_r x_\mu$.
Finally, combining above gives
%Summarize the results above
$$z=\sum_{1\leq k\leq n\atop a_{h,k}>0}\left\{(\diamondsuit)_k+(q-1)(\diamondsuit\diamondsuit)_k\right\}+\sum_{{1\leq k\leq n}\atop {a_{h,k}>0}}\left\{(\clubsuit)_k+(q-1)(\clubsuit\clubsuit)_k\right\}.$$
Since $z\in x_{\lambda_{(h)}^-}\HCR\cap \HCR x_\mu$ and $(\diamondsuit)_k$ and $(\diamondsuit\diamondsuit)_k$ are in $x_{\lambda_{(h)}^-}\HCR\cap \HCR x_\mu$, it follows that
$$  \sum_{{1\leq k\leq n}\atop {a_{h,k}>0}}\left\{(\clubsuit)_k+(q-1)(\clubsuit\clubsuit)_k\right\}\in x_{\lambda_h^-}\HCR\cap \HCR x_\mu$$

Since $ \TDIJ(\widetilde a-1,\widetilde a-a+1)\in\mathcal{H}_{\mu,R}$, and so, $ \sum_{\sigma\in\mathcal D_{\nu_A}\cap\mathfrak{S}_\mu}\TDIJ(\widetilde a-1,\widetilde a-a+1)
	c_{\Ad}T_\sigma\in \mathcal{H}^c_{\mu,R}$,
	a similar argument at the end of the proofs of Theorems \ref{phidiag1} and \ref{phiupper1} shows that
$$\sum_{{1\leq k\leq n}\atop {a_{h,k}>0}}\left\{(\clubsuit)_k+(q-1)(\clubsuit\clubsuit)_k\right\}=\sum_{ \Bd\in\MNZ(n,r)\atop \exists k,\lfloor{\Bd}\rfloor\prec  {A}^-_{h,k}}f^{\Fd_{\bar h},\Ad}_{\Bd}T_{\Bd},$$
for some $f^{\Fd_{\bar h},\Ad}_{\Bd}\in R$. This completes the proof of (2).
\end{proof}

\begin{rem}\label{regular}
Let $\boldsymbol{\mathcal Q}(n,r)_{\mathbb C(\up)}=\mathcal Q_{\boldsymbol q}(n,r;\mathbb C(\up))$, where ${\boldsymbol q}=\up^2$ and let ${\bf U}(\mathfrak{q}_n)$ be the quantum queer supergroup over $\mathbb C(\up)$ with generators $E_h,F_h,K_j,K_{\bar n}$ ($h\in[1,n-1],j\in[1,n]$). Then, by \cite{Ol}, there is a superalgebra homomorphism $\varphi: {\bf U}(\mathfrak{q}_n)\to\boldsymbol{\mathcal Q}(n,r)_{\mathbb C(\up)}$. 
Let $e_h=\varphi(E_h)$, $f_h=\varphi(F_h)$, $k_j=\varphi(K_j)$, and $k_{\bar n}=\varphi(K_{\bar n})$. Then, by \cite[Th.9.5]{DW2},  $\boldsymbol{\mathcal Q}(n,r)=\mathcal Q_{\boldsymbol q}(n,r;\mathbb Q(\up))$ is generated by $e_h\phi_{(\lambda|{O})}$, $f_h\phi_{(\lambda|{O})}$, $k_i\phi_{(\lambda|{O})}$, and $k_{\bar n}\phi_{(\lambda|{O})}$, for all $h\in[1,h-1]$, $\lambda\in\Lambda(n,r)$. We will see in \cite{DGLW} that these generators are exactly those $\phi_\Xd$, where $\Xd$ is one of the matrices
$$(\lambda|O),\;\; (\lambda + E_{h, h+1} -E_{h+1, h+1} |O), \;\;(\lambda-E_{h,h}+E_{h+1,h}|O),\;\;
 (\lambda-E_{n,n}|E_{n,n}).$$
 Hence, the multiplication formulas given in Theorem \ref{phiupper0} and Theorem \ref{phidiag1} give the matrix form of the regular representation of  $\boldsymbol{\mathcal Q}(n,r)$.
\end{rem}

\section{Some special cases}
In this section, we establish several special multiplication formulas related to certain generators. We single out them for the convenience when checking the defining relations in \cite{DGLW}. We see below that the tail parts in Theorems \ref{phiupper1} and \ref{philower1} are computable if $A$ is simple enough.

%Recall the notation in Remark \ref{rem_phi_short_0}.
\begin{prop}\label{phiupper2}
	For any given {$\mu \in \CMN(n,r-1)$}, {$h \in [1,n-1]$}, the following formulas hold in {$\QqnrR $}.
	\begin{align*}
{\rm {(1)}}& \quad
\phi_{(\mu | E_{h, h+1})}
\phi_{({\mu}+  E_{h+1, h}|O)}
= 			\phi_{({\mu}-E_{h+1, h+1}   + E_{h+1, h}| E_{h, h+1})}
	+	 {q}^{{\mu}_{h+1}}  \phi_{( {\mu}   | E_{h,h})} \\
	& \qquad\qquad\qquad\qquad\qquad\qquad
	- 	({q}-1) \STEP{{\mu}_{h}+1}	\phi_{( {\mu}-E_{h+1, h+1}   + E_{h, h}  | E_{h+1,h+1})}, \\
{\rm {(2)}}& \quad
\phi_{({\mu}| E_{h, h+1})}
\phi_{({\mu} | E_{h+1, h})}
=
-	\STEP{{\mu}_{h} + 1} {q}^{ {\mu}_{h+1} }
		\phi_{( {\mu}  +  E_{h,h}|{O})}  - 	\phi_{({\mu}-E_{h+1, h+1}   | E_{h, h+1} + E_{h+1, h} )}\\
	& \qquad\qquad\qquad\qquad\qquad\qquad
	+ 	({q}-1) 	\phi_{( {\mu}-E_{h+1, h+1}    | E_{h,h} + E_{h+1,h+1})} \\
{\rm {(3)}}& \quad
\phi_{({\mu}| E_{h+1, h})}
\phi_{({\mu}+  E_{h, h+1}|O)}
=  \phi_{({\mu} - E_{h,h} + E_{h, h+1}| E_{h+1, h})}
+\phi_{( {\mu} | E_{h+1, h+1})}, \\
{\rm (4)}& \quad
 \phi_{({\mu}| E_{h+1, h})}
\phi_{({\mu} | E_{h, h+1})}
= \phi_{({\mu} - E_{h,h} |   E_{h, h+1} +  E_{h+1, h})}
	- \STEP{ {\mu}_{h+1} + 1 } \phi_{( {\mu} + E_{h+1, h+1} | O)}.
\end{align*}
\end{prop}
\begin{proof} Parts (3) and (4) clearly follow from Theorem \ref{philower1}(1) as the matrix $\diag(\mu)+E_{h,h+1}$ satisfies SDP condition on the $h$th row by Theorem \ref{CdA1}.
We now prove parts (1) and (2).

Let $A^\star=  ({\SEE{a}_{i,j}} | {\SOE{a}_{i,j}})  =(\mu+E_{h+1,h}|{O})$
or $(\mu|E_{h+1,h})$ which have the same base matrix
$$A=\lfloor A^\star\rfloor=\diag(\mu)+E_{h+1,h}=(a_{i,j}).$$ Thus, $\lambda:=\ro(A)=\mu+\epsilon_{h+1}$. Hence, $\Xd=(\mu|E_{h,h+1})=(\lambda-E_{h+1,h+1}|E_{h,h+1})=\Ed_{\bar h}$ in both cases. This means that cases (1) and (2) are special cases of Theorem \ref{phiupper1}(2) and $\ro(X)=\ro(\lfloor X^\star\rfloor)=\lambda^+_{(h)},$ and   {$c_{X^\star} = c_{\widetilde{\lambda}_{h}+1}$}. Since $A=\diag(\lambda)+E_{h+1,h}-E_{h+1,h+1}$, it follows that
\begin{equation}\label{eq:T-A-special}
\aligned
\nu_A&=(\lambda_1, 0,\cdots,0,\lambda_h, 0,\cdots,0,1,\lambda_{h+1}-1,0,\cdots),\\
\co(A)&=(\lambda_1, \cdots\cdots,\lambda_h+1,\lambda_{h+1}-1,\cdots,\lambda_n),\\
\D_{{\nu}_{_A}} \cap \fS_{\co(A)}&=\{1,s_{\widetilde\la_{h}},s_{\widetilde\la_{h}}s_{\widetilde\la_{h}-1}\cdots,s_{\widetilde\la_h}s_{\widetilde\la_{h}-1}\cdots s_{\widetilde\la_{h-1}+1}\},\\
\Sigma_A&\big(=\Sigma_{v\in\D_{{\nu}_{_A}} \cap \fS_{\co(A)}} T_v\big)
= \TDIJ({\widetilde{\lambda}_{h}},  {\widetilde{\lambda}_{h-1}+1}).
\endaligned
\end{equation}
By Example \ref{lem_dA1} and \eqref{eq cA},
 we also have
 \begin{equation}\label{choice cA}
 d_A =1,\qquad
 c_{A^\star} =\begin{cases} 1,&\text{ if }\Ad=({\mu}+  E_{h+1, h}|{O});\\ c_{\widetilde{\lambda}_{h} + 1},&\text{ if }\Ad=({\mu}|  E_{h+1, h}).\end{cases}
 \end{equation}
 Since, by Theorem \ref{CdA1}, $A^+_{h,h+1}$ does not satisfy the SDP condition at $(h,h+1)$ when $\mu_{h+1}>0$ (see footnote 3), we must compute the tail term explicitly in Theorem \ref{phiupper1}(2).

By \eqref{z elt2}, we have, for $\AK(h,k)=\AK(h,k)(A)$,%and  with $B=X$ the equation \eqref{eq_gB} becomes
\begin{align}\label{eq:z-zk-special}
	z=\sum_{1\leq k\leq n\atop a_{h+1,k}>0} z_k,
	\mbox{ where }
	 z_k&=
		\sum_{j=\AK(h+1,k)}^{\AK(h+1,k+1) - 1}
		x_{\lambda^+_{(h)}}  c_{\widetilde{\lambda}_{h}+1}
		T_{\widetilde{\lambda}_{h}+1} T_{\widetilde{\lambda}_{h}+2} \cdots T_{\widetilde{\lambda}_{h}+j}
		{T_{d_A}} c_{\Ad}\Sigma_A.
		%\sum _{{\sigma} \in \D_{{\nu}_{_A}} \cap \fS_{\mu}} T_{\sigma} .
\end{align}
Further, since {$a_{h+1,k} = 0$}, for all $k\neq h, h+1$, it follows that $z=z_h+z_{h+1}$, where
$$z_{h}=		x_{\lambda^+_{(h)}}  c_{\widetilde{\lambda}_{h}+1}
		c_{A^\star}\Sigma_A=		x_{\lambda^+_{(h)}}  c_{\widetilde{\lambda}_{h}+1}
		c_{A^\star}
		\TDIJ( {\widetilde{\lambda}_{h}} , {\widetilde{\lambda}_{h} -  ({\lambda}_{h} - 1)}),\;\text{  as $d_A=1, a_{h+1,h} = 1$,}$$
 and, as $\AK(h+1,h+1) = 1, \AK(h+1,h+2) = \AK(h+1,h+1) + a_{h+1, h+1} = {\lambda}_{h+1}$, and $d_A=1$,
$$\aligned
z_{h+1}&=		x_{\lambda^+_{(h)}}  c_{\widetilde{\lambda}_{h}+1}\Big(\sum_{j=\AK(h+1,h+1)}^{\AK(h+1,h+2) - 1}
		T_{\widetilde{\lambda}_{h}+1} T_{\widetilde{\lambda}_{h}+2} \cdots T_{\widetilde{\lambda}_{h}+j}\Big)
		{T_{d_A}} c_{\Ad}\Sigma_A\\
	&=x_{\lambda^+_{(h)}}  c_{\tilde{\lambda}_{h}+1}T_{\tilde{\lambda}_{h}+1}
		\TAIJ( {\tilde{\lambda}_{h}+2} , {\tilde{\lambda}_{h}+ {\lambda}_{h+1}- 1}  )
		 c_{A^{\star}}\Sigma_A.
\endaligned$$
%For $k=h$, since {$a_{h+1,h} = 1$}, by \eqref{eq:z-zk-special} and
Now, depending on the choice of $c_\Ad$ in \eqref{choice cA}, Lemma \ref{xcT}(3) together with \eqref{eq cA} implies
\begin{align}
z_{h}
=&
\begin{cases}
(1)\;x_{\lambda^+_{(h)}}  c_{{q}, \tilde{\lambda}_{h-1}+1, \tilde{\lambda}_{h}+1} =T_{(\lambda-E_{h+1,h+1}|E_{h,h})}, &\text{ if } A^\star=(\mu+E_{h+1,h}|O),\\
(2)\; -x_{\lambda^+_{(h)}}
		\TDIJ( {\tilde{\lambda}_h} , {\tilde{\lambda}_h -  (\lambda_h - 1)})= -	\STEP{{\lambda}_{h} + 1} T_{( {\lambda}-E_{h+1,h+1}   +  E_{h,h}| O)},&\text{ if }A^\star=(\mu|E_{h+1,h}).\end{cases}
\label{eq:zh-special}
\end{align}
It remains to compute $z_{h+1}$.
Note that $c_{\tilde{\lambda}_{h}+1}T_{\tilde{\lambda}_{h}+1}=T_{\tilde{\lambda}_{h}+1}  c_{\tilde{\lambda}_{h}+2} + ({q}-1)(c_{\tilde{\lambda}_{h}+1} -  c_{\tilde{\lambda}_{h}+2})$, by \eqref{Hecke-Cliff}.
\begin{equation}\label{z_{h+1}}
\aligned
z_{h+1}&=
		x_{\lambda^+_{(h)}}  (T_{\tilde{\lambda}_{h}+1}  c_{\tilde{\lambda}_{h}+2} + ({q}-1)(c_{\tilde{\lambda}_{h}+1} -  c_{\tilde{\lambda}_{h}+2}))
		\TAIJ( {\tilde{\lambda}_{h}+2} , {\tilde{\lambda}_{h}+ {\lambda}_{h+1}- 1}  )
		 c_{A^{\star}}\Sigma_A\\
%		\sum _{{\sigma} \in \D_{{\nu}_{_A}} \cap \fS_{\mu}} T_{\sigma} \notag\\
&=
		x_{\lambda^+_{(h)}}  T_{\tilde{\lambda}_{h}+1}  c_{\tilde{\lambda}_{h}+2}
		\TAIJ({\tilde{\lambda}_{h}+2} ,  {\tilde{\lambda}_{h}+ {\lambda}_{h+1}- 1}  )
		 c_{A^{\star}}\Sigma_A\\
%		\sum _{{\sigma} \in \D_{{\nu}_{_A}} \cap \fS_{\mu}} T_{\sigma}\notag \\
&\qquad
	+	
		({q}-1)x_{\lambda^+_{(h)}}
		\TAIJ({\tilde{\lambda}_{h}+2} ,  {\tilde{\lambda}_{h}+ {\lambda}_{h+1}- 1}  )c_{\tilde{\lambda}_{h}+1}
		 c_{A^{\star}}\Sigma_A\\
%		\sum _{{\sigma} \in \D_{{\nu}_{_A}} \cap \fS_{\mu}} T_{\sigma}\notag \\
&\qquad
		-
		({q}-1) x_{\lambda^+_{(h)}}  c_{\tilde{\lambda}_{h}+2}
		\TAIJ({\tilde{\lambda}_{h}+2} , {\tilde{\lambda}_{h}+ {\lambda}_{h+1}- 1}  	)
		 c_{A^{\star}}\Sigma_A.
		%\sum _{{\sigma} \in \D_{{\nu}_{_A}} \cap \fS_{\mu}} T_{\sigma}. \label{eq:zh+1-special}
\endaligned
\end{equation}
%due to the fact $c_{\tilde{\lambda}_{h}+1}\TAIJ({\tilde{\lambda}_{h}+2} ,  {\tilde{\lambda}_{h}+ {\lambda}_{h+1}- 1}  )=
%		\TAIJ({\tilde{\lambda}_{h}+2} ,  {\tilde{\lambda}_{h}+ {\lambda}_{h+1}- 1}  )c_{\tilde{\lambda}_{h}+1}$.
		
If $A^\star=(\mu+E_{h+1,h}|O)$, then  {$c_{A^\star} = 1$} and
\begin{align}\label{eq:T-sigma-C}
	\TAIJ( {\widetilde{\lambda}_{h}+2} ,  {\widetilde{\lambda}_{h}+ {\lambda}_{h+1}-1}  )\Sigma_A
%\sum _{{\sigma} \in \D_{{\nu}_{_A}} \cap \fS_{\mu}} T_{\sigma}
&=\TAIJ({\widetilde{\lambda}_{h}+2} ,  {\widetilde{\lambda}_{h}+ {\lambda}_{h+1}-1}  	)
	 \TDIJ( {\widetilde{\lambda}_{h}},  {\widetilde{\lambda}_{h} -  ({\lambda}_{h} - 1)} )
=\sum _{{\sigma} \in \D_{{\nu}_{_C}} \cap \fS_{\co(C)}} T_{\sigma}=\Sigma_C,
\end{align}
where {$C = {\lambda}-2E_{h+1, h+1} + E_{h, h+1} + E_{h+1, h}$} with $\nu_C=(\cdots,\lambda_h, 1,0,\cdots,0,1,\lambda_{h+1}-2,\cdots)$. (Compare $\nu_C$ with $\co(A)$ in \eqref{eq:T-A-special}.)
%Direct calculation using ??? shows {$d_C = s_{\widetilde{\mu}_{h} + 1}$}, and
Hence, by %\eqref{eq:T-A-special}, \eqref{eq:T-sigma-C}, \eqref{eq:zh+1-special} and
Lemmas \ref{xTinverse} and \ref{xcT}(2)\&(3), we obtain
\begin{align*}
z_{h+1}
&=
		x_{\lambda^+_{(h)}}  T_{\tilde{\lambda}_{h}+1}  c_{\tilde{\lambda}_{h}+2}\Sigma_C
		%\sum _{{\sigma} \in \D_{{\nu}_{_C}} \cap \fS_{\mu}} T_{\sigma}
		 +	
		  ({q}-1) \STEP{ {\lambda}_{h+1} -1 } x_{\lambda^+_{(h)}}
		c_{\tilde{\lambda}_{h}+1}
		 \TDIJ( {\tilde{\lambda}_h} , {\tilde{\lambda}_h -  (\lambda_h - 1)} 	)
		 \\
&\qquad
		-
		({q}-1) x_{\lambda^+_{(h)}} c_{{q}, \tilde{\lambda}_{h}+2, \tilde{\lambda}_{h+1}}
		\TDIJ( {\tilde{\lambda}_h}, {\tilde{\lambda}_h -  (\lambda_h - 1)} )
		\\
%&=	x_{\lambda^+_{(h)}}  T_{\tilde{\lambda}_{h}+1}  c_{\tilde{\lambda}_{h}+2}\Sigma_C
		%\sum _{{\sigma} \in \D_{{\nu}_{_C}} \cap \fS_{\mu}} T_{\sigma}  \\
%	+	({q}-1) \STEP{ {\lambda}_{h+1} -1 } x_{\lambda^+_{(h)}}
%		c_{{q}, \tilde{\lambda}_{h-1}+1, \tilde{\lambda}_{h}+1}		 \\
%&\qquad-({q}-1) x_{\lambda^+_{(h)}}
%		\TDIJ({\tilde{\lambda}_h} , {\tilde{\lambda}_h -  (\lambda_h - 1)} )c_{{q}, \tilde{\lambda}_{h}+2, \tilde{\lambda}_{h+1}}		\\
&=
		x_{\lambda^+_{(h)}}  T_{\tilde{\lambda}_{h}+1}  c_{\tilde{\lambda}_{h}+2}\Sigma_C
		%\sum _{{\sigma} \in \D_{{\nu}_{_C}} \cap \fS_{\mu}} T_{\sigma}
	+	
		  ({q}-1) \STEP{ {\lambda}_{h+1} -1 } x_{\lambda^+_{(h)}}
		c_{{q}, \tilde{\lambda}_{h-1}+1, \tilde{\lambda}_{h}+1}
		 \\
&\qquad
		-
		({q}-1) \STEP{{\lambda}_{h}+1}
		x_{\lambda^+_{(h)}}
		c_{{q}, \tilde{\lambda}_{h}+2, \tilde{\lambda}_{h+1}}
		\\
&=
		T_{({\lambda} - 2 E_{h+1, h+1}   + E_{h+1, h}| E_{h, h+1})}
	+	
		  ({q}-1) \STEP{ {\lambda}_{h+1} -1 } T_{ ({\lambda}- E_{h+1, h+1}  | E_{h,h})}
		 \\
&\qquad
		-
		({q}-1) \STEP{{\lambda}_{h}+1}
		T_{( {\lambda}- 2 E_{h+1, h+1}   + E_{h, h}   | E_{h+1,h+1})}.
\end{align*}
This together with \eqref{eq:zh-special}(1), after replacing $\lambda-E_{h+1,h+1}$ by $\mu$, proves (1).

If $A=(\mu|E_{h+1,h})$, then {$c_{A^\star} = c_{\widetilde{\lambda}_{h}+1}$} and, hence, $\TAIJ({\tilde{\lambda}_{h}+2} , {\tilde{\lambda}_{h}+ {\lambda}_{h+1}- 1}  	)
		c_{A^\star}=c_{A^\star}\TAIJ({\tilde{\lambda}_{h}+2} , {\tilde{\lambda}_{h}+ {\lambda}_{h+1}- 1})$. Then, with a similar argument, \eqref{z_{h+1}} becomes
\begin{align*}
z_{h+1}
&=
		- x_{\lambda^+_{(h)}}  T_{\tilde{\lambda}_{h}+1}
		c_{\tilde{\lambda}_{h}+1} c_{\tilde{\lambda}_{h}+2}
		\TAIJ({\tilde{\lambda}_{h}+2},  {\tilde{\lambda}_{h}+ {\lambda}_{h+1}- 1}  	)\Sigma_A
		- ({q}-1)x_{\lambda^+_{(h)}}
		\TAIJ({\tilde{\lambda}_{h}+2},  {\tilde{\lambda}_{h}+ {\lambda}_{h+1}- 1}  	)
		\TDIJ({\tilde{\lambda}_h}, {\tilde{\lambda}_h -  (\lambda_h - 1)}) \\
&\qquad
		-({q}-1)
		x_{\lambda^+_{(h)}}  c_{\tilde{\lambda}_{h}+2}
		\TAIJ({\tilde{\lambda}_{h}+2} , {\tilde{\lambda}_{h}+ {\lambda}_{h+1}- 1}  	)
		c_{\tilde{\lambda}_{h}+1}
		\TDIJ( {\tilde{\lambda}_h}, {\tilde{\lambda}_h -  (\lambda_h - 1)}) \\
&=
		-x_{\lambda^+_{(h)}}  T_{\tilde{\lambda}_{h}+1}
		c_{\tilde{\lambda}_{h}+1} c_{\tilde{\lambda}_{h}+2}\Sigma_C-  ({q}-1)  x_{\lambda^+_{(h)}} \STEP{ {\lambda}_{h+1} - 1 }  \STEP{ {\lambda}_{h} + 1 } \\
&\qquad
		+
		({q}-1) x_{\lambda^+_{(h)}}
		c_{ {q}, \tilde{\lambda}_{h-1}+1,  \tilde{\lambda}_{h}+1}
		c_{{q}, \tilde{\lambda}_{h}+2, \tilde{\lambda}_{h+1} }
		\\
&=
		-T_{ ({\lambda} - 2 E_{h+1, h+1}   | E_{h, h+1} + E_{h+1, h}) }
		 -  ({q}-1)  \STEP{ {\lambda}_{h+1} - 1 }  \STEP{ {\lambda}_{h} + 1 }
	 	T_{( {\lambda}- E_{h+1, h+1}   + E_{h, h}  | O )}	
		 \\
&\qquad +
		({q}-1) 	T_{( {\lambda}- 2 E_{h+1, h+1}  | E_{h,h} + E_{h+1,h+1} )} .
\end{align*}
This together with \eqref{eq:zh-special}(2), after replacing $\lambda-E_{h+1,h+1}$ by $\mu$, proves (2).
\end{proof}

The following relation gives an example of the fact that a product of two odd elements in the superalgebra $\QqnrR $ is an even element.

\begin{lem}\label{relation_eei}Let $h\in[1,n-1]$ and
 {${\lambda}  \in \CMN(n, r - 1)$}. Then the following relation holds in {$\QqnrR $}:
\begin{equation}\label{oo=ee}
\aligned
 &{\phi}_{( {\lambda}  | E_{h, h+1} )}
		{\phi}_{( {\lambda}  |  E_{h+1, h+2} )}	+ {q}  {\phi}_{( {\lambda}   + \alpha_h | E_{h+1, h+2} )}
		{\phi}_{(  {\lambda} -\alpha_{h+1} | E_{h, h+1}  )}=	 \\
- &{\phi}_{(   {\lambda} + E_{h, h+1} | O)}
		{\phi}_{( {\lambda}  +  E_{h+1, h+2} | O)}		 +    {q}  {\phi}_{( {\lambda}  + \alpha_{h} + E_{h+1, h+2} |O  )}
		{\phi}_{(  {\lambda} -\alpha_{h+1} + E_{h, h+1}  | O)} .
		\endaligned
\end{equation}
\end{lem}
\begin{proof}Observe that each term in \eqref{oo=ee} is a product of the form $\phi_\Ad\phi_\Bd$. For all choices of $\Bd$, the base matrix $B$ have the same $\co(B)=\lambda+\bs{\ep}_{h+2}$. %All $\Ad$'s have the same base matrix $A=\diag(\lambda)+E_{h,h+1}$ and all $\Bd$'s have the base matrix $B=\diag(\lambda)+E_{h+1,h+2}$.
We prove the result by applying both sides to $x_{{\lambda} + \bs{\ep}_{h+2}} $ and check the identity in $\HCR$.

Let
{${A}^{\star}_1 =  ( {\lambda} | E_{h, h+1}) $} and
 {${B}^{\star}_1= ( {\lambda}  |  E_{h+1, h+2})$} with base $A$ and $B$, respectively. Then
\begin{align*}
 {\fcY}_1 &=
		{\phi}_{( {\lambda}  | E_{h, h+1} )}
		\big({\phi}_{( {\lambda} |  E_{h+1, h+2} )}	(x_{{\lambda} + \bs{\ep}_{h+2}})\big )=
		{\phi}_{( {\lambda}  | E_{h, h+1} )}
		( T_{( {\lambda}  |  E_{h+1, h+2}) }  ) \\
&=
		{\phi}_{( {\lambda} | E_{h, h+1} )}
		( x_{\ro({B})}
		T_{d_{B}} c_{{B}^{\star}_1}
		%\sum _{{\sigma} \in \D_{{\nu}_{{B}}} \cap \fS_{\co({B})}} T_{\sigma}
		\Sigma_{B}
		 )=	
		x_{\ro({A})}
		T_{d_{A}} c_{{A}^{\star}_1}
		%\sum _{{\sigma} \in \D_{{\nu}_{{A}}} \cap \fS_{\co({A})}} T_{\sigma}
		\Sigma_{A} \cdot
		T_{d_{B}} c_{{B}^{\star}_1}
		%\sum _{{\sigma} \in \D_{{\nu}_{{B}}} \cap \fS_{\co({B})}} T_{\sigma}
		\Sigma_{B}
		.
\end{align*}
By Example \ref{lem_dA1} and various definitions, it is easy to verify
{$d_{A} = d_{B} = 1$},
{$c_{{A}^{\star}_1} = c_{\widetilde{\lambda}_{h} + 1}$} ,
{$c_{{B}^{\star}_1} = c_{\widetilde{\lambda}_{h+1} + 1}$},
{$\ro({A}) =  {\lambda} + \bs{\ep}_{h}$},
and
\begin{align*}
\Sigma_{A}
=
\TAIJ({\widetilde{\lambda}_{h}+1} , {\widetilde{\lambda}_{h} + {\lambda}_{h+1} }), \qquad
\Sigma_{B}
= \TAIJ( {\widetilde{\lambda}_{h+1} +1} , {\widetilde{\lambda}_{h+1} + {\lambda}_{h+2} } ).
\end{align*}
Hence,
\begin{align*}
 {\fcY}_1
& =
		x_{{\lambda} + \bs{\ep}_{h}}
		c_{\widetilde{\lambda}_{h} + 1}
		\TAIJ({\widetilde{\lambda}_{h}+1} , {\widetilde{\lambda}_{h} + {\lambda}_{h+1} })
		c_{\widetilde{\lambda}_{h+1} + 1}
		 \TAIJ( {\widetilde{\lambda}_{h+1} +1} , {\widetilde{\lambda}_{h+1} + {\lambda}_{h+2} } )
	 .
\end{align*}
Similarly,
set
{$M^{\star}_1 = ( {\lambda} +\alpha_h | E_{h+1, h+2} )$},
{$ N^{\star}_1 = ( {\lambda} -\alpha_{h+1} | E_{h, h+1} ) $}, having base matrices, $M$, $N$, respectively,
then we have
\begin{align*}
\Sigma_M
 =\TAIJ( {\widetilde{\lambda}_{h+1} +1}, {\widetilde{\lambda}_{h+1} + {\lambda}_{h+2} } ) \ , \qquad
\Sigma_N
= \TAIJ( {\widetilde{\lambda}_{h}+1} , {\widetilde{\lambda}_{h} + {\lambda}_{h+1} - 1}) \ ,
\end{align*}
and direct calculation shows
\begin{align*}
 {\fcY}_2
 &=	
		{\phi}_{( {\lambda}  + \bs{\ep}_{h} - \bs{\ep}_{h+1} | E_{h+1, h+2} )}
		\cdot
		{\phi}_{(  {\lambda} + \bs{\ep}_{h+2} -\bs{\ep}_{h+1} | E_{h, h+1}  )}	(x_{{\lambda} + \bs{\ep}_{h+2}})\\
& =
		x_{{\lambda} + \bs{\ep}_{h}}
		c_{\widetilde{\lambda}_{h+1} + 1}
		\Sigma_M \cdot
		c_{\widetilde{\lambda}_{h} + 1}
		\Sigma_N
		   \\
& =
		x_{{\lambda} + \bs{\ep}_{h}}
		c_{\widetilde{\lambda}_{h+1} + 1}
		\TAIJ( {\widetilde{\lambda}_{h+1} +1}, {\widetilde{\lambda}_{h+1} + {\lambda}_{h+2} } )
		c_{\widetilde{\lambda}_{h} + 1}
		\TAIJ( {\widetilde{\lambda}_{h}+1} , {\widetilde{\lambda}_{h} + {\lambda}_{h+1} - 1})
		.
\end{align*}
On the other hand, let
{${{A}}^{\star}_2 =  ( {\lambda} + E_{h, h+1} | O) $},
 {${{B}}^{\star}_2 = ( {\lambda}  + E_{h+1, h+2} | O)$} which have the same base matrices $A$ and $B$ as $A^\star_1$ and $B^\star_1$.
Using the proof of Proposition \ref{phiupper0}(2),
a straightforward calculation implies
\begin{align*}
 {\fcY}_3
&=
		{\phi}_{(   {\lambda} + E_{h, h+1} | O)}
		\big(
		{\phi}_{( {\lambda}  +  E_{h+1, h+2} | O)}	(x_{{\lambda} + \bs{\ep}_{h+2}} )\big ) =
		{\phi}_{( {\lambda} + E_{h, h+1} | O)}
		 (T_{(  {\lambda}  +  E_{h+1, h+2} | O ) } )   \\
&=
		{\phi}_{( \lambda+  E_{h, h+1}| O )}
		( x_{\ro({{B}})}
		T_{d_{{B}}} c_{{{B}}^{\star}_2}
		\Sigma_{{B}} 		 )=
		x_{\ro({{A}})}
		T_{d_{{A}}} c_{{{A}}^{\star}_2}
		\Sigma_{{A}}  \cdot
		T_{d_{{B}}} c_{{{B}}^{\star}_2}
		\Sigma_{{B}}
		 \\
& =
		x_{{\lambda} + \bs{\ep}_{h}}
		\TAIJ({\widetilde{\lambda}_{h}+1} , {\widetilde{\lambda}_{h} + {\lambda}_{h+1} })
		 \TAIJ( {\widetilde{\lambda}_{h+1} +1} , {\widetilde{\lambda}_{h+1} + {\lambda}_{h+2} } )
		.
\end{align*}
Similarly, let
{$M^{\star}_2 = ( {\lambda} + \bs{\ep}_{h} - \bs{\ep}_{h+1}  + E_{h+1, h+2} | O)$},
{$ {N}^{\star}_2 = ( {\lambda} + \bs{\ep}_{h+2} -\bs{\ep}_{h+1} + E_{h, h+1}  | O) $},
and we have
\begin{align*}
 {\fcY}_4
 &=	
		{\phi}_{( {\lambda} + \bs{\ep}_{h} - \bs{\ep}_{h+1} + E_{h+1, h+2} |O  )}
		\big(
		{\phi}_{(  {\lambda} + \bs{\ep}_{h+2} -\bs{\ep}_{h+1} + E_{h, h+1}  | O)}	(x_{{\lambda} + \bs{\ep}_{h+2}} )\big)\\
& =
		x_{{\lambda} + \bs{\ep}_{h}}
		\Sigma_{M}  \cdot
		\Sigma_{N}
		 =
		x_{{\lambda} + \bs{\ep}_{h}}
		\TAIJ( {\widetilde{\lambda}_{h+1} +1}, {\widetilde{\lambda}_{h+1} + {\lambda}_{h+2}} )
		\TAIJ( {\widetilde{\lambda}_{h}+1} , {\widetilde{\lambda}_{h} + {\lambda}_{h+1} - 1})
		 .
\end{align*}
\iffalse
To prove
{$
  {\fcY}_1 +   {q} {\fcY}_2 = - {\fcY}_3 +  {q}  {\fcY}_4
  $},
where
\begin{align*}
{\fcY}_1
&=
		 x_{({\lambda} + \bs{\ep}_{h})}
		c_{\widetilde{\lambda}_{h} + 1}
		 \TAIJ( {\widetilde{\lambda}_{h}+1} , {\widetilde{\lambda}_{h} + {\lambda}_{h+1} })
		c_{\widetilde{\lambda}_{h+1} + 1}
		 \TAIJ( {\widetilde{\lambda}_{h+1} +1}, {\widetilde{\lambda}_{h+1} + {\lambda}_{h+2} + 1 - 1}) , \\
  {q}  {\fcY}_2
&=
		  {q}
		x_{({\lambda} + \bs{\ep}_{h})}
		c_{\widetilde{\lambda}_{h+1} + 1}
		 \TAIJ( {\widetilde{\lambda}_{h+1} +1}, {\widetilde{\lambda}_{h+1} + {\lambda}_{h+2} + 1 - 1})
		c_{\widetilde{\lambda}_{h} + 1}
		 \TAIJ( {\widetilde{\lambda}_{h}+1}, {\widetilde{\lambda}_{h} + {\lambda}_{h+1} - 1}	), \\
{\fcY}_3
&=
		x_{({\lambda} + \bs{\ep}_{h})}
		 \TAIJ(  {\widetilde{\lambda}_{h}+1} , {\widetilde{\lambda}_{h} + {\lambda}_{h+1} }	)
		 \TAIJ( {\widetilde{\lambda}_{h+1} +1} , {\widetilde{\lambda}_{h+1} + {\lambda}_{h+2} + 1 - 1}) , \\
  {q} {\fcY}_4
&=
		x_{({\lambda} + \bs{\ep}_{h})}
		 \TAIJ( {\widetilde{\lambda}_{h+1} +1} , {\widetilde{\lambda}_{h+1} + {\lambda}_{h+2} + 1 - 1} )
		 \TAIJ( {\widetilde{\lambda}_{h}+1} , {\widetilde{\lambda}_{h} + {\lambda}_{h+1} - 1}	).
\end{align*}
\fi
Let {$ \widetilde{a} = \widetilde{\lambda}_{h}$}, {$ {b} = {\lambda}_{h+1}$},  {$\widetilde c = \widetilde{\lambda}_{h+2} $}. Then,
 by Lemma \ref{xcT}(2),
\begin{align*}
{\fcY}_1
&=
		  x_{{\lambda} + \bs{\ep}_{h}}  c_{{\widetilde{a}} + 1} \TAIJ( {\widetilde{a}}+1, {\widetilde{a}} + {b})  c_{{\widetilde{a}} + {b} + 1} \TAIJ({\widetilde{a}} + {b} +1, \widetilde c) \\
&=
	 {q}^{b}x_{{\lambda} + \bs{\ep}_{h}}
	( c_{{\widetilde{a}}+1} + T_{{\widetilde{a}}+1}^{-1} c_{{\widetilde{a}}+2} +  T_{{\widetilde{a}}+1}^{-1} T_{{\widetilde{a}}+2}^{-1}   c_{{\widetilde{a}}+3}
		+ \cdots
		+ T_{{\widetilde{a}}+1}^{-1}
		%T_{{\widetilde{a}}+2}^{-1} T_{{\widetilde{a}}+3}^{-1}
		\cdots T_{{\widetilde{a}} + {b}}^{-1}  c_{{\widetilde{a}} + {b} +1}
	)  c_{{\widetilde{a}} + {b} + 1} \TAIJ({\widetilde{a}} + {b} +1, \widetilde c), \\
  {q}  {\fcY}_2
&=
	  {q}
	x_{{\lambda} + \bs{\ep}_{h}}
	c_{{\widetilde{a}} + {b} + 1}\TAIJ({\widetilde{a}} + {b} +1, \widetilde{c})
		c_{{\widetilde{a}} + 1}\TAIJ({\widetilde{a}}+1, {\widetilde{a}} + {b}-1) \\
&=
	-{q}^{ b }	x_{{\lambda} + \bs{\ep}_{h}}
	( c_{{\widetilde{a}}+1} + T_{{\widetilde{a}}+1}^{-1} c_{{\widetilde{a}}+2} +  T_{{\widetilde{a}}+1}^{-1} T_{{\widetilde{a}}+2}^{-1}   c_{{\widetilde{a}}+3}
	+ \cdots
	+ T_{{\widetilde{a}}+1}^{-1}
	%T_{{\widetilde{a}}+2}^{-1} T_{{\widetilde{a}}+3}^{-1}
	\cdots T_{{\widetilde{a}} + {b}-1}^{-1}  c_{{\widetilde{a}} + {b}} )
	c_{{\widetilde{a}} + {b} + 1}\TAIJ({\widetilde{a}} + {b} +1, \widetilde c).
\end{align*}
Consequently,
\begin{align*}
&{\fcY}_1 +   {q}  {\fcY}_2\\
&=  	{q}^{b}x_{{\lambda} + \bs{\ep}_{h}} ( c_{{\widetilde{a}}+1} + T_{{\widetilde{a}}+1}^{-1} c_{{\widetilde{a}}+2} +  T_{{\widetilde{a}}+1}^{-1} T_{{\widetilde{a}}+2}^{-1}   c_{{\widetilde{a}}+3}
	+ \cdots
	+ T_{{\widetilde{a}}+1}^{-1}
	%T_{{\widetilde{a}}+2}^{-1} T_{{\widetilde{a}}+3}^{-1}
	\cdots T_{{\widetilde{a}} + {b}}^{-1}  c_{{\widetilde{a}} + {b} +1}  )  c_{{\widetilde{a}} + {b} + 1} \TAIJ({\widetilde{a}} + {b} +1, \widetilde c) \\
	&
	- {q}^{b }	x_{{\lambda} + \bs{\ep}_{h}}
	( c_{{\widetilde{a}}+1} + T_{{\widetilde{a}}+1}^{-1} c_{{\widetilde{a}}+2} +  T_{{\widetilde{a}}+1}^{-1} T_{{\widetilde{a}}+2}^{-1}   c_{{\widetilde{a}}+3}
	+ \cdots
	+ T_{{\widetilde{a}}+1}^{-1}
	%T_{{\widetilde{a}}+2}^{-1} T_{{\widetilde{a}}+3}^{-1}
	\cdots T_{{\widetilde{a}} + {b}-1}^{-1}  c_{{\widetilde{a}} + {b}} )
	c_{{\widetilde{a}} + {b} + 1}\TAIJ({\widetilde{a}} + {b} +1, \widetilde c) \\
&=	- {q}^{  {b}  }x_{{\lambda} + \bs{\ep}_{h}} T_{{\widetilde{a}}+1}^{-1} T_{{\widetilde{a}}+2}^{-1} T_{{\widetilde{a}}+3}^{-1} \cdots T_{{\widetilde{a}} + {b}}^{-1}   \TAIJ({\widetilde{a}} + {b} +1, \widetilde{c}) \\
&=	- {q}^{  {b}  }
		\big(
			{q}^{    - {b}  } x_{{\lambda} + \bs{\ep}_{h}}  T_{ {\widetilde{a}}+1 }   T_{ {\widetilde{a}}+2} \cdots T_{  {\widetilde{a}} + {b} }
		- {q}^{   - {b}  }  ({q} - 1)  x_{{\lambda} + \bs{\ep}_{h}}  \TAIJ( {\widetilde{a}}+1 ,   {\widetilde{a}} + {b}-1 )\big)
		\TAIJ({\widetilde{a}} + {b} +1, \widetilde c)
		 \ \mbox{(by Lemma \ref{xTinverse})}
		\\
&=
	- {q}^{  {b}  }   {q}^{   - {b}} x_{{\lambda} + \bs{\ep}_{h}}T_{ {\widetilde{a}}+1 }   T_{ {\widetilde{a}}+2} \cdots T_{ {\widetilde{a}} + {b} }  \TAIJ({\widetilde{a}} + {b} +1, \widetilde c)
	+{q}^{  {b}  }   {q}^{   - {b} }  ({q} - 1)   x_{{\lambda} + \bs{\ep}_{h}}  \TAIJ( {\widetilde{a}}+1 ,   {\widetilde{a}} + {b}-1 )   \TAIJ({\widetilde{a}} + {b} +1, \widetilde c)
	\\
&=
	-x_{{\lambda} + \bs{\ep}_{h}}    T_{ {\widetilde{a}}+1 }   T_{ {\widetilde{a}}+2} \cdots T_{ {\widetilde{a}} + {b} }  \TAIJ({\widetilde{a}} + {b} +1, \widetilde c)
	+  ({q} - 1)  x_{{\lambda} + \bs{\ep}_{h}}   \TAIJ({\widetilde{a}}+1 ,   {\widetilde{a}} + {b} - 1)  \TAIJ({\widetilde{a}} + {b} +1, \widetilde c).
\end{align*}
On the other hand, we have
\begin{align*}
{\fcY}_3
&=	x_{{\lambda} + \bs{\ep}_{h}} 	\TAIJ({\widetilde{a}}+1, {\widetilde{a}} + {b})	\TAIJ({\widetilde{a}} + {b} +1, \widetilde c), \\
  {q}  {\fcY}_4
&=		 {q} 	x_{{\lambda} + \bs{\ep}_{h}}  \TAIJ({\widetilde{a}} + {b} +1, \widetilde c) \TAIJ({\widetilde{a}}+1, {\widetilde{a}} + {b}-1)
=		{q}	x_{{\lambda} + \bs{\ep}_{h}}  \TAIJ({\widetilde{a}}+1, {\widetilde{a}} + {b}-1) \TAIJ({\widetilde{a}} + {b} +1, \widetilde c),
\end{align*}
and
\begin{align*}
-{\fcY}_3 +  {q}  {\fcY}_4
&=
		- x_{{\lambda} + \bs{\ep}_{h}}  T_{{\widetilde{a}}+1} T_{{\widetilde{a}}+2} \cdots T_{{\widetilde{a}} + {b}}  \TAIJ({\widetilde{a}} + {b} +1, \widetilde c)
		+ ( {q} - 1)x_{{\lambda} + \bs{\ep}_{h}}  \TAIJ({\widetilde{a}}+1, {\widetilde{a}} + {b}-1) \TAIJ({\widetilde{a}} + {b} +1, \widetilde c).
\end{align*}
Hence,
{$  {\fcY}_1 +  {q}  {\fcY}_2 =- {\fcY}_3 +    {q} {\fcY}_4$}, as desired.
\end{proof}

\begin{rems}\label{raw}
%(1) Using Proposition \ref{prop_PhiAPhiB}, one may write down the multiplication formulas for $\phi_\Xd\phi_\Ad$, where $X$ is one of the six matrices given in \eqref{even-odd},
%$X=(\lambda|O), (\lambda-E_{j,j}|E_{j,j})$, $(\lambda-E_{h+1, h+1}+ E_{h, h+1}|O), (\lambda-E_{h,h}+E_{h+1,h}|O), (\lambda-E_{h+1, h+1}| E_{h, h+1} )$ or $(\lambda-E_{h,h}|E_{h+1,h})$
%(together with the required SDP conditions)in the twisted queer $q$-Schur superalgebra.
%(2)
 The multiplication formulas in Theorems \ref{phiupper0}, \ref{phidiag1}, \ref{phiupper1} and \ref{philower1}
are the ``raw'' formulas which are the counterpart of \cite[Lemma 3.2]{BLM} or \cite[Lemma 3.1]{DG}.
 By normalising the $\phi$-basis,
a normalised version of these formulas using symmetric Gaussian polynomials in \cite[Lemma 3.4 and \S4.6]{BLM}
 or \cite[Proposition 4.4 and 4.5]{DG} are used in the construction of the BLM type realisations.
 However, an inspection on the coefficients of these multiplication formulas in Sections 5 and 6 shows that,
for the candidate basis $\{[\Ad]\mid \Ad \in \MNZ(n,r)\}$, setting
{$[\Ad]:={v}^{-l(w^+_{A}) + l(w_{0,\ro(A)})}\phi_\Ad$} is not enough to obtain
a normalised version of these formulas.\footnote{Here $w^+_{A}$ is the longest elements of the double coset
	associated with $A$; see \cite[Lemma 13.10]{DDPW} for a geometric interpretation of $-w^+_{A}+w_{0,ro(A)}$.}
In fact, the factor $c_\Ad$ in defining $\phi_\Ad$ needs to be normalised too.
%  since the coefficients involve not just the entries of $A$. Fortunately, we are able to present a new BLM type realization for ${\bf U}(\mathfrak q_n)$ by simply using these raw multiplication formulas. This requires certain adjustments on generators. 
For more details, see \cite{DGLW}.
%see footnotes \ref{raw1} and \ref{raw2} below.
\end{rems}

\section{An outline of some applications}
We end this paper with an outline of some works in progress.
The first main application is to use the multiplication formulas discovered in \S\S5--7 to give a BLM type realisation for the quantum queer supergroup.

Using a modified approach \cite{DF}, the work in \cite{BLM} can be interpreted as an algebra monomorphism $f_\bsS$ from the quantum linear group
${\bf U}(\mathfrak{gl}_n)$ to the direct product of $q$-Schur algebras $\bsS(n,r):=\sS_q(n,r;\mathbb Q(\up))$ ($q=\up^2$), the top horizontal map in the diagram below  (see Corollary \ref{q-Sch}).
\begin{center}
\begin{tikzpicture}[scale=1.5]
\fill(0,2) node {${\bf U}(\mathfrak{gl}_n)$}; \fill(4,2) node {$\prod_{r\geq0}\bsS(n,r)$};
\fill(2,2.2) node {$f_{\bsS}$}; \fill(2,0.2) node {$\exists f_{\boldsymbol{\mathcal Q}}?$};
\fill(-.15,1) node {$\cap$}; \fill(4.15,1) node {$\cap$};
\fill(0,0) node {${\bf U}(\mathfrak{q}_n)$}; \fill(4,0) node {$\prod_{r\geq0}\widetilde{\boldsymbol{\mathcal Q}}(n,r)$};
\draw[->](0,1.8) -- (0,.2); \draw[->](0.5,2) -- (3.2,2); \draw[->](4,1.8) -- (4,.3);
                                       \draw[->](0.5,0) -- (3.2,0);
\end{tikzpicture}
\end{center}

The image of $f_\bsS$ can be presented in terms of a basis and explicit multiplication formulas as relations.
Since ${\bf U}(\mathfrak{gl}_n)$ is a subalgebra of the quantum queer supergroup ${\bf U}(\mathfrak q_n)$ and, by Corollary \ref{q-Sch}, each $q$-Schur algebra is a subalgebra of the ``twisted'' queer $q$-Schur superalgebra $\widetilde{\boldsymbol{\mathcal Q}}(n,r)$ (see Remark \ref{sHom}), there are two vertical inclusion maps.\footnote{In the diagram, a twisted version $\widetilde{\boldsymbol{\mathcal Q}}(n,r)$ is used in order to make it true over $\mathbb Q$. If we use the complex number field $\mathbb C$, then the twisted version is isomorphic to $\boldsymbol{\mathcal Q}(n,r)$; see \cite{DGLW}.} Thus, it is natural to ask if there exists a superalgebra homomorphism $f_{\boldsymbol{\mathcal Q}}$ such that the above diagram commutes. We will answer this question in the forthcoming paper \cite{DGLW}.

Here we outline the method used in earlier works and point out the degree of difficulty in the queer case.
The fundamental multiplication formulas established in \S\S5--7 are related to the standard basis for queer $q$-Schur superalgebras.
In the non-super case, stabilisation properties of these formulas are established to construct the modified quantum linear groups, i.e., the algebra $\dot{\bf U}$ in \cite{Lubk} associated with quantum $\mathfrak{gl}_n$; see \cite{BLM} and \cite[Part 5]{DDPW} for more details. For the algebra $\dot{\bf U}$ (without 1) of ${\bf U}={\bf U}(\mathfrak{gl}_{m|n})$, a similar work has been done; see \cite{DG, DGZ2}.
Then, by introducing certain ``long elements'' from the standard basis, the ``short'' multiplication formulas can be extended to these long elements. One then uses these long element multiplication formulas to define a subalgebra inside the direct product of $q$-Schur algebra/superalgebras. By further checking the defining relations, one proves that this subalgebra is isomorphic to a quantum linear group/supergroup.

It has been expected for years that such a construction should be generalised to quantum queer supergroup $\mathbf U(\mathfrak q_n)$.
However, deriving the multiplication formulas by the odd generators became the main hurdle. Just as indicated in the theorems in Section 6, there are no such explicit formulas! However, the SDP condition introduced in the paper provides a path for us to overcome the hurdle. The work \cite{DGLW} will complete the task above for quantum queer supergroups. In this way, we also reestablish the Schur--Olshanski duality given by truncating an infinite sum to a finite one as surjective algebra homomorphisms from $\mathbf U(\mathfrak q_n)$ to queer $q$-Schur algebras.

We then can give another important application which establishes the Schur--Olshanski duality at the integral level. Thus, representations of the hyperalgebra of $\mathbf U(\mathfrak q_n)$, especially those polynomial ones, in positive quantum characteristic can be investigated with a connection to the modular representations of Hecke--Clifford superalgebras, generalising some of the classical theories in \cite{BK1, BK2} to the quantum case. For more details, see the forthcoming paper \cite{DWZ}.

\appendix
\section{Some multiplication formulas in $\HCR$}
In this appendix, we prove \eqref{A.1}. Let $(\diamondsuit)_k:=x_{{{\lambda}_{(h)}^{-}}} \Big(\TDAM c_{\tilde{a}_{h,k}}\Big) \TDIJ(\widetilde a_{h,k}-1,\widetilde a_{h,k}-a_{h,k}+1)
	%T_{\tilde{a}_{h,k}-1} \cdots T_{\tilde{a}_{h,k}-q_j}
	c_{\Ad}
	\Sigma_A$.

\begin{lem}\label{SDP-FA}
Let {$h \in [1,n-1]$}  and {$\Ad =  ({\SEE{a}_{i,j}} | {\SOE{a}_{i,j}})   \in \MNZ(n,r)$}. Assume $A=\lfloor\Ad\rfloor$  and
{${\lambda} = \ro(A)$}. %, and set $q_j=j-\BK(h,k)$.
The following formula holds in $\HCR$. \begin{equation}\label{FSDP}
\begin{aligned}
(\diamondsuit)_k&=% \Big\{
	{(-1)}^{{\SOE{\widetilde{a}}}_{h-1,k} + \SOE{a}_{h,k}}  	T_{(\SE{A} - E_{h,k}| \SO{A}+ E_{h+1, k})} \\
	& \qquad +
	{(-1)}^{{\SOE{\widetilde{a}}}_{h-1,k} + \SOE{a}_{h,k}+1}
	{q}^{ -1}  \STEPPDR{a_{h+1, k} +1}
	T_{(\SE{A} - E_{h,k} + 2E_{h+1, k} | \SO{A}  -E_{h+1, k})} \\
	& \qquad +
	{(-1)}^{{\SOE{\widetilde{a}}}_{h-1,k}+1} {q}^{ a_{h,k}-1}
	\STEP{ \SEE{a}_{h+1,k}+1}
	T_{(\SE{A} + E_{h+1,k} | \SO{A}  - E_{h,k} )}.
%\Big\}.
\end{aligned}
\end{equation}
\end{lem}
\begin{proof} Throughout the proof, we keep the notations in the proof of Theorem \ref{philower1} as well as
 Notation \ref{rem_phi_short_0}. In particular, $\nu = \nu_A$, $\nu^- = \nu_{\AM}$ and $\delta^-$ are as in \eqref{nu+nu-} and
 $\tilde{a}=\tilde{a}_{h,k},a=a_{h,k}, b=a_{h+1,k}.$ We also record the following some of which have been (implicitly) used before.
\begin{equation}\label{diampf-2}
\aligned
(1)\;&x_{\nu^-} = x_{\D_{-}} x_{\delta^-} =x_{\delta^-} \TAIJ({\widetilde a},  {\widetilde a+b-1}), \qquad x_{\nu}=x_{\delta^-} \TDIJ({\widetilde a-1},  {\widetilde a-a+1});\\
(2)\;&x_\nu c_{q,\tilde{a}-a+1,\tilde{a}}=(c_{q,\tilde{a}-a+1,\tilde{a}})'x_\nu, \qquad  x_\nu c_{q,\tilde{a}+1,\tilde{a}+b}=(c_{q,\tilde{a}+1,\tilde{a}+b})'x_\nu;\\
(3)\;&x_{\nu^-}c_{q,\tilde{a}-a+1,\tilde{a}-1}=(c_{q,\tilde{a}-a+1,\tilde{a}-1})'x_{\nu^-},\qquad  x_{\nu^-}c_{q,\tilde{a},\tilde{a}+b}=(c_{q,\tilde{a},\tilde{a}+b})'x_{\nu^-};\\
(4)\;&x_{\delta^-}c_{q,\tilde{a}-a+1,\tilde{a}-1}=(c_{q,\tilde{a}-a+1,\tilde{a}-1})'x_{\delta^-},\quad  x_{\delta^-}c_{\tilde{a}}=c_{\tilde{a}}x_{\delta^-},\quad
x_{\delta^-}c_{q,\tilde{a}+1,\tilde{a}+b}=(c_{q,\tilde{a}+1,\tilde{a}+b})'x_{\delta^-}.
\endaligned
\end{equation}

Because it can not be able to commute $x_{\nu^-}$ with $(c_{q,\tilde{a}-a+1,\tilde{a}})'$ and $c_{q,\tilde{a}+1,\tilde{a}+b}$ as above formulas, in the following proof, we usually rewrite
\begin{equation}\label{diampf-5}
(c_{q,\tilde{a}-a+1,\tilde{a}})'=(c_{q,\tilde{a}-a+1,\tilde{a}-1})'+q^{a-1}c_{\tilde{a}}, \quad c_{q,\tilde{a}+1,\tilde{a}+b}=c_{q,\tilde{a},\tilde{a}+b}-q^bc_{\tilde{a}}.
\end{equation}
Since $s_{\tilde{a}},s_{\tilde{a}+1},\ldots,s_{\tilde{a}+b-1}\in \fS_{\nu^-}$,
\begin{equation}\label{diampf-3}
x_{\nu^-}\TAIJ({\widetilde a},  {\widetilde a+b-1})=\STEP{b+1}x_{\nu^-},
\end{equation} and from Lemma \ref{xcT},
\begin{equation}\label{diampf-4}
x_{\nu^-}c_{\tilde{a}}\TAIJ({\widetilde a},  {\widetilde a+b-1})=x_{\nu^-}c_{q,\tilde{a},\tilde{a}+b}.
\end{equation}
By \eqref{eq_ta+} and \eqref{diampf-2}(1),
$$x_{{{\lambda}_{(h)}^{-}}} {\TDAM}= x_{{{\lambda}_{(h)}^{-}}\backslash{}'\!\nu^-}  {\TDAM} x_{\D_{-}} x_{\delta^-}.$$
Thus, $(\diamondsuit)_k$ can be rewritten as
\begin{align*}
(\diamondsuit)_k
&=x_{{{\lambda}_{(h)}^{-}}\backslash{}'\!\nu^-}  \TDAM x_{\D_{-}} x_{\delta^-}c_{\tilde{a}} \TDIJ({\tilde{a}-1},  {\tilde{a}-a +1}) c_{\Ad}
	\Sigma_A.
\end{align*}

Moreover, by \eqref{diampf-2}(2)\&(4), $x_{\delta^-}c_{\tilde{a}} =c_{\tilde{a}}x_{\delta^-}$ and $x_\nu c_{\Ad}=c'_{\Ad}x_\nu$. Combining $x_{\nu}=x_{\delta^-} \TDIJ({\widetilde a-1},  {\widetilde a-a+1})$ and
$x_\nu\Sigma_A=x_\mu=x_{\nu^-}\Sigma_{A^-_{h,k}}$ gives
\begin{equation}\label{diamon-k}
(\diamondsuit)_k=x_{{{\lambda}_{(h)}^{-}}\backslash{}'\!\nu^-}  \TDAM x_{\D_{-}}c_{\tilde{a}} x_{\delta^-}\TDIJ({\tilde{a}-1},  {\tilde{a}-a +1}) c_{\Ad}
	\Sigma_A=x_{{{\lambda}_{(h)}^{-}}\backslash{}'\!\nu^-}\TDAM  \underbrace{ x_{\D_{-}}c_{\tilde{a}}c'_{\Ad}x_{\nu^-}}_{=\ufm^-}\Sigma_{A^-_{h,k}}.
\end{equation}
 Since $c'_{\Ad}=({\cbefore})' ({\cmiddle})' ({\cafter})'$ (see \eqref{cA-decomp}), moving $({\cmiddle})'$ to the left yields $$c'_{\Ad}=(-1)^{(a^{\bar 1}_{h,k}+a^{\bar 1}_{h+1,k})\tilde{a}^{\bar1}_{h-1,k}}({\cmiddle})'({\cbefore})'({\cafter})'.$$
On the other hand, $x_{\nu^-}=x_{\delta^-} \TAIJ({\widetilde a},  {\widetilde a+b-1})$, $({\cbefore})'({\cafter})'x_{\delta^-} =x_{\delta^-}({\cbefore})({\cafter})$, and $\TAIJ({\widetilde a},  {\widetilde a+b-1})$ commutes with ${\cbefore}{\cafter}$. The middle part $\ufm^-$ becomes
\begin{equation*}\label{mathcal-M}
\ufm^-= (-1)^{(a^{\bar 1}_{h,k}+a^{\bar 1}_{h+1,k})\tilde{a}^{\bar1}_{h-1,k}} x_{\D_{-}}c_{\tilde{a}} ({\cmiddle})'x_{\delta^-}\TAIJ({\widetilde a},  {\widetilde a+b-1}){\cbefore}{\cafter}.
\end{equation*}

{\it The following discussion is parallel to the cases in the proof of Theorem \ref{phiupper1} according to the four values of $(\SOE{a}_{h,k},\SOE{a}_{h+1,k})$ in
${\cmiddle}= (c_{{q}, {\widetilde{a}}_{h-1,k}+1, {\widetilde{a}}_{h,k}})^{\SOE{a}_{h,k}}
		\cdot (c_{{q}, {\widetilde{a}}_{h,k}+1, {\widetilde{a}}_{h+1,k}})^{\SOE{a}_{h+1,k}}$ to determine $\ufm^-$ and then $(\diamondsuit)_k$.}

{\bf Case 1 ({$\SOE{a}_{h,k} = 0$}, {$\SOE{a}_{h+1,k} = 0$}).}
 In this case,  {${\cmiddle} = 1$} and so,
\begin{equation*}
\begin{aligned}
 \ufm^-&=  x_{\D_{-}}c_{\tilde{a}} x_{\delta^-}\TAIJ({\widetilde a},  {\widetilde a+b-1}){\cbefore}{\cafter}
 =x_{\D_{-}}x_{\delta^-}c_{\tilde{a}}\TAIJ({\widetilde a},  {\widetilde a+b-1}){\cbefore}{\cafter}
\end{aligned}
\end{equation*}
 Applying \eqref{diampf-2}(2) and \eqref{diampf-4}, we have
 $x_{\D_{-}}x_{\delta^-}c_{\tilde{a}}\TAIJ({\widetilde a},  {\widetilde a+b-1})=x_{\nu^-}c_{{q}, \tilde{a}, \tilde{a}+b}.$
So,
\begin{equation*}
\begin{aligned}
 \ufm^-=x_{\nu^-}c_{{q}, \tilde{a}, \tilde{a}+b}{\cbefore}{\cafter}&=(-1)^{\tilde{a}^{\bar1}_{h-1,k}} x_{\nu^-}{\cbefore}c_{{q}, \tilde{a}, \tilde{a}+b}{\cafter}\\
 &=(-1)^{\tilde{a}^{\bar1}_{h-1,k}} x_{\nu^-}c_{(A^{\bar{0}}-E_{h,k}|A^{\bar{1}}+E_{h+1,k})}
\end{aligned}
\end{equation*}

Hence, \eqref{diamon-k} becomes
\begin{equation}\label{case1}
\begin{aligned}
(\diamondsuit)_k&=x_{{{\lambda}_{(h)}^{-}}\backslash{}'\!\nu^-}\TDAM ((-1)^{\tilde{a}^{\bar1}_{h-1,k}} x_{\nu^-}c_{(A^{\bar{0}}-E_{h,k}|A^{\bar{1}}+E_{h+1,k})})\Sigma_{A^-_{h,k}}\\
&=(-1)^{\tilde{a}^{\bar1}_{h-1,k}} x_{{\lambda}_{(h)}^{-}}\TDAM c_{(A^{\bar{0}}-E_{h,k}|A^{\bar{1}}+E_{h+1,k})}\Sigma_{A^-_{h,k}}=(-1)^{\tilde{a}^{\bar1}_{h-1,k}}T_{(A^{\bar{0}}-E_{h,k}|A^{\bar{1}}+E_{h+1,k})}.
\end{aligned}
\end{equation}

{\bf Case 2 ({$\SOE{a}_{h,k} = 0$}, {$\SOE{a}_{h+1,k} = 1$}).}
Here {${\cmiddle} = c_{{q}, \tilde{a} +1, \tilde{a} +b}$}. By \eqref{diampf-2}(4), $$ c_{\tilde{a}}(c_{{q}, \tilde{a} +1, \tilde{a} +b})' x_{\delta^-}=x_{\delta^-}c_{\tilde{a}}  (c_{{q}, \tilde{a} +1, \tilde{a} +b})$$ Thus,
\begin{align*}
\ufm^- &={(-1)}^{\YK }
	x_{{\D_{-}}} c_{\tilde{a}}  (c_{{q}, \tilde{a} +1, \tilde{a} +b})' x_{\delta^-} \TAIJ( {\tilde{a}}, {\tilde{a}+b-1}) {\cbefore} {\cafter}\\
&={(-1)}^{\YK }x_{\nu^-}
	c_{\tilde{a}}  (c_{{q}, \tilde{a} +1, \tilde{a} +b}) \TAIJ( {\tilde{a}}, {\tilde{a}+b-1}) {\cbefore} {\cafter}.
\end{align*}
Since $c_{\tilde{a}} c_{{q}, \tilde{a} +1, \tilde{a} +b}=-c_{{q}, \tilde{a} +1, \tilde{a} +b}c_{\tilde{a}}$, \eqref{diampf-5} and  \eqref{diampf-2}(3) imply
\begin{equation*}
\begin{aligned}
&x_{\nu^-} c_{\tilde{a}} c_{{q}, \tilde{a} +1, \tilde{a} +b} \TAIJ( {\tilde{a}}, {\tilde{a}+b-1})=-x_{\nu^-} c_{{q}, \tilde{a} +1, \tilde{a} +b} c_{\tilde{a}} \TAIJ( {\tilde{a}}, {\tilde{a}+b-1})\\
 &=  - x_{\nu^-}  (c_{{q}, \tilde{a}, \tilde{a} +b} -  {q}^b c_{\tilde{a}} ) c_{\tilde{a}} \TAIJ( {\tilde{a}}, {\tilde{a}+b-1})=  -x_{\nu^-} c_{{q}, \tilde{a}, \tilde{a} +b} c_{\tilde{a}} \TAIJ( {\tilde{a}}, {\tilde{a}+b-1}) + x_{\nu^-}  {q}^b c_{\tilde{a}} c_{\tilde{a}} \TAIJ( {\tilde{a}}, {\tilde{a}+b-1}).
\end{aligned}
\end{equation*}
By \eqref{diampf-2}(3)  again and \eqref{diampf-4},
$$\begin{aligned}
-x_{\nu^-} c_{{q}, \tilde{a}, \tilde{a} +b} c_{\tilde{a}} \TAIJ( {\tilde{a}}, {\tilde{a}+b-1}) &= -c'_{{q}, \tilde{a}, \tilde{a} +b} x_{\nu^-}c_{\tilde{a}} \TAIJ( {\tilde{a}}, {\tilde{a}+b-1})\\
&=-c'_{{q}, \tilde{a}, \tilde{a} +b}  x_{\nu^-}c_{q,\tilde{a},\tilde{a}+b}=-x_{\nu^-}c^2_{q,\tilde{a},\tilde{a}+b}={\STEPP{b+1}}x_{\nu^-}
\end{aligned}$$
and, by applying \eqref{diampf-3},
$x_{\nu^-}  {q}^b c_{\tilde{a}} c_{\tilde{a}} \TAIJ( {\tilde{a}}, {\tilde{a}+b-1})=-{q}^b\STEP{b+1}x_{\nu^-}.$
Hence,
\begin{align*}
\ufm^-&={(-1)}^{\YK }x_{\nu^-}({\STEPP{b+1}}-{q}^b\STEP{b+1}){\cbefore} {\cafter}\\
&={(-1)}^{\YK +1}q^{-1}{\STEPPDR{b+1}}x_{\nu^-}c_{(A^{\bar 0}-E_{h,k}+2E_{h+1,k}|A^{\bar 1}-E_{h+1,k})},
\end{align*}
and
\begin{equation}\label{case2}
\begin{aligned}
(\diamondsuit)_k&={(-1)}^{\YK +1}q^{-1}{\STEPPDR{b+1}}x_{\nu^-}c_{(A^{\bar 0}-E_{h,k}+2E_{h+1,k}|A^{\bar 1}-E_{h+1,k})}\Sigma_{A^-_{h,k}}\\
&={(-1)}^{\YK+1}q^{-1}{\STEPPDR{b+1}}T_{(A^{\bar 0}-E_{h,k}+2E_{h+1,k}|A^{\bar 1}-E_{h+1,k})}.
\end{aligned}
\end{equation}

{\bf Case 3 ({$\SOE{a}_{h,k} = 1$}, {$\SOE{a}_{h+1,k} = 0$)}.}
We have  {${\cmiddle} = c_{{q}, \tilde{a}-a+1, \tilde{a}}$} and
\begin{equation*}
\ufm^-= (-1)^{\tilde{a}^{\bar1}_{h-1,k}} x_{\D_{-}}c_{\tilde{a}} ( c_{{q}, \tilde{a}-a+1, \tilde{a}})'x_{\delta^-}\TAIJ({\widetilde a},  {\widetilde a+b-1}){\cbefore}{\cafter}.
\end{equation*}
Applying \eqref{diampf-2}(4),\eqref{diampf-5},\eqref{diampf-3} and \eqref{diampf-4} yields
\begin{equation*}
\begin{aligned}
&x_{\D_{-}}c_{\tilde{a}} ( c_{{q}, \tilde{a}-a+1, \tilde{a}})'x_{\delta^-}\TAIJ({\widetilde a},  {\widetilde a+b-1})\\
&=x_{\D_{-}}c_{\tilde{a}}((c_{{q}, \tilde{a}-a+1, \tilde{a} -1})' + {q}^{a-1} c_{\tilde{a}} )x_{\delta^-}\TAIJ({\widetilde a},  {\widetilde a+b-1})\\
&=x_{\D_{-}}c_{\tilde{a}}(c_{{q}, \tilde{a}-a+1, \tilde{a} -1})' x_{\delta^-}\TAIJ({\widetilde a},  {\widetilde a+b-1})+
{q}^{a-1}x_{\D_{-}}c_{\tilde{a}}c_{\tilde{a}}x_{\delta^-}\TAIJ({\widetilde a},  {\widetilde a+b-1})\\
&=-x_{\D_{-}}(c_{{q}, \tilde{a}-a+1, \tilde{a} -1})'x_{\delta^-}c_{\tilde{a}}\TAIJ({\widetilde a},  {\widetilde a+b-1})-{q}^{a-1}x_{\D_{-}}x_{\delta^-}\TAIJ({\widetilde a},  {\widetilde a+b-1})\\
&=-x_{\nu^-}(c_{{q}, \tilde{a}-a+1, \tilde{a} -1})(c_{{q}, \tilde{a}, \tilde{a}+b})-{q}^{a-1}{\STEP{b+1}}x_{\nu^-},
\end{aligned}
\end{equation*}
and
\begin{equation*}
\begin{aligned}
\ufm^-&= (-1)^{\tilde{a}^{\bar1}_{h-1,k}}(-x_{\nu^-}(c_{{q}, \tilde{a}-a+1, \tilde{a} -1})(c_{{q}, \tilde{a}, \tilde{a}+b})-{q}^{a-1}{\STEP{b+1}}x_{\nu^-}){\cbefore}{\cafter}\\
&=(-1)^{\tilde{a}^{\bar1}_{h-1,k}+1}x_{\nu^-}{\cbefore}(c_{{q}, \tilde{a}-a+1, \tilde{a} -1})(c_{{q}, \tilde{a}, \tilde{a}+b}){\cafter}
+(-1)^{\tilde{a}^{\bar1}_{h-1,k}+1}{q}^{a-1}{\STEP{b+1}}x_{\nu^-}{\cbefore}{\cafter}\\
&=(-1)^{\tilde{a}^{\bar1}_{h-1,k}+1}(x_{\nu^-}c_{(A^{\bar0}-E_{h,k}|A^{\bar1}+E_{h+1,k})}+{q}^{a-1}{\STEP{b+1}}x_{\nu^-}c_{(A^{\bar0}+E_{h+1,k}|A^{\bar1}-E_{h,k})}).
\end{aligned}
\end{equation*}
Consequently, \eqref{diamon-k} becomes
\begin{equation}\label{case3}
\begin{aligned}
(\diamondsuit)_k&=x_{{{\lambda}_{(h)}^{-}}\backslash{}'\!\nu^-}\TDAM \{(-1)^{\tilde{a}^{\bar1}_{h-1,k}+1}(x_{\nu^-}c_{(A^{\bar0}-E_{h,k}|A^{\bar1}+E_{h+1,k})}\\
&\quad+{q}^{a-1}{\STEP{b+1}}x_{\nu^-}c_{(A^{\bar0}+E_{h+1,k}|A^{\bar1}-E_{h,k})})\}\Sigma_{A^-_{h,k}}\\
&=(-1)^{\tilde{a}^{\bar1}_{h-1,k}+1}(T_{(A^{\bar0}-E_{h,k}|A^{\bar1}+E_{h+1,k})}+{q}^{a-1}{\STEP{b+1}}T_{(A^{\bar0}+E_{h+1,k}|A^{\bar1}-E_{h,k})}).
\end{aligned}
\end{equation}

{\bf Case 4 ({$\SOE{a}_{h,k} = 1$}, {$\SOE{a}_{h+1,k} = 1$}).}
Here we have {${\cmiddle} = c_{{q}, \tilde{a}-a+1, \tilde{a}} c_{{q}, \tilde{a} +1, \tilde{a} +b}$} and
\begin{equation*}
\ufm^-=  x_{\D_{-}}\underbrace{c_{\tilde{a}} (c_{{q}, \tilde{a}-a+1, \tilde{a}} c_{{q}, \tilde{a} +1, \tilde{a} +b})'x_{\delta^-}}_{=\mathcal{N}}\TAIJ({\widetilde a},  {\widetilde a+b-1}){\cbefore}{\cafter}.
\end{equation*}
Since
$$
\begin{aligned}
&c_{\tilde{a}}c'_{{q}, \tilde{a}-a+1, \tilde{a}}=c_{\tilde{a}}(c'_{{q}, \tilde{a}-a +1, \tilde{a}-1}+q^{a-1}c_{\tilde{a} })=-c'_{{q}, \tilde{a}-a +1, \tilde{a}-1}c_{\tilde{a}}-q^{a-1},\\
&c_{\tilde{a}}c_{{q}, \tilde{a}, \tilde{a} +b}=c_{\tilde{a}}(c_{{q}, \tilde{a}+1, \tilde{a} +b}+q^bc_{\tilde{a}})=-c_{{q}, \tilde{a}+1, \tilde{a} +b}c_{\tilde{a}}-q^b=-c_{{q}, \tilde{a}, \tilde{a} +b}c_{\tilde{a}}-2q^b,
\end{aligned}
$$
by \eqref{diampf-2}(4) and \eqref{diampf-5},
we have
\begin{equation*}
\begin{aligned}
\mathcal{N}&=c_{\tilde{a}} (c_{{q}, \tilde{a}-a+1, \tilde{a}})'( c_{{q}, \tilde{a} +1, \tilde{a} +b})'x_{\delta^-}\\
&=(-c'_{{q}, \tilde{a}-a +1, \tilde{a}-1}c_{\tilde{a}}-q^{a-1})x_{\delta^-}( c_{{q}, \tilde{a} +1, \tilde{a} +b})\\
&=x_{\delta^-}(-c_{{q}, \tilde{a}-a +1, \tilde{a}-1}c_{\tilde{a}}-q^{a-1})( c_{{q}, \tilde{a} +1, \tilde{a} +b}) \\
&=-x_{\delta^-}c_{{q}, \tilde{a}-a +1, \tilde{a}-1}c_{\tilde{a}}( c_{{q}, \tilde{a} +1, \tilde{a} +b})-q^{a-1} x_{\delta^-}( c_{{q}, \tilde{a} +1, \tilde{a} +b}).
\end{aligned}
\end{equation*}
Thus,
\begin{equation}\label{M}
\begin{aligned}
\ufm^-&= x_{\D_{-}}{\mathcal{N}}\TAIJ({\widetilde a},  {\widetilde a+b-1}){\cbefore}{\cafter}\\
&= x_{\D_{-}}\left\{-x_{\delta^-}c_{{q}, \tilde{a}-a +1, \tilde{a}-1}c_{\tilde{a}}( c_{{q}, \tilde{a} +1, \tilde{a} +b})-q^{a-1} x_{\delta^-}( c_{{q}, \tilde{a} +1, \tilde{a} +b})\right\}\TAIJ({\widetilde a},  {\widetilde a+b-1}){\cbefore}{\cafter}\\
&=\left\{-x_{\nu^-}c_{{q}, \tilde{a}-a +1, \tilde{a}-1}c_{\tilde{a}}( c_{{q}, \tilde{a} +1, \tilde{a} +b})\right\}\TAIJ({\widetilde a}, {\widetilde a+b-1}){\cbefore}{\cafter}\\
&\quad+\left\{-q^{a-1} x_{\nu^-}( c_{{q}, \tilde{a} +1, \tilde{a} +b})\right\}\TAIJ({\widetilde a},  {\widetilde a+b-1}){\cbefore}{\cafter}\\
&=\mathcal{M}_1+\mathcal{M}_2,
\end{aligned}
\end{equation}
where
 $$\begin{aligned}\mathcal{M}_1&= \left\{-x_{\nu^-}c_{{q}, \tilde{a}-a +1, \tilde{a}-1}c_{\tilde{a}}( c_{{q}, \tilde{a} +1, \tilde{a} +b})\right\}\TAIJ({\widetilde a},  {\widetilde a+b-1}){\cbefore}{\cafter},\\
 \mathcal{M}_2&= \left\{-q^{a-1} x_{\nu^-}( c_{{q}, \tilde{a} +1, \tilde{a} +b})\right\}\TAIJ({\widetilde a},  {\widetilde a+b-1}){\cbefore}{\cafter}.
 \end{aligned}$$
We now modify $\mathcal{M}_1$ and $\mathcal{M}_2$ into a desired form. Since
$$
\begin{aligned}
&x_{\nu^-}(c_{{q}, \tilde{a}-a +1, \tilde{a}-1})c_{\tilde{a}}( c_{{q}, \tilde{a}+1, \tilde{a} +b})\TAIJ({\widetilde a},  {\widetilde a+b-1})\\
&=x_{\nu^-}(c_{{q}, \tilde{a}-a +1, \tilde{a}-1})c_{\tilde{a}}( c_{{q}, \tilde{a}, \tilde{a} +b}-q^bc_{\tilde{a}})\TAIJ({\widetilde a},  {\widetilde a+b-1})\\
&=x_{\nu^-}(c_{{q}, \tilde{a}-a +1, \tilde{a}-1})(-c_{{q}, \tilde{a}, \tilde{a} +b}c_{\tilde{a}}+q^bc^2_{\tilde{a}})\TAIJ({\widetilde a},  {\widetilde a+b-1})\\
&=(-x_{\nu^-}(c_{{q}, \tilde{a}-a +1, \tilde{a}-1})(c_{{q}, \tilde{a}, \tilde{a} +b})c_{\tilde{a}}\TAIJ({\widetilde a},  {\widetilde a+b-1}))+
(-q^bx_{\nu^-}(c_{{q}, \tilde{a}-a +1, \tilde{a}-1})\TAIJ({\widetilde a},  {\widetilde a+b-1})),\\
\end{aligned}
$$
applying \eqref{diampf-2}(3) and \eqref{diampf-3} yields
$$
\begin{aligned}
&-x_{\nu^-}(c_{{q}, \tilde{a}-a +1, \tilde{a}-1})(c_{{q}, \tilde{a}, \tilde{a} +b})c_{\tilde{a}}\TAIJ({\widetilde a},  {\widetilde a+b-1})\\
&=-(c_{{q}, \tilde{a}-a +1, \tilde{a}-1})'(c_{{q}, \tilde{a}, \tilde{a} +b})'x_{\nu^-}c_{\tilde{a}}\TAIJ({\widetilde a},  {\widetilde a+b-1})=-(c_{{q}, \tilde{a}-a +1, \tilde{a}-1})'(c_{{q}, \tilde{a}, \tilde{a} +b})'x_{\nu^-}(c_{{q}, \tilde{a}, \tilde{a} +b})\\
&=-x_{\nu^-}(c_{{q}, \tilde{a}-a +1, \tilde{a}-1})(c_{{q}, \tilde{a}, \tilde{a} +b})^2
=\STEPP{b+1}x_{\nu^-}(c_{{q}, \tilde{a}-a +1, \tilde{a}-1}),
\end{aligned}
$$
and, by \eqref{diampf-4},
$$
-q^bx_{\nu^-}(c_{{q}, \tilde{a}-a +1, \tilde{a}-1})\TAIJ({\widetilde a},  {\widetilde a+b-1})=-q^b\STEP{b+1}x_{\nu^-}(c_{{q}, \tilde{a}-a +1, \tilde{a}-1}).
$$

Because of  $\STEPP{b+1}-q^b\STEP{b+1}=-q^{-1}{\STEPPDR{b+1}}$,
\begin{equation}\label{M_1}
\begin{aligned}
\mathcal{M}_1&= (-x_{\nu^-}c_{{q}, \tilde{a}-a +1, \tilde{a}-1}c_{\tilde{a}}( c_{{q}, \tilde{a} +1, \tilde{a} +b}))\TAIJ({\widetilde a},  {\widetilde a+b-1}){\cbefore}{\cafter}\\
&=q^{-1}{\STEPPDR{b+1}} x_{\nu^-}(c_{{q}, \tilde{a}-a +1, \tilde{a}-1}){\cbefore}{\cafter}\\
&=(-1)^{\tilde{a}^{\bar1}_{h-1,k}}q^{-1}{\STEPPDR{b+1}}x_{\nu^-}{\cbefore}(c_{{q}, \tilde{a}-a +1, \tilde{a}-1}){\cafter}\\
&=(-1)^{\tilde{a}^{\bar1}_{h-1,k}}q^{-1}\STEPPDR{b+1}x_{\nu^-}c_{(A^{\bar0}-E_{h,k}+2E_{h+1,k}|A^{\bar1}-E_{h+1,k})}.
\end{aligned}
\end{equation}

Now, we calculate $\mathcal{M}_2$. Since
$$
\begin{aligned}
-&q^{a-1} x_{\nu^-}( c_{{q}, \tilde{a} +1, \tilde{a} +b})\TAIJ({\widetilde a},  {\widetilde a+b-1}){\cbefore}{\cafter}\\
&=-q^{a-1} x_{\nu^-}( c_{{q}, \tilde{a}, \tilde{a} +b}-q^{b}c_{\tilde{a}})\TAIJ({\widetilde a},  {\widetilde a+b-1}){\cbefore}{\cafter}\;\;(\mbox{by \eqref{diampf-5}})\\
&=-q^{a-1}x_{\nu^-} c_{{q}, \tilde{a}, \tilde{a} +b}\TAIJ({\widetilde a},  {\widetilde a+b-1}){\cbefore}{\cafter}+
q^{a+b-1} x_{\nu^-}c_{\widetilde a}\TAIJ({\widetilde a},  {\widetilde a+b-1}){\cbefore}{\cafter}\\
&=\left(-q^{a-1}\STEP{b+1}x_{\nu^-} c_{{q}, \tilde{a}, \tilde{a} +b}+q^{a+b-1} x_{\nu^-} c_{{q}, \tilde{a}, \tilde{a} +b}\right){\cbefore}{\cafter}\;\;(\mbox{by \eqref{diampf-3} and \eqref{diampf-4}})\\
&=(-q^{a-1}\STEP{b+1}+q^{a+b-1})x_{\nu^-}c_{{q}, \tilde{a}, \tilde{a} +b}{\cbefore}{\cafter},\\
\end{aligned}
$$
and  $(-q^{a-1}\STEP{b+1}+q^{a+b-1})=-q^{a-1}\STEP{b}$, we have
\begin{equation}\label{M_2}
\begin{aligned}\mathcal{M}_2&=-q^{a-1}\STEP{b}x_{\nu^-}( c_{{q}, \tilde{a}, \tilde{a} +b}){\cbefore}{\cafter}\\
&=(-1)^{\tilde{a}^{\bar1}_{h-1,k}+1}q^{a-1}\STEP{b}x_{\nu^-}{\cbefore}( c_{{q}, \tilde{a}, \tilde{a} +b}){\cafter}\\
&=(-1)^{\tilde{a}^{\bar1}_{h-1,k}+1}q^{a-1}\STEP{b}x_{\nu^-}c_{(A^{\bar0}+E_{h+1,k}|A^{\bar1}-E_{h,k})}.
\end{aligned}
\end{equation}
Now, combining the formulas \eqref{M},\eqref{M_1} and \eqref{M_2}, we obtain in this case
\begin{equation}\label{case4}
\begin{aligned}
(\diamondsuit)_k&=x_{{{\lambda}_{(h)}^{-}}\backslash{}'\!\nu^-}\TDAM (\ufm^-)\Sigma_{A^-_{h,k}}\\
&=(-1)^{\tilde{a}^{\bar1}_{h-1,k}+a^{\bar1}_{h,k}+1}q^{-1}\STEPPDR{b+1}T_{(A^{\bar0}-E_{h,k}+2E_{h+1,k}|A^{\bar1}-E_{h+1,k})}\\
&\quad+(-1)^{\tilde{a}^{\bar1}_{h-1,k}+1}q^{a-1}\STEP{b} T_{(A^{\bar0}+E_{h+1,k}|A^{\bar1}-E_{h,k})}.
\end{aligned}
\end{equation}

Finally, let
\begin{equation*}\label{FSDP}
\begin{aligned}
(\diamondsuit)_k'%=
	%\sum_{j=\BK(h,k)}^{\BK(h,k-1) - 1}x_{{{\lambda}_{(h)}^{-}}} \left(\TDAM c_{\tilde{a}_{h,k}}\right)T_{\tilde{a}_{h,k}-1} \cdots T_{\tilde{a}_{h,k}-q_j}
	%c_{\Ad}
	%\Sigma_A\\
&= %\Big\{
	{(-1)}^{{\SOE{\widetilde{a}}}_{h-1,k} + \SOE{a}_{h,k}} 	T_{(\SE{A} - E_{h,k}| \SO{A}+ E_{h+1, k})} \\
	& \qquad +
	{(-1)}^{{\SOE{\widetilde{a}}}_{h-1,k} + \SOE{a}_{h,k}+1}
	{q}^{-1}  \STEPPDR{a_{h+1, k} +1}
	T_{(\SE{A} - E_{h,k} + 2E_{h+1, k} | \SO{A}  -E_{h+1, k})} \\
	& \qquad +
	{(-1)}^{{\SOE{\widetilde{a}}}_{h-1,k}+1} {q}^{ a_{h,k}-1}
	\STEP{ \SEE{a}_{h+1,k}+1}
	T_{(\SE{A} + E_{h+1,k} | \SO{A}  - E_{h,k} )}.
%\Big\}.
\end{aligned}
\end{equation*}
A case by case discussion, using \eqref{case1}, \eqref{case2}, \eqref{case3}, \eqref{case4}, and Convention \ref{CONV}, proves $(\diamondsuit)_k'=(\diamondsuit)_k$.
\end{proof}

\end{document}